\documentclass{article}
\usepackage{amsmath, amsthm, amssymb,float}
\usepackage{amsfonts}
\usepackage[linesnumbered,ruled]{algorithm2e}

\usepackage{graphicx}
\usepackage{caption}
\usepackage{subcaption}
\usepackage{centernot}

\usepackage{textcomp}
\usepackage{cases}
\usepackage{braket}
\usepackage[normalem]{ulem}

\usepackage{chngcntr}

\usepackage{xspace}

\usepackage{pifont}

\usepackage[margin=2.5cm]{geometry}

\usepackage[colorlinks=true, urlcolor=blue, anchorcolor=blue, citecolor=blue,filecolor=blue,linkcolor=blue,menucolor=blue,pagecolor=blue]{hyperref}

\usepackage[capitalise]{cleveref}
\crefformat{equation}{(#2#1#3)}

\theoremstyle{theorem}
\newtheorem{theorem}{Theorem}
\numberwithin{theorem}{section}

\theoremstyle{lemma}
\newtheorem{lemma}[theorem]{Lemma}

\theoremstyle{remark}
\newtheorem{remark}[theorem]{Remark}

\theoremstyle{corollary}
\newtheorem{corollary}[theorem]{Corollary}

\theoremstyle{proposition}
\newtheorem{proposition}[theorem]{Proposition}

\theoremstyle{definition}
\newtheorem{definition}[theorem]{Definition}

\theoremstyle{claim}
\newtheorem{claim}[theorem]{Claim}

\DeclareMathOperator{\Tr}{Tr}
\newcommand{\GNMR}{{\texttt{GNMR}}\xspace}

\DeclareMathOperator*{\argmin}{arg\,min}
\newenvironment{psmallmatrix}
  {\bigl(\begin{smallmatrix}}
  {\end{smallmatrix}\bigr)}

\begin{document}
	
\title{{GNMR}: A provable one-line algorithm for low rank matrix recovery}	

\author{Pini Zilber\footnotemark[1]\thanks{Faculty of Mathematics and Computer Science, Weizmann Institute of Science \newline (\href{mailto:pini.zilber@weizmann.ac.il}{pini.zilber@weizmann.ac.il}, \href{mailto:boaz.nadler@weizmann.ac.il}{boaz.nadler@weizmann.ac.il})}
\and Boaz Nadler\footnotemark[1]}
\date{}

\maketitle
\nopagebreak

\begin{abstract}
Low rank matrix recovery problems
appear in a broad range of applications.
In this work we present \texttt{GNMR} 
--- an extremely simple iterative algorithm for low rank matrix recovery,  based on a Gauss-Newton linearization. 
On the theoretical front, we derive recovery guarantees for \GNMR in both matrix sensing and matrix completion settings.
Some of these results improve upon the best currently known for other methods.
A key property of \GNMR is that it implicitly keeps the factor matrices approximately balanced
throughout its iterations. 
On the empirical front, we show that for matrix completion with uniform sampling, \GNMR performs better than several popular methods, especially when given very few observations close to the information limit.
\end{abstract}

\section{Introduction}
Low rank matrices play a fundamental role in a broad range of applications in multiple scientific fields.
In many cases the matrix is not fully observed, and yet it is often possible to recover it due to its assumed low rank structure.
In this paper we propose a novel method, denoted \GNMR, to tackle this class of problems. \GNMR (Gauss-Newton Matrix Recovery) is a very simple iterative method with state-of-the-art performance, for which we also derive strong theoretical recovery guarantees.
As detailed below, some of our guarantees improve upon the best currently available for other methods.

Concretely, consider the problem of recovering a matrix $X^* \in \mathbb R^{n_1\times n_2}$ of known rank $r$ from a set of $m$ linear measurements $b \equiv \mathcal A(X^*) + \xi$ where $\mathcal A: \mathbb R^{n_1\times n_2}\to \mathbb R^m$ is a sensing operator and $\xi\in \mathbb R^m$ is additive error.
Formally, the goal is to solve the optimization problem
\begin{align}\label{eq:matrixRecovery_problem}
\min_X \, f(X) \quad \text{s.t. rank(}X) \leq r
\end{align}
where
\begin{align}\label{eq:matrixRecovery_objective}
f(X) = \|\mathcal A(X) - b\|_2^2 .
\end{align}
Two common cases of \cref{eq:matrixRecovery_problem} are matrix sensing and matrix completion.
These and related problems appear in a wide variety of applications, including collaborative filtering, manifold learning, quantum computing, image processing and computer vision, see \cite{buchanan2005damped,candes2010matrix,davenport2016overview,chi2019matrix,chi2019nonconvex} and references therein.
In the \textit{matrix sensing} problem, for well-posedness of \cref{eq:matrixRecovery_problem}
for any rank-$r$ matrix $X^*$, the sensing operator $\mathcal A$ is required to satisfy a suitable RIP (Restricted Isometry Property) \cite{candes2008restricted,recht2010guaranteed}.
In the \textit{matrix completion} problem, 
the operator $\mathcal A$ extracts $m$ entries of the underlying matrix, $X^*_{i,j}$ for $(i,j) \in \Omega$ where $\Omega \subseteq [n_1]\times [n_2]$ is of size $m$.
In this case, $\mathcal A$ does not satisfy an RIP. However, if $\Omega$ is sampled uniformly at random with large enough cardinality $|\Omega|$ and $X^*$ is incoherent, then with high probability $X^*$ is the unique solution of \cref{eq:matrixRecovery_problem}; see \cite{candes2010matrix,candes2010power,gross2011recovering,singer2010uniqueness,pimentel2016characterization}
for more details.

In general, matrix recovery problems of the form of \cref{eq:matrixRecovery_problem}
are NP-hard. Yet, due to their importance, various methods
to find approximate solutions were developed. Most of them can be assigned to one of two classes. The first consists of algorithms which optimize over the full $n_1\times n_2$ matrix. 
Some methods in this class replace the rank-$r$ constraint by a suitable matrix penalty that promotes a low rank solution.
One popular choice is the nuclear norm, which leads to a convex semi-definite program \cite{fazel2001rank}. Nuclear norm minimization enjoys strong theoretical guarantees \cite{candes2009exact,candes2010power,recht2011simpler}, but in general is computationally slow.
Hence, several works developed fast optimization methods, see \cite{rennie2005fast,ji2009accelerated,cai2010singular,mazumder2010spectral,toh2010accelerated,fornasier2011low,ma2011fixed,avron2012efficient} and references therein.
Another matrix penalty that promotes low rank solutions is the non-convex Schatten $p$-norm with $p<1$ \cite{marjanovic2012l_q,kummerle2018harmonic}.

The other class consists of methods that explicitly enforce the rank-$r$ constraint in \cref{eq:matrixRecovery_problem}. 
For example, hard thresholding methods keep at each iteration only the top $r$ principal components \cite{jain2010guaranteed,tanner2013normalized,blanchard2015cgiht,kyrillidis2014matrix}.
Other methods in this class employ the decomposition $X = UV^\top$ with $U\in\mathbb{R}^{n_1\times r}$, $V\in\mathbb{R}^{n_2\times r}$. The matrix recovery objective \cref{eq:matrixRecovery_problem} then reads
\begin{align} \label{eq:matrixRecovery_factorizedObjective}
\min_{U,V} f\left(UV^\top\right) .
\end{align}
As the factorized problem \cref{eq:matrixRecovery_factorizedObjective} involves only $(n_1+n_2)r$ variables, these methods are in general more scalable and can cope with larger matrices.
One approach to solve \cref{eq:matrixRecovery_factorizedObjective}
is alternating minimization \cite{haldar2009rank,keshavan2012efficient,wen2012solving,jain2013low}.  Another approach is gradient descent, either on the Euclidean manifold \cite{sun2016guaranteed,tu2016low} or on other Riemannian manifolds \cite{keshavan2010matrix,ngo2012scaled,vandereycken2013low,mishra2014r3mc,mishra2014fixed,boumal2015low}.
Some of the above works also derived recovery guarantees for these methods. For additional guarantees, see  \cite{hardt2014understanding,jain2015fast,yi2016fast,zheng2016convergence,wei2016guarantees,ma2019implicit, chen2020nonconvex,ma2021beyond,tong2021accelerating} and references therein.

Most matrix completion methods proposed thus far suffer from two limitations: they may fail to recover the underlying matrix $X^*$ if it is even mildly ill-conditioned, or if the number of observed entries $m$ is relatively small \cite{tanner2013normalized,bauch2021rank,kummerle2020escaping}.
This may pose a significant drawback in practical applications.
Two recent algorithms that are relatively scalable and perform well with ill-conditioning and few observations are \texttt{R2RILS} \cite{bauch2021rank} and 
\texttt{MatrixIRLS} \cite{kummerle2020escaping,kummerle2021scalable}.
However, only limited recovery guarantees are available for them.
This raises the following question: \textit{is there an algorithm that is both computationally efficient, 
succeeds on ill-conditioned matrices with few measurements, and enjoys strong theoretical recovery guarantees?}

In this work we make a step towards answering this question. 
In \cref{sec:GNMR_description} we present a novel iterative algorithm which both empirically outperforms existing methods (including \texttt{R2RILS} and \texttt{MatrixIRLS}) in the task of recovering ill-conditioned matrices from few observations, and for which we are able to derive strong recovery guarantees due to its remarkable simplicity.
Our proposed factorization-based algorithm, named \GNMR, is based on the classical Gauss-Newton method. At each iteration, \GNMR solves a simple least squares problem obtained by a linearization of the factorized objective \cref{eq:matrixRecovery_factorizedObjective}.
The resulting least squares problem can be solved efficiently by standard solvers.

On the theoretical front, in \cref{sec:theory} we present recovery guarantees for \GNMR in both 
the matrix sensing and matrix completion settings.
In matrix sensing, we prove that starting from a sufficiently accurate initial estimate, \GNMR recovers the underlying matrix with quadratic rate under the minimal RIP assumptions on the sensing operator $\mathcal A$, see \cref{thm:sens_quadConvergence}. 
To the best of our knowledge, this guarantee is among the sharpest currently available for any recovery algorithm.
Moreover, we prove that in matrix sensing, a slightly modified variant of \GNMR is stable against arbitrary additive error of bounded norm. Importantly, this type of error also captures the practical setting of approximately low rank, where $X^*$ has $r$ large singular values and the remaining ones are much smaller yet nonzero.
Next, in \cref{sec:theory_completion} we analyze \GNMR in the matrix completion setting. Here we follow the standard approach in the literature, whereby to derive recovery guarantees {for factorization-based methods}, suitable regularization terms are added to the respective algorithm, see for example \cite{keshavan2010matrix,sun2016guaranteed}. 
In \cref{thm:comp_linearConvergence} we prove that given a sufficiently accurate initialization, a regularized variant of \GNMR recovers the target matrix at a linear rate under the weakest known assumptions for non-convex optimization methods. In addition, in \cref{thm:comp_quadConvergence} we prove that 
near the global optimum, the convergence rate of \GNMR is quadratic.

Our proof technique builds upon recent works which derived guarantees for gradient descent algorithms \cite{keshavan2010matrix,tu2016low,sun2016guaranteed,zheng2016convergence,yi2016fast,ma2021beyond}.
However, as \GNMR is markedly different, deriving recovery guarantees for it required several non-trivial modifications.
In particular, while the iterates of gradient descent have a simple explicit formula, \GNMR solves a degenerate least squares problem. 
In our analysis, we exploit this degeneracy in our favor, and show that by choosing the minimal norm solution the iterates of \GNMR enjoy some desirable properties such as implicit balance regularization, see \cref{sec:implicit_balance} for more details.
In the course of our proofs, we extended and improved several technical results from previous works, including \cite[Lemma~5.14]{tu2016low}, \cite[Lemma~1]{ma2021beyond} and \cite[Claim~3.1]{sun2016guaranteed}. 
Specifically, in \cref{thm:RIP} we present a novel RIP-like guarantee for matrix completion which is in several aspects sharper than \cite[Claim~3.1]{sun2016guaranteed}, especially in terms of the required number of observations. 
These improvements may be of independent interest, e.g.~for proving recovery guarantees of other algorithms.

On the empirical front, in \cref{sec:numerics} we present several simulations with ill-conditioned matrices and few observed entries chosen uniformly at random. We show that \GNMR improves upon the state of the art in these settings, outperforming several popular algorithms. 
In particular, \GNMR is able to successfully recover matrices from very few observations close to
the information limit, where all other compared methods fail. 

\vspace{1em}
\noindent
{\bf Notation.}
The $i$'th largest singular value of a matrix $X$ is denoted by $\sigma_i(X)$. 
The condition number of a rank-$r$ matrix is denoted by $\kappa = \sigma_1/\sigma_r$.
Denote the Euclidean norm of a vector $x$ by $\|x\|$.
Denote the trace of a matrix $A$ by $\Tr(A)$, its operator norm (a.k.a.~spectral norm) by $\|A\|_2$, its Frobenius norm by $\|A\|_F$, its $i$'th row by $A^{(i)}$, and its largest row norm by $\|A\|_{2,\infty} \equiv \max_i \|A^{(i)}\|$.
The transpose of the inverse of $A$ is denoted by $A^{-\top} \equiv (A^{-1})^\top$.
In the matrix completion problem, the fraction of observed entries is denoted by $p = {|\Omega|} / {(n_1n_2)} = {m} / {(n_1 n_2)}$. The sampling operator $\mathcal P_\Omega$ extracts the entries of a matrix according to $\Omega$, such that $\mathcal P_\Omega(X)$ is a vector of size $m$ with entries $X_{ij}$ for $(i,j)\in \Omega$. Denote $\|A\|_{F(\Omega)}^2 = \|\mathcal P_\Omega(A)\|^2 = {\sum_{(i,j)\in \Omega} A_{ij}^2}$ (note this is not a norm).
Denote $n = \max\{n_1, n_2\}$.
When discussing sample or computational complexity, for simplicity we assume $n_1 \sim n_2$, namely the ratio $\min\{n_1, n_2\}/n$ is considered a constant.
Finally, unless stated otherwise, $C$, $c_e$ and $c_l$ denote absolute constants independent of the problem parameters such as $n, r, \kappa, \Omega$ etc..

\section{Description of \GNMR}\label{sec:GNMR_description}
Given an estimate $(U_0, V_0)$, factorization based methods seek an update $(\Delta U, \Delta V)$ such that $(U_1, V_1) = (U_0 + \Delta U, V_0 + \Delta V)$ minimizes \cref{eq:matrixRecovery_factorizedObjective}. The original problem \cref{eq:matrixRecovery_factorizedObjective} can be equivalently written in terms of the update $(\Delta U, \Delta V)$ as
\begin{align*}
\min_{\Delta U, \Delta V} \|\mathcal A\left(U_0V_0^\top + U_0 \Delta V^\top + \Delta U V_0^\top + \Delta U \Delta V^\top\right) - b\|^2.
\end{align*}
This problem is non-convex due to the second order term $\Delta U \Delta V^\top$.
The idea of \GNMR is to neglect this term, yielding the convex least squares scheme
\begin{subequations}
		\label{eq:updatingVariant}\begin{align}
\begin{pmatrix} \Delta U_0 \\ \Delta V_0 \end{pmatrix} &= \argmin_{\Delta U, \Delta V} \|\mathcal A\left(U_0V_0^\top + U_0 \Delta V^\top + \Delta U V_0^\top\right) - b\|^2, \label{eq:updatingVariant_LSQR}\\
\begin{pmatrix} U_1 \\ V_1 \end{pmatrix} &= \begin{pmatrix} U_0 + \Delta U_0 \\ V_0 + \Delta V_0 \end{pmatrix}.
\end{align}\end{subequations}
It is easy to see that the above is simply an instance of the Gauss-Newton method applied to matrix recovery.
This scheme, however, is not well defined since the least squares problem \cref{eq:updatingVariant_LSQR} is rank deficient, and thus has an infinite number of solutions. 
For example, if $(\Delta U, \Delta V)$ is a solution, so is $(\Delta U + U_0 R, \Delta V - V_0R^\top)$ for any $R \in \mathbb R^{r\times r}$.
We now describe several variants of \GNMR, which correspond to different solutions of \cref{eq:updatingVariant_LSQR}. 
Specifically, in the {\em updating variant} of \GNMR, we choose $(\Delta U, \Delta V)$ to be the minimal norm solution, namely the minimizer of \cref{eq:updatingVariant_LSQR} whose norm $\|\Delta U\|_F^2 + \|\Delta V\|_F^2$ is smallest.

Next, to describe the other variants of \GNMR, 
we define a one-dimensional family of solutions of
\cref{eq:updatingVariant_LSQR}, parametrized by a scalar $\alpha \in \mathbb R$.
By making a change of optimization variables $\Delta U = U-\frac{1+\alpha}{2} U_0$, $\Delta V = V-\frac{1+\alpha}{2} V_0$ in \cref{eq:updatingVariant_LSQR} we obtain
\begin{subequations}\label{eq:generalVariant}\begin{align}
\begin{pmatrix} \tilde U_0 \\ \tilde V_0 \end{pmatrix} &= \argmin_{U,V} \|\mathcal A\left(U_0 V^\top + U V_0^\top - \alpha U_0V_0^\top\right) - b\|^2, \label{eq:generalVariant_LSQR} \\
\begin{pmatrix} U_1 \\ V_1 \end{pmatrix} &= \begin{pmatrix} \frac{1-\alpha}{2}U_0 + \tilde U_0 \\ \frac{1-\alpha}{2}V_0 + \tilde V_0 \end{pmatrix}, \label{eq:generalVariant_update}
\end{align}\end{subequations}
where in \cref{eq:generalVariant_LSQR} we take 
the minimal norm solution with smallest $\|U\|_F^2+\|V\|_F^2$. 
%
The updating variant, for example, corresponds to $\alpha = -1$ in \cref{eq:generalVariant}.
Another two variants we consider in this work are the setting and the averaging variants. 
The \textit{setting variant}, corresponds to $\alpha = 1$, 
\begin{align}\label{eq:settingVariant}
\begin{pmatrix} U_1 \\ V_1 \end{pmatrix} &= \argmin_{U,V} \|\mathcal A\left(U_0 V^\top + U V_0^\top - U_0V_0^\top\right) - b\|^2,
\end{align}
minimizes the norm of the new estimate $\|U_1\|_F^2 + \|V_1\|_F^2$. As we shall see later on, this choice encourages the iterates to have bounded imbalance $\|U_1^\top U_1 - V_1^\top V_1\|_F$. Another variant with a similar property is the \textit{averaging variant}, which corresponds to $\alpha = 0$,
\begin{subequations}\label{eq:averagingVariant}\begin{align}
\begin{pmatrix} \tilde U \\ \tilde V \end{pmatrix} &= \argmin_{U, V} \|\mathcal A\left(U_0 V^\top + U V_0^\top\right) - b\|^2, 
	\label{eq:averagingVariant_LSQR}
\\
\begin{pmatrix} U_1 \\ V_1 \end{pmatrix} &= \begin{pmatrix} U_0/2 + \tilde U \\ V_0/2 + \tilde V \end{pmatrix} .
	\label{eq:averagingVariant_update}
\end{align}\end{subequations}
We emphasize that each choice of $\alpha$ yields a different algorithm in the following sense: In general, starting from the same initial condition $(U_0, V_0)$, already after one iteration each value of $\alpha$ yields a different $(U_1, V_1)$ and thus a different sequence $\{(U_t, V_t)\}$.

\GNMR is sketched in \cref{alg:GNMR}.
The minimal norm solution of the least squares problem can be computed with the LSQR algorithm \cite{paige1982lsqr}, implemented in most standard packages.
One of the inputs to \GNMR is an initial guess $U_0,V_0$. In our matrix completion simulations we initialized these values by the Singular Value Decomposition (SVD) of the observed matrix (a.k.a.~the spectral method). However, \GNMR performed well also from random initializations.
Note that \GNMR returns the best rank-$r$ approximation of the linearized estimate $U_{T-1} \tilde V_{T-1}^\top + \tilde U_{T-1} V_{T-1}^\top - \alpha U_{T-1} V_{T-1}^\top$, which is the last matrix fitted to the observations. When \GNMR converges, this quantity coincides with $U_TV_T^\top$.
Among the different variants,  we found that the setting one ($\alpha=1$) had the best empirical performance in matrix completion, especially at very low oversampling ratios, see \cref{sec:numerics}. This should not be surprising, as choosing the estimate with the minimal norm $\|U_1\|_F^2+\|V_1\|_F^2$ is akin to regularizing the norm of the estimate, a very common form of regularization in optimization. In matrix sensing, however, this type of regularization seems to be unnecessary, as the different variants of \GNMR have similar empirical performance.
In \cref{sec:comparison} we discuss the relation between \GNMR and three other methods: Wiberg's algorithm \cite{wiberg1976}, \texttt{PMF} \cite{paatero1994positive} and \texttt{R2RILS} \cite{bauch2021rank}.

Our \GNMR approach enjoys some appealing properties: it is easy to implement, requires no tuning parameters other than maximal number of iterations, it is computationally efficient and requires little memory.
In some sense, \GNMR combines the best of two popular approaches: it updates $U, V$ both \textit{globally}, as in alternating minimization, and \textit{simultaneously}, as in gradient descent. Finally, \GNMR exhibits excellent empirical performance, as illustrated in \cref{sec:numerics}, and also enjoys strong theoretical guarantees, as detailed in the following section.

\begin{algorithm}[t]
\caption{\GNMR} \label{alg:GNMR}
\SetKwInOut{Return}{return}
\SetKwInOut{Input}{input}
\SetKwInOut{Output}{output}
\Input{$\mathcal A: \mathbb R^{n_1\times n_2}\to \mathbb R^m$ - sensing operator ($\mathcal P_\Omega$ in the case of matrix completion) \\ 
$b \in \mathbb R^m$ - vector of observations \\
$r$ - rank of $X^*$ \\
$T$ - maximal number of iterations \\
$\alpha$ - a scalar that indicates the variant of \GNMR (e.g.~$\alpha=1$ is the setting variant) \\
$(U_0, V_0)\in \mathbb R^{n_1\times r}\times \mathbb R^{n_2\times r}$ - initialization
}
\Output{$\hat X$ - rank-$r$ (approximate) solution to $\mathcal A(X) = b$}
\For{ $t=1,\ldots,T$}{
compute $\tilde Z_t$, the minimal norm solution of
\[
\argmin_{U,V} \| \mathcal A(U_t V^\top + U V_t^\top - \alpha U_tV_t^\top) - b \|^2
\]
set $\begin{psmallmatrix} U_{t+1} \\ V_{t+1} \end{psmallmatrix} = \frac{1- \alpha}{2} \begin{psmallmatrix} U_{t} \\ V_{t} \end{psmallmatrix} + \begin{psmallmatrix} \tilde U_t \\ \tilde V_t \end{psmallmatrix}$ where $\begin{psmallmatrix} \tilde U_t \\ \tilde V_t \end{psmallmatrix} = \tilde Z_t$
}
\Return{$\hat X$, the best rank-$r$ approximation of $U_{T-1} \tilde V_{T-1}^\top + \tilde U_{T-1} V_{T-1}^\top - \alpha U_{T-1} V_{T-1}^\top$
}
\end{algorithm}

\section{Theoretical results for \GNMR}\label{sec:theory}
Let us start with some useful notations and definitions.
First, we recall the definition of the Restricted Isometry Property (RIP) for matrices \cite{candes2008restricted,recht2010guaranteed}.
\begin{definition}[Restricted Isometry Property]\label{def:RIP}
A linear map $\mathcal A: \mathbb R^{n_1\times n_2}\to \mathbb R^m$ satisfies an $r$-RIP with constant $\delta_r \in [0,1)$, if for all matrices $X \in \mathbb R^{n_1\times n_2}$ of rank at most $r$,
\begin{align*}
(1-\delta_r) \|X\|_F^2 \leq \|\mathcal A(X)\|^2 \leq (1+\delta_r) \|X\|_F^2 .
\end{align*}
\end{definition}

A common example for linear maps that satisfy the RIP are ensembles of Gaussian Matrices.
Let $\{A_i\}_{i=1}^m \subset \mathbb R^{n_1\times n_2}$ be $m$ measurement matrices, whose entries are independently drawn from a Gaussian distribution $\mathcal N(0,1)$. Then the corresponding linear map $\mathcal A$, defined by $[\mathcal A(X)]_i = \Tr(A_i^\top X)/\sqrt m$, satisfies an $r$-RIP with constant $\delta_r$ with high probability, provided that $m \gtrsim (n_1+n_2)r/\delta_r^2$ \cite{recht2010guaranteed}.
It is easy to show that if $\mathcal A$ satisfies a $2r$-RIP, then the matrix recovery problem \cref{eq:matrixRecovery_problem} is well posed, with a unique solution $X^*$. Moreover, this is the minimal sufficient condition in terms of RIP, as $(2r-1)$-RIP does not guarantee a unique solution.

In the matrix completion setup, the sampling operator $\mathcal P_\Omega$ does not satisfy an RIP. 
Instead, in our theoretical analysis, we assume that $\Omega$ is uniformly sampled at random and $|\Omega|$ is sufficiently large.
However, this assumption is insufficient to ensure well posedness of the matrix completion problem:
For example, the rank-$1$ matrix $X^*= e_i e_j^\top$ with a single non-zero value in its $(i,j)$-th entry
cannot be exactly recovered unless the $(i,j)$-th entry is observed. 
Hence, an additional standard assumption is incoherence of $X^*$, first introduced in \cite{candes2009exact}. In this work we adopt the following modified definition \cite{keshavan2010matrix}:
\begin{definition}[$\mu$-incoherence]\label{def:incoherence}
A matrix $X \in \mathbb R^{n_1\times n_2}$ of rank $r$ is $\mu$-incoherent if its SVD, $X = U \Sigma V^\top$ with $U \in \mathbb R^{n_1\times r}$ and $V \in \mathbb R^{n_2\times r}$, satisfies
\begin{align*}
\|U\|_{2,\infty} \leq \sqrt{\mu r/n_1}, \quad
\|V\|_{2,\infty} \leq \sqrt{\mu r/n_2}.
\end{align*}
\end{definition}
For convenience, we denote by $\mathcal M(n_1, n_2, r, \mu, \kappa)$ the set of all $\mu$-incoherent $n_1\times n_2$ matrices of rank $r$ and condition number $\kappa$.

Next, we define some relevant subsets of factor matrices $U,V$.
These or similar subsets have been considered in previous theoretical works on factorization-based matrix recovery methods, see \cite{keshavan2010matrix,sun2016guaranteed}.
First, we denote all the decompositions of rank-$r$ matrices with a bounded \textit{error} from $X^*$ by
\begin{align}\label{eq:B_e_def}
\mathcal B_\text{err}(\epsilon) &= \left\{\begin{pmatrix} U \\ V \end{pmatrix} \in \mathbb R^{(n_1+n_2)\times r} \,\mid\, \|UV^\top - X^*\|_F \leq \epsilon \sigma_r^* \right\},
\end{align}
where here and henceforth, $\sigma_r^* = \sigma_r(X^*)$.
In particular, we denote by $\mathcal B^* = \mathcal B_\textnormal{err}(0)$ the set of all decompositions of $X^*$,
\begin{align}\label{eq:B*_def}
\mathcal B^* = \left\{ \begin{pmatrix} U \\ V \end{pmatrix} \in \mathbb R^{(n_1+n_2)\times r} \,\mid\, UV^\top = X^* \right\} .
\end{align}
Second, we say that the factors $U,V$ are balanced if $U^\top U = V^\top V$, and measure the imbalance by $\|U^\top U - V^\top V\|_F$.
We denote all the factor matrices which are approximately \textit{balanced} by
\begin{align}
\mathcal B_\textnormal{bln}(\delta) &= \left\{\begin{pmatrix} U \\ V \end{pmatrix} \in \mathbb R^{(n_1+n_2)\times r} \,\mid\, \|U^\top U - V^\top V\|_F \leq \delta \sigma_r^* \right\} .
\end{align}
Third, we denote the subset of factor matrices with bounded row norms by
\begin{align}\label{eq:B_mu_def}
\mathcal B_\mu &= \left\{\begin{pmatrix} U \\ V \end{pmatrix} \in \mathbb R^{(n_1+n_2)\times r} \,\mid\, \|U\|_{2,\infty} \leq \sqrt\frac{3\mu r \sigma_1^*}{n_1}, \quad \|V\|_{2,\infty} \leq \sqrt\frac{3\mu r \sigma_1^*}{n_2} \right\} 
\end{align}
where $\mu$ is the incoherence parameter of $X^*$.
The constant $3$ in \cref{eq:B_mu_def} is arbitrary.

Finally, we denote the stacking of factor matrices $U,V$ by $Z$, namely $Z = \begin{psmallmatrix} U \\ V \end{psmallmatrix} \in \mathbb R^{(n_1+n_2)\times r}$. In particular, $Z_0 = \begin{psmallmatrix} U_0 \\ V_0 \end{psmallmatrix}$ is the initial iterate provided as input to \GNMR.

\begin{table}[t]
 \label{table:theory_summary}
\centering
 \caption{Recovery guarantees for \GNMR. All guarantees are with constant contraction factors, independent of the incoherence parameter $\mu$, the rank $r$ and the condition number $\kappa$.}
 \begin{tabular}{|c  c  c|}
 \hline
 Assumption & Basin of attraction & Recovery rate \\
 \hline\hline
 \multicolumn{3}{|c|}{\textbf{Matrix sensing}} \\
 \hline
 $2r$-RIP with $\delta_{2r} < 1$ & $\|X_0 - X^*\|_F = \mathcal O(\sigma_r^*)$ & quadratic \\
 same, with error $\|\xi\| = \mathcal O(\sigma_r^*)$ & $\|X_0 - X^*\|_F = \mathcal O(\sigma_r^*)$ & $\|\xi\|$-dependent \\
 \hline\hline
 \multicolumn{3}{|c|}{\textbf{Matrix completion}} \\
 \hline
 $np = \Omega(\mu r \max\{\log n, \mu r \kappa^2\})$ & $\|X_0 - X^*\|_F = \mathcal O (\sigma_r^*/\sqrt{\kappa})$ & linear \\ 
 $np = \Omega(\mu r\log n)$ & $\|X_0 - X^*\|_F = \mathcal O(\sigma_r^* \sqrt{p/\kappa})$ &  quadratic \\
 \hline
 \end{tabular}
\end{table}

\subsection{Recovery guarantees for matrix sensing}\label{sec:theory_sensing}
The following theorem states that in the noiseless matrix sensing setup, starting from a sufficiently accurate balanced initialization, \GNMR recovers $X^*$ with a quadratic convergence rate.
\begin{theorem}[Matrix sensing, quadratic convergence]\label{thm:sens_quadConvergence}
Let $\delta$ be any positive constant strictly smaller than one,
and let $c_e=c_e(\delta)$ be sufficiently large.
Assume that the sensing operator $\mathcal A$ satisfies a $2r$-RIP with 
$\delta_{2r} \leq \delta$. 
Let $X^*\in \mathbb R^{n_1\times n_2}$ be a matrix of rank $r$ and $b = \mathcal A(X^*)$.
Denote $\gamma = c_e/(2\sigma_r^*)$.
Then, for any initial iterate $Z_0 \in \mathcal B_\textnormal{err}({1}/{c_e}) \cap \mathcal B_\textnormal{bln}(1/(2c_e))$,
the estimates $X_t = U_tV_t^\top$ of \cref{alg:GNMR} with $\alpha = -1$ (the updating variant of \GNMR) satisfy
\begin{align}\label{eq:sens_quadConvergence}
\|X_{t+1} - X^*\|_F \leq \gamma\cdot \|X_t - X^*\|_F^2, \quad \forall t=0, 1, \ldots .
\end{align}
\end{theorem}
Note that the assumption $Z_0 \in \mathcal B_\textnormal{err}({1}/{c_e})$ implies $\gamma\cdot \|X_0 - X^*\|_F \leq 1/2$. Hence, by \cref{eq:sens_quadConvergence}, \GNMR exactly recovers $X^*$, since $X_t \to X^*$ as $t\to\infty$.

Before we compare \cref{thm:sens_quadConvergence} to previous works, we make several remarks.
The theorem is stated and proved only for the updating variant of \GNMR, which is the simplest to analyze. In simulations we noted that other \GNMR variants were also able to perfectly recover $X^*$. We thus conjecture that the theorem holds also for other variants. 
Next, assuming that the sensing operator $\mathcal A$ satisfies a $4r$-RIP with a sufficiently small constant $\delta_{4r}$, then an initialization $Z_0$ that satisfies the conditions of the theorem can be constructed in polynomial time as in \cite[Alg.~2]{tu2016low}, see \cite[proof of Eq.~(3.6) of their Theorem~3.3]{tu2016low}. 

The main ingredients in the proof of \cref{thm:sens_quadConvergence} are described in \cref{sec:sensing_proof_inText}.
A key property is that the factor matrices $U_t,V_t$ remain approximately balanced throughout the iterations of \GNMR.
We note that if the matrix to be recovered is positive semi-definite (PSD), $X^* = UU^\top$ with $U\in \mathbb R^{n\times r}$, a much simpler proof is possible for a slightly modified algorithm which explicitly enforces $U_t = V_t$, for which perfect balance holds trivially.

In fact, the need for a balance analysis can be avoided even in the general rectangular case, 
for a slightly modified variant of \GNMR which explicitly enforces the iterates to be perfectly balanced, see \cref{alg:GNMR_bSVD}.
Empirically, this variant of \GNMR has similar performance. Furthermore, it is provably stable against arbitrary additive error, and in particular works for approximately low rank $X^*$, as stated in the next theorem. 

\begin{theorem}[Noisy matrix sensing]\label{thm:sens_noisy}
Let $\delta$ be any positive constant strictly smaller than one,
and denote $c=7(1+\delta)^\frac{3}{2}/(1-\delta)^\frac{3}{2}$.
Assume that the sensing operator $\mathcal A$ satisfies a $2r$-RIP with 
$\delta_{2r} \leq \delta$.
Let $b = \mathcal A(X^*) + \xi$ where $X^* \in \mathbb R^{n_1\times n_2}$ is of rank $r$ and $\xi\in \mathbb R^m$ satisfies
\begin{align}
\|\xi\| \leq \frac{\sigma_r^*\sqrt{1-\delta}}{6c}.
\end{align}
Denote $\gamma = c/(4\sigma_r^*)$.
Then, for any initial iterate $Z_0 \in \mathcal B_\textnormal{err}(1/c)$,
the estimates $X_t = U_tV_t^\top$ of \cref{alg:GNMR_bSVD} with $\alpha = -1$ satisfy
\begin{align}\label{eq:sens_noisy}
\|X_{t+1} - X^*\|_F \leq \gamma\cdot \|X_t - X^*\|_F^2 + \frac{3\|\xi\|}{\sqrt{1-\delta}}, \quad \forall t=0, 1, \dots .
\end{align}
As a result, $\|X_{t} - X^*\|_F \leq \sigma_r^*/ (4^{2^{t}-1} c) + 6\|\xi\|/\sqrt{1-\delta} \stackrel{t\to\infty}{\longrightarrow} 6\|\xi\|/\sqrt{1-\delta}$.
\end{theorem}

\begin{algorithm}[t]
\caption{\GNMR with SVD-balancing step}
\label{alg:GNMR_bSVD}
\SetKwInOut{Return}{return}
\SetKwInOut{Input}{input}
\SetKwInOut{Output}{output}
\Input{same as \cref{alg:GNMR}, with initialization $(U'_0, V'_0)$}
\Output{$\hat X$ - rank-$r$ (approximate) solution to \cref{eq:matrixRecovery_problem}} 
\For{$t=0, \ldots ,T-1$}{
		compute the balanced factors $\begin{psmallmatrix} U_t \\ V_t \end{psmallmatrix} = \begin{psmallmatrix} U\Sigma^\frac{1}{2} \\ V\Sigma^\frac{1}{2} \end{psmallmatrix}$ where $U\Sigma V^\top = \text{SVD}\left(U'_t {V_t'}^{\top}\right)$ \\
	compute $\tilde Z_t$, the minimal norm solution of
	$ \argmin_{U,V} \| \mathcal A(U_t V^\top + U V_t^\top - \alpha U_tV_t^\top) - b \|^2 $ \\
	set $\begin{psmallmatrix} U'_{t+1} \\ V'_{t+1} \end{psmallmatrix} = \frac{1- \alpha}{2} \begin{psmallmatrix} U_{t} \\ V_{t} \end{psmallmatrix} + \begin{psmallmatrix} \tilde U_t \\ \tilde V_t \end{psmallmatrix}$ where $\begin{psmallmatrix} \tilde U_t \\ \tilde V_t \end{psmallmatrix} = \tilde Z_t$
	}
\Return{$\hat X$, the best rank-$r$ approximation of $U_{T-1} \tilde V_{T-1}^\top + \tilde U_{T-1} V_{T-1}^\top - \alpha U_{T-1} V_{T-1}^\top$ 
}
\end{algorithm}

In the absence of noise, $\xi = 0$, the guarantee of \cref{thm:sens_noisy} for \cref{alg:GNMR_bSVD} reduces to the exact recovery with quadratic rate of \cref{alg:GNMR} guaranteed by \cref{thm:sens_quadConvergence}.

\textbf{Comparison to previous works.}
Recht {\em et al.}~\cite{recht2010guaranteed} were the first to derive recovery guarantees in the matrix sensing setup. They proved that under suitable assumptions, nuclear norm minimization recovers the true rank-$r$
matrix $X^*$ {from an arbitrary initialization}. 
Recovery guarantees for factorization-based methods, with a linear convergence rate and assuming a sufficiently accurate initialization, were derived by various authors, see for example \cite{jain2013low,zheng2015convergent,tu2016low,ma2021beyond,tong2021accelerating}.
These works required more stringent RIP conditions than ours.
Moreover, the contraction factor in some of these works is not an absolute constant,
but rather depends on the problem parameters, such as the rank $r$ and the condition number $\kappa$.

To the best of our knowledge, only three recent works obtained results similar to our \cref{thm:sens_quadConvergence}. Yue {\em et al.}~\cite{yue2019quadratic} derived a recovery guarantee for a cubic regularization method from an arbitrary initialization, with an asymptotic quadratic convergence rate. However, they proved it only for a PSD matrix $X^*$, and required an RIP constant $\delta_{2r} < 1/10$. 
Charisopoulos {\em et al.}~\cite{charisopoulos2021low} proved quadratic convergence for a prox-linear algorithm whose objective is more complicated, as it involves a least squares term and an $\ell_1$ penalty term that requires delicate tuning.
Finally, Luo {\em et al.}~\cite{luo2020recursive} proved quadratic convergence for an importance sketching scheme, but required a $3r$-RIP assumption on $\mathcal A$.
Our quadratic rate guarantee, in contrast, holds in the general rectangular case for a computationally simple algorithm that solves a least squares problem at each iteration, and requires the minimal RIP condition of a $2r$-RIP with $\delta_{2r}<1$. 
As for the stability to additive error, \cref{thm:sens_noisy}, similar results were proved by \cite{charisopoulos2021low,luo2020recursive,tong2021low} for other algorithms.

\subsection{Recovery guarantees for matrix completion}\label{sec:theory_completion}
Similar to other works on matrix completion, we derive guarantees for a constrained version of \GNMR, described in \cref{alg:reg_GNMR}. Specifically, \cref{alg:reg_GNMR} is a constrained version of the setting variant ($\alpha=1$), but as explained below, the results in this section hold for all the (constrained) variants of \GNMR.
The only difference in this version is that its least squares problem is constrained to the subset $\mathcal B_\mu \cap \mathcal C^{(t)}$, where $\mathcal C^{(t)}$ is the following neighborhood of the current factor matrices $U_t,V_t$,
\begin{align}\begin{aligned}\label{eq:Ct_def}
\mathcal C^{(t)} =& \left\{\begin{pmatrix} U \\ V \end{pmatrix} \in \mathbb R^{(n_1+n_2)\times r} \,\mid\,  \|U - U_t\|^2_F + \|V - V_t\|^2_F \leq \frac{8}{p \sigma_r^*} \|X_t-X^*\|^2_{F(\Omega)}
\right\} .
\end{aligned}\end{align}
Similar constraints/regularizations were employed in previous works, see for example \cite{keshavan2010matrix,sun2016guaranteed}.
As these constraints are quadratic, the constrained problem may be equivalently written as a regularized least squares problem with quadratic regularization terms. Hence, each iteration of the constrained \GNMR can be solved computationally efficiently.
In \cref{claim:feasible_constraints}, we prove that starting from the initialization described in \cref{rem:comp_initialization} below, then w.h.p.~the constraints are feasible at all iterations, namely $\mathcal B_u \cap \mathcal C^{(t)} \neq \emptyset$ for all $t$.
Note that \cref{alg:reg_GNMR} requires as input the incoherence $\mu$ and the smallest non-zero singular value $\sigma_r^*$ of the true matrix $X^*$. If these quantities are unknown, they may be estimated from the observed data, see \cref{rem:comp_parametersInput}.
Finally, we emphasize that these constraints serve only for technical purposes in our theoretical analysis. In practice, \GNMR works well without them, and we did not employ them in our simulations.


\begin{algorithm}[t]
\caption{Constrained \GNMR for matrix completion (setting variant)}
\label{alg:reg_GNMR}
\SetKwInOut{Return}{return}
\SetKwInOut{Input}{input}
\SetKwInOut{Output}{output}
\Input{$\mathcal P_\Omega$ - sampling operator $\mathbb R^{n_1\times n_2}\to \mathbb R^m$ ($m = |\Omega|$) \\ 
	$b$ - observed entries of the underlying matrix $\mathcal P_\Omega(X^*)$ \\
	$r, \mu, \sigma_r^*$ - rank, incoherence parameter and $r$-th singular value of $X^*$ \\
	$T$ - maximal number of iterations \\
	$(U_0, V_0)\in \mathbb R^{n_1\times r}\times \mathbb R^{n_2\times r}$ - initialization}
\Output{$\hat X$ - rank-$r$ (approximate) solution to \cref{eq:matrixRecovery_problem} with $\mathcal A\to \mathcal P_\Omega$} 
\For{$t=0, \ldots ,T-1$}{
	compute $\begin{psmallmatrix} U_{t+1} \\ V_{t+1} \end{psmallmatrix} = \argmin \{\| \mathcal P_\Omega(U_t V^\top + U V_t^\top - U_tV_t^\top ) - b \|^2 \,\mid\, \begin{psmallmatrix} U \\ V \end{psmallmatrix} \in \mathcal B_\mu \cap \mathcal C^{(t)}\},$ \\
	\quad where $\mathcal B_\mu$ is defined in \cref{eq:B_mu_def} and $\mathcal C^{(t)}$ is defined in \cref{eq:Ct_def}
}
\Return{$\hat X$, the best rank-$r$ approximation of $U_{T-1} V_T^\top + U_T V_{T-1}^\top - U_{T-1} V_{T-1}^\top$}
\end{algorithm}

Below we present recovery guarantees for \GNMR in the matrix completion setting assuming ideal error-free measurements. Analyzing the stability to measurement error is left for future work.
The following theorem, proven in \cref{sec:comp_linearConvergence_proof}, states that starting from a sufficiently accurate balanced initialization with bounded row norms, \cref{alg:reg_GNMR} recovers $X^*$ with a linear convergence rate.
\begin{theorem}[Matrix completion, linear convergence]\label{thm:comp_linearConvergence}
There exist constants $C$, $c_e$, $c_l$ such that the following holds.
Let $X^* \in \mathcal M(n_1, n_2, r, \mu, \kappa)$.
Assume $\Omega \subseteq [n_1]\times [n_2]$ is randomly sampled with $np \geq C \mu r \max\{\log n, \mu r \kappa^2\}$.
Then w.p.~at least $1 - 3/n^3$, starting from any $Z_0 \in \mathcal B_\textnormal{err}(1/(c_e\sqrt\kappa)) \cap \mathcal B_\textnormal{bln}(1/c_l) \cap \mathcal B_\mu$,
the estimates $X_{t} = U_{t}V_{t}^\top$ of \cref{alg:reg_GNMR} satisfy
\begin{align*}
\|X_{t+1} - X^*\|_F \leq \tfrac{1}{2} \|X_{t} - X^*\|_F  .
\end{align*}
\end{theorem}

\Cref{thm:comp_linearConvergence}, as well as the following \cref{thm:comp_quadConvergence}, are stated for \cref{alg:reg_GNMR}, which is a constrained version of the setting variant of \GNMR ($\alpha=1$). However, they can be extended in a straightforward manner to any other variant. The technical reason is that the constraints replace the need to choose the minimal norm solution to the least squares problem in \cref{alg:reg_GNMR}, so that the proof works for any feasible solution.

\begin{remark}[Initialization for matrix completion]\label{rem:comp_initialization}
In \cref{lem:comp_initialization}, we prove that for a sufficiently large $|\Omega|$, a standard spectral-based initialization provides $Z_0 \in \mathcal B_\textnormal{err}(1/(c_e\sqrt\kappa)) \cap \mathcal B_\textnormal{bln}(1/c_l) \cap \mathcal B_\mu$. A similar initialization was employed in \cite{sun2016guaranteed,zheng2016convergence,yi2016fast}.
\end{remark}

\Cref{thm:comp_linearConvergence} guarantees a linear convergence rate.
As stated in the next theorem, once the error $\|X_t - X^*\|_F$ becomes small enough, the convergence rate becomes quadratic.
\begin{theorem}[Matrix completion, quadratic convergence]\label{thm:comp_quadConvergence}
There exist constants $C, c_e, c_l$ such that the following holds.
Let $X^* \in \mathcal M(n_1, n_2, r, \mu, \kappa)$.
Assume $\Omega \subseteq [n_1]\times [n_2]$ is randomly sampled with $np \geq C \mu r \log n$.
Then w.p.~at least $1 - 3/n^3$, starting from any initial iterate $Z_0 \in \mathcal B_\textnormal{err}(\sqrt p/(c_e\sqrt\kappa)) \cap \mathcal B_\textnormal{bln}(1/c_l) \cap \mathcal B_\mu$, the estimates $X_t = U_tV_t^\top$ of \cref{alg:reg_GNMR} satisfy
\begin{align*}
\|X_{t+1} - X^*\|_F \leq \gamma \|X_{t} - X^*\|_F^2
\end{align*}
where $\gamma \|X_{t} - X^*\|_F \leq 1/(2 \sqrt\kappa) \leq 1/2$.
\end{theorem}

\Cref{thm:comp_quadConvergence} is proven in \cref{sec:comp_quadConvergence_proof}. Combining it with \cref{thm:comp_linearConvergence} gives the following overall behavior of \GNMR: using the initialization procedure discussed in \cref{rem:comp_initialization}, \cref{alg:reg_GNMR} converges linearly according to \cref{thm:comp_linearConvergence}. After $t \sim \mathcal O( \log 1/p )$ iterations, it converges quadratically according to \cref{thm:comp_quadConvergence}.

We remark that the first condition in \cref{thm:comp_quadConvergence}, namely the stricter accuracy requirement $\|X_0 - X^*\|_F \lesssim \sqrt p / (c_e\sqrt\kappa)$, allows a reduced number of required observations compared to \cref{thm:comp_linearConvergence}.
Moreover, with such an accurate initial estimate $X_0$, \cref{thm:comp_quadConvergence} holds for a modified variant of \cref{alg:reg_GNMR} without the additional two conditions of balance and bounded row norms, $Z_0 \in \mathcal B_\textnormal{bln}(1/c_l)\cap \mathcal B_\mu$. In the modified variant we initialize $Z_0 = \begin{psmallmatrix} U \Sigma^{1/2} \\ V \Sigma^{1/2} \end{psmallmatrix}$ where $U \Sigma V^\top$ is the SVD of the initial estimate $X_0$. In addition, we may remove the constraint $Z_t \in \mathcal B_\mu$ from the iterative least squares problem of \cref{alg:reg_GNMR}.

\begin{remark}\label{rem:comp_parametersInput}
\Cref{thm:comp_linearConvergence,thm:comp_quadConvergence} assume that the parameters $\mu$ and $\sigma_r^*$ of the underlying matrix $X^*$ are known. Similar assumptions were made in previous works, e.g.~\cite{sun2016guaranteed,yi2016fast,zheng2016convergence}.
While in practice these parameters are often unknown, they can be estimated from the observed matrix.
The parameter $\sigma_r^*$, for example, can be estimated by $\hat\sigma_r = \sigma_r(X/p)$.
We show in \cref{sec:comp_estimating_proof} that if $np \geq C \mu r \kappa^2 \log n$ with a sufficiently large $C$, then with high probability $|\hat\sigma_r - \sigma_r^*|/\sigma_r^* \leq 1/10$.
Hence, \cref{alg:reg_GNMR} with $\hat\sigma_r$ in place of $\sigma_r^*$ enjoys the same recovery guarantees (with different constants).
\end{remark}

\textbf{Comparison to previous works.}
In terms of sample complexity, the best known recovery guarantee was derived by \cite{ding2020leave}, which required $np \sim \mathcal O(\mu r \log(\mu r) \log (n) )$.
However, this result holds for nuclear norm minimization, which is computationally demanding.
For factorization based methods, the recovery guarantee with the smallest sample complexity requirement was derived by \cite{zheng2016convergence} for projected gradient descent. Our \cref{thm:comp_linearConvergence} matches this result for \GNMR. The basin of attraction in our result, however, is smaller by a factor of $\sqrt\kappa$. Consequently, our initialization guarantee requires a larger sample complexity by a factor of $\kappa^2$. On the other hand, our linear convergence guarantee is amongst the first to hold with a constant contraction factor. \cite{zheng2016convergence}, for example, had a contraction factor of $1 - 1/\mathcal O(\mu^2 r^2 \kappa^2)$.
A constant contraction factor for a scaled variant of projected gradient descent was recently proved in \cite{tong2021accelerating}; however, their required sample complexity is larger than ours by a factor of $\kappa^2$.

Next, we discuss the quadratic convergence guarantee. Several Riemannian optimization methods are guaranteed an asymptotic quadratic rate of convergence, see for example \cite{mishra2013low, boumal2015low}.
These guarantees follow from general results in Riemannian optimization \cite{absil2007trust,absil2009optimization}. In these works, the basin of attraction and the required sample complexity for the quadratic convergence of their methods were not specified.
In contrast, \cref{thm:comp_quadConvergence} provides explicit expressions for the corresponding basin of attraction and for the required sample complexity.
To the best of our knowledge, the only work which obtained a similar result to \cref{thm:comp_quadConvergence} is \cite{kummerle2021scalable} for \texttt{MatrixIRLS}. However, their basin of attraction is significantly smaller: $\|X_0 - X^*\|_2/\sigma_r^* \lesssim \frac{\mu^{3/2} r^{1/2}}{n^2 (\log n)^{3/2} \kappa}$ compared to our $\|X_0 - X^*\|_F/\sigma_r^* \lesssim \sqrt\frac{\mu r \log n}{n \kappa}$.

\subsection{Uniform RIP for matrix completion}\label{sec:theory_RIP}
To prove \cref{thm:comp_linearConvergence} we first derive a novel RIP guarantee for matrix completion. This result may be of independent interest, e.g.~as a building block for proving recovery guarantees of other methods.
In contrast to the RIP assumed in the matrix sensing setup (see \cref{def:RIP}), here the RIP is \textit{local}, and applies to the difference $X-X^*$ where $X$ is a rank-$r$ matrix close to $X^*$. Formally, for a given $\epsilon \in (0, 1]$, we ask which rank-$r$ matrices $X \in \mathbb R^{n_1\times n_2}$ satisfy the following RIP inequalities,
\begin{align}\label{eq:RIP}
(1-\epsilon) \|X - X^*\|_F \leq \tfrac{1}{\sqrt p}\|\mathcal P_\Omega(X - X^*)\| \leq (1+\epsilon) \|X - X^*\|_F .
\end{align}
The local RIP guarantee we present below and prove in \cref{sec:RIP_proof} is uniform, namely it applies to all matrices $X$ in a neighborhood of $X^*$, independently of $\Omega$.
This allows us to avoid the sample splitting schemes which were employed in some early works. For a discussion on this issue see \cite[section I.B.2]{sun2016guaranteed}.
Since \GNMR is factorization-based, the RIP result we present poses requirements on the optimization variables $U,V$ such that \cref{eq:RIP} holds rather than on $X=UV^\top$.
One requirement is approximate balance of $U,V$.
This is the reason for the condition $Z_0 \in \mathcal B_\textnormal{bln}(1/c_l)$ in \cref{thm:comp_linearConvergence}: Such an initialization guarantees that the subsequent iterates of \GNMR remain approximately balanced.

\begin{theorem}[uniform RIP for matrix completion]\label{thm:RIP}
There exist constants $C, c_l, c_e$ such that the following holds.
Let $X^* \in \mathcal M(n_1, n_2, r, \mu, \kappa)$.
Let $\epsilon \in (0,1)$, and assume $\Omega \subseteq [n_1]\times [n_2]$ is randomly sampled with $np \geq \frac{C}{\epsilon^2} \mu r \max\left\{\log n, \frac{\mu r \kappa^2}{\epsilon^2}\right\}$.
Then w.p.~at least $1 - 3/n^3$, for all
matrices $X = UV^\top$ where 
 $\begin{psmallmatrix} U \\ V \end{psmallmatrix} \in \mathcal B_\textnormal{err}(\epsilon/c_e) \cap \mathcal B_\textnormal{bln}(1/c_l) \cap \mathcal B_\mu$, the RIP \cref{eq:RIP} holds.
\end{theorem}

\Cref{thm:RIP} is in several aspects sharper than the RIP guarantee of \cite[Claim~3.1]{sun2016guaranteed}. 
Specifically, \cite{sun2016guaranteed} required three conditions on the factor matrices $Z = \begin{psmallmatrix} U \\ V \end{psmallmatrix}$:
(i) $Z \in \mathcal B_\epsilon(\epsilon/(c_e r^{3/2}\kappa))$, which is more restrictive than ours by a factor of $r^{3/2}\kappa$;
(ii) $Z \in \mathcal B_{\mu \sqrt r}$, which is more restrictive by a factor of $\sqrt r$;
and (iii) the balance requirement $\|U\|_F^2 + \|V\|_F^2 \lesssim r \sigma_1^*$, which is replaced in our result by $Z \in \mathcal B_\textnormal{bln}(1/c_l)$.
More importantly, their guarantee requires a sample complexity of $np \gtrsim \mu r \kappa^2 \max\{\log n, \mu r^6 \kappa^4\}$,\footnote{In fact, \cite{sun2016guaranteed} required this sample complexity for additional results. For \cite[Claim~3.1]{sun2016guaranteed} by itself, it seems that $np \gtrsim \mu r \kappa^2 \max\{\log n, \mu r^5 \kappa^2\}$ suffices, see the end of the proof of \cite[Proposition~4.3]{sun2016guaranteed}.} compared to our $np \gtrsim \mu r \max\left\{\log n, \mu r \kappa^2\right\}$.

We remark that if the difference $\|X - X^*\|_F$ is of order $\mathcal O(\sqrt{\log(n)/n})$, then \cref{thm:RIP} holds without any additional requirements such as approximate balance or bounded row norms of the factor matrices, see \cref{lem:quad_RIP}.
We use this fact for our quadratic convergence guarantee (\cref{thm:comp_quadConvergence}), which requires a very accurate initialization.

A special case of \cref{thm:RIP} is a uniform RIP for the difference of incoherent matrices, as stated in the following corollary, proven in \cref{sec:RIP_proof}.
\begin{corollary}\label{corollary:RIP_incoherent}
Under the assumptions of \cref{thm:RIP} and with the same probability, for all rank-$r$, $3\mu/2$-incoherent matrices $X$ that satisfy $\|X - X^*\|_F \leq \epsilon \sigma_r^*/c_e$, the RIP \cref{eq:RIP} holds.
\end{corollary}
This corollary is not used in our proofs, but may be of independent interest.
In particular, it settles an open question posed in \cite{davenport2016overview}. 
In \cite[section V.B]{davenport2016overview} the authors wrote that an RIP holds for incoherent matrices, but "the difference between two sufficiently close incoherent matrices is not necessarily itself incoherent, which leads to some significant challenges in an RIP-based analysis." \Cref{corollary:RIP_incoherent} shows that although not incoherent, the difference between two incoherent matrices does satisfy an RIP.

\subsection{Stationary points analysis}\label{sec:theory_stationryPoints}
The previous subsections presented recovery guarantees for \GNMR under suitable assumptions on the initialization accuracy and on the number of observations. Without such assumptions, \GNMR is not guaranteed to converge at all. However, as it typically does converge, it is interesting to explore its set of stationary points.
In this subsection we analyze and compare the stationary points of \GNMR to those of two other methods: a regularized variant of gradient descent (\texttt{GD}), and the classical alternating least squares (\texttt{ALS}).
Specifically, as in \cite{tu2016low,zheng2016convergence,yi2016fast,park2018finding,li2020non,chen2020nonconvex}, consider \texttt{GD} applied to the regularized objective
\begin{align}
g_\lambda(Z) &= f(UV^\top) + \lambda\cdot \rho(Z), \label{eq:regularized_objective}
\end{align}
where $Z = \begin{psmallmatrix} U \\ V \end{psmallmatrix} \in \mathbb R^{(n_1+n_2)\times r}$, $\rho(Z) = \|U^\top U - V^\top V\|_F^2$ is an imbalance penalty, and $\lambda$ is a regularization parameter.
In particular, $\lambda=0$ corresponds to vanilla \texttt{GD}.
Starting from an initial $Z_0$, \texttt{GD} updates $Z_{t+1} = Z_t - \eta\cdot \nabla g_\lambda(Z_t)$ where the step-size $\eta$ may depend on $t$.

The second algorithm in the following comparison is \texttt{ALS} \cite{haldar2009rank,keshavan2012efficient,jain2013low}. Given an initial estimate $V_0 \in \mathbb R^{n_2\times r}$, \texttt{ALS} iteratively updates
\begin{align*}
U_{t+1} &= \argmin_{U} f(UV_t^\top), \quad
V_{t+1} = \argmin_{V} f(U_{t+1} V^\top) .
\end{align*}

Let $X^* \in \mathbb R^{n_1\times n_2}$ be of rank $r$, and consider the problem \cref{eq:matrixRecovery_factorizedObjective} with $b = \mathcal A(X^*)$.
Let $\mathcal F$ be the set of factor matrices at which the gradients of $f(UV^\top)$ w.r.t.~both $U$ and $V$ vanish,
\begin{align*}
\mathcal F = \left\{ \begin{pmatrix} U \\ V \end{pmatrix} \in \mathbb R^{(n_1+n_2)\times r} \,\mid\, \nabla f(UV^\top) V = 0,\, \nabla f(UV^\top)^\top U = 0 \right\},
\end{align*}
and let $\mathcal G \subseteq \mathbb R^{(n_1+n_2)\times r}$ be the set of balanced factor matrices ($U^\top U = V^\top V$).
Denote the sets of stationary points of vanilla \texttt{GD} ($\lambda=0$), regularized \texttt{GD} ($\lambda>0)$, \texttt{ALS}, the updating variant of \GNMR \cref{eq:updatingVariant} and the other variants of \GNMR (\cref{eq:generalVariant} with $\alpha\neq -1)$ by $\mathcal S_\text{GD}$, $\mathcal S_\text{reg-GD}$, $\mathcal S_\text{ALS}$, $\mathcal S_\text{updt-GNMR}$ and $\mathcal S_\text{GNMR}$, respectively.

\begin{theorem}[Stationary points]\label{thm:stationaryPoints}
The above sets of stationary points satisfy
\begin{subequations}\label{eq:stationaryPoints}\begin{align}
\mathcal S_\text{updt-GNMR} &= \mathcal S_\text{GD} = \mathcal S_\text{ALS} = \mathcal F, \label{eq:stationaryPoints_GD} \\
\mathcal S_\text{GNMR} &\subseteq \mathcal S_\text{reg-GD} = \mathcal F \cap \mathcal G. \label{eq:stationaryPoints_regGD}
\end{align}\end{subequations}
In addition, all the balanced global minima of \cref{eq:matrixRecovery_factorizedObjective} are stationary points of \GNMR, namely
\begin{align}
\left(\mathcal B^* \cap \mathcal G\right) &\subseteq \mathcal S_\text{GNMR}
\end{align}
where $\mathcal B^*$ is defined in \cref{eq:B*_def}, in the following two settings: (i)
In matrix sensing, where $\mathcal A$ satisfies a $2r$-RIP; (ii) With probability at least $1 - 3/n^3$ in matrix completion ($\mathcal A = \mathcal P_\Omega$), assuming $X^*$ is $\mu$-incoherent and the sampling pattern $\Omega \subseteq [n_1]\times [n_2]$ is randomly sampled with $np \geq C \mu r \log n$ for some constant $C$.
\end{theorem}

The identities $\mathcal S_\text{GD} = \mathcal F$ and $\mathcal S_\text{reg-GD} = \mathcal F \cap \mathcal G$ were discussed in previous works. For a detailed analysis of the geometry of the regularized \texttt{GD} objective \cref{eq:regularized_objective}, see \cite{ge2016matrix,ge2017no,zhu2018global,li2019symmetry}.
\Cref{thm:stationaryPoints} shows that the updating variant of \GNMR has a different behavior from the other variants.
Specifically, a parameter value $\alpha\neq -1$ in \GNMR, analogously to $\lambda>0$ in regularized \texttt{GD}, plays a role of an implicit balance regularizer in the sense that it enforces the stationary points to be balanced.
This theoretical observation supports the empirical finding that in matrix completion, the \GNMR variants with $\alpha\neq -1$ are superior to the updating variant, see \cref{sec:numerics}.

In addition, the theorem states that the stationary points of \GNMR variants with $\alpha\neq -1$ form a subset of those of regularized \texttt{GD}, $\mathcal S_\text{GNMR} \subseteq \mathcal S_\text{reg-GD}$, but do not necessarily coincide with them. The question if this is a desirable property of \GNMR depends on whether $\mathcal S_\text{GNMR}$ 'loses' some of the global minima in $\mathcal S_\text{reg-GD}$, or just bad local minima.
This is where the second part of the theorem comes into play: it states that in the matrix sensing and matrix completion settings, $\mathcal S_\text{GNMR}$ contains \textit{all} the balanced minimizers of \cref{eq:matrixRecovery_factorizedObjective}.

Finally, the recovery guarantees for \GNMR, \cref{thm:sens_quadConvergence,thm:comp_linearConvergence,thm:comp_quadConvergence}, required certain conditions on the initialization.
In contrast, several works \cite{ge2016matrix,ge2017no,zhu2018global,li2019symmetry} proved that regularized \texttt{GD} enjoys recovery guarantees from a random initialization due to a benign optimization landscape.
The similarity between the stationary points of the variants of \GNMR with $\alpha\neq -1$ and regularized \texttt{GD}, as implied by \cref{thm:stationaryPoints}, together with empirical evidence, suggest that an analogous result may hold also for \GNMR. Namely, even though many stationary points of \GNMR are local minima, it seems that the algorithm somehow avoids them.
We leave this open question for future research.

\subsection{Implicit balance regularization}\label{sec:implicit_balance}
Optimizing the factorized objective \cref{eq:matrixRecovery_factorizedObjective} rather than the original one \cref{eq:matrixRecovery_problem} introduces a scaling ambiguity: if $X^* = UV^\top$ where $U \in \mathbb R^{n_1\times r}$ and $V \in \mathbb R^{n_2\times r}$, then $X^* = (UQ)(VQ^{-\top})^\top$ for any invertible $r\times r$ matrix $Q$. As a result, the scales of the factor matrices $U,V$ may be highly imbalanced, e.g.~$\|U\|_F \ll \|V\|_F$. This may lead to significant challenges, involving two aspects: geometric and algorithmic.
The first aspect was discussed in the preceding subsection: in short, an imbalance penalty often leads to a benign optimization landscape.
However, while making the analysis easier, a recent work \cite{li2020global} proved that in some matrix recovery problems, including matrix sensing, the imbalance penalty is in fact unnecessary from this geometrical perspective.

Here we highlight that keeping the factors approximately balanced has important consequences from an algorithmic viewpoint.
If the factor matrices $U, V$ are not balanced, small changes in $U,V$ may lead to huge changes in the resulting estimate $UV^\top$. This ill-conditioning can lead to both computational problems as well as significant challenges in the theoretical analysis of matrix recovery algorithms.
While the iterates of (vanilla) gradient descent enjoy implicit balance regularization, as was recently shown by \cite{ma2019implicit,ma2021beyond,tong2021accelerating,ye2021global,wang2021large}, its available recovery guarantees require stringent conditions.
Several works on factorization-based methods explicitly added a balance regularization term to their algorithm to ease its analysis. The regularization term is either of the form $\|U\|_F^2 + \|V\|_F^2$ \cite{sun2016guaranteed,chen2020noisy,chen2021bridging} or $\|U^\top U - V^\top V\|_F^2$ \cite{tu2016low,zheng2016convergence,yi2016fast,park2018finding,zhang2018fast,li2020non,chen2020nonconvex}.
In contrast, \GNMR has a built-in implicit balance regularization, which manifests itself both during the iterates (\cref{lem:minNormSol_implicitBalance}) and in the set of stationary points (\cref{thm:stationaryPoints}). As our analysis shows, the underlying reason is the choice of the \textit{minimal norm solution} to the degenerate least squares problem \cref{eq:generalVariant_LSQR}.

\section{\Cref{thm:sens_quadConvergence} proof outline and key lemmas}\label{sec:sensing_proof_inText}
In this section we describe the skeleton of the proof of \cref{thm:sens_quadConvergence}. The proof relies on three key lemmas, \cref{lem:sens_et+1_bound,lem:minNormSol_implicitBalance,lem:minNormSol_normBound}, which we formally state below. The full proof appears in \cref{sec:sensing_proof_SM}.

Let $Z_t = \begin{psmallmatrix} U_t \\ V_t \end{psmallmatrix}$ be the current iterate of \GNMR, and denote the current and next estimates by $X_t = U_tV_t^\top$ and $X_{t+1} = U_{t+1}V_{t+1}^\top$, respectively.
Recall that the least squares problem \cref{eq:updatingVariant_LSQR} has an infinite number of solutions, and the updating variant of \GNMR, which corresponds to \cref{alg:GNMR} with $\alpha=-1$, chooses the one with minimal norm of the update
$\Delta Z_t$.
The proof of \cref{thm:sens_quadConvergence} proceeds as follows.
First, in \cref{lem:sens_et+1_bound} we show that if the current iterate $Z_t$ is approximately balanced and has a sufficiently small error $\|X_t - X^*\|_F$, then any feasible solution $\Delta Z_t$ to the least squares problem \cref{eq:updatingVariant_LSQR} satisfies $\|X_{t+1} - X^*\|_F \lesssim \|X_t - X^*\|_F^2/\sigma_r^* + \|\Delta Z_t\|_F^2$. 
Next, we show that by taking the minimal norm solution, the following two key properties hold:
(i) $\|\Delta Z_t\|_F^2$ is comparable to $\|X_t-X^*\|_F^2/\sigma_r^*$ (\cref{lem:minNormSol_normBound}), and (ii) the next iterate remains approximately balanced (\cref{lem:minNormSol_implicitBalance}), so we may apply \cref{lem:sens_et+1_bound}. This yields quadratic convergence of the form $\|X_{t+1} - X^*\|_F \lesssim \|X_t - X^*\|_F^2/\sigma_r^*$, thus completing the proof of \cref{thm:sens_quadConvergence}.

Let us now formally state the lemmas mentioned above.
For all three lemmas we assume $X^* \in \mathbb R^{n_1\times n_2}$ is of rank $r$ and that the sensing operator $\mathcal A$ satisfies a $2r$-RIP with constant $\delta_{2r} < 1$. Hence, it also satisfies an $r$-RIP with constant $\delta_r \leq \delta_{2r}$.

\begin{lemma}[error contraction]\label{lem:sens_et+1_bound}
Let $\Delta Z_t = \begin{psmallmatrix} \Delta U_t \\ \Delta V_t \end{psmallmatrix}$ be any feasible solution to \cref{eq:updatingVariant_LSQR}, not necessarily the minimal norm one. Assume that the current estimate $X_t = U_tV_t^\top$ satisfies
\begin{align}\label{eq:dP_assumption}
\|X_t - X^*\|_F^2 + \frac 14 \|U_t^\top U_t - V_t^\top V_t\|_F^2 \leq \tfrac{\sqrt 2 - 1}{400}\sigma_r^2(X^*).
\end{align}
Then the next estimate $X_{t+1} = U_{t+1}V_{t+1}^\top = (U_t + \Delta U_t)(V_t + \Delta V_t)^\top$ satisfies
\begin{align*}
\|X_{t+1} - X^*\|_F \leq \frac 12 \sqrt\frac{1+\delta_r}{1-\delta_{2r}} \left( \frac{25}{4\sigma_r(X^*)} \|X_t - X^*\|_F^2 + \|\Delta Z_t\|_F^2 \right) .  
\end{align*} 
\end{lemma}

For the next two lemmas, the factor matrices $U_t, V_t$ do not have to satisfy condition \cref{eq:dP_assumption}, but are required to have full column rank.
\begin{lemma}[norm of minimal norm solution]\label{lem:minNormSol_normBound}
Let $\begin{psmallmatrix} U_t \\ V_t \end{psmallmatrix}$ be of full column rank. Then the minimal norm solution $\Delta Z_t$ to \cref{eq:updatingVariant_LSQR} satisfies
\begin{align}\label{eq:minNormSol_normBound}
\|\Delta Z_t\|^2_F \leq \frac{1+\delta_{2r}}{1-\delta_{2r}} \frac{\|X_t - X^*\|_F^2}{\min\{\sigma^2_r(U_t), \sigma^2_r(V_t)\}} .
\end{align}
\end{lemma}

\begin{lemma}[balance of minimal norm solution]\label{lem:minNormSol_implicitBalance}
Let $\begin{psmallmatrix} U_t \\ V_t \end{psmallmatrix}$ be of full column rank. Then the next iterate $Z_{t+1} = \begin{psmallmatrix} U_{t+1} \\ V_{t+1} \end{psmallmatrix}$ given by \cref{eq:updatingVariant} satisfies
\begin{align*}\begin{aligned}
\|U_{t+1}^\top U_{t+1} - V_{t+1}^\top V_{t+1}\|_F &\leq \|U_{t}^\top U_{t} - V_{t}^\top V_{t}\|_F + \frac{1+\delta_{2r}}{1-\delta_{2r}} \frac{\|X_t - X^*\|_F^2}{\min\{\sigma^2_r(U_t), \sigma^2_r(V_t)\}} .
\end{aligned}\end{align*}
\end{lemma}

In the next subsections we prove the three lemmas.
The proofs provide useful insights on the inner mechanism of \GNMR, in particular on the importance of balanced factors and the role of the minimal norm solution. These proofs may be relevant to the analysis of Gauss-Newton based methods in other settings, especially for rank-deficient problems.
To prove the lemmas, in the following subsection we present a key auxiliary lemma, with implications beyond matrix sensing.
Next, in \cref{sec:Q_distance}, we introduce some definitions and related technical results.
Then, in \cref{sec:sensingLemmas_proofs_inText} we prove \cref{lem:sens_et+1_bound,lem:minNormSol_normBound,lem:minNormSol_implicitBalance}.

\subsection{A key property of \GNMR in the general matrix recovery problem}
Given $\begin{psmallmatrix} U_t \\ V_t \end{psmallmatrix} \in \mathbb R^{(n_1+n_2)\times r}$ and a sensing operator $\mathcal A: \mathbb R^{n_1\times n_2}\to \mathbb R^m$, define the linear operators $\mathcal L^{(t)}: \mathbb R^{(n_1+n_2)\times r} \to \mathbb R^{n_1\times n_2}$ and $\mathcal L_A^{(t)}: \mathbb R^{(n_1+n_2)\times r} \to \mathbb R^{m}$ as
\begin{align}\label{eq:L_LA_operators}
\mathcal L^{(t)}\left( Z \right) &= U_tV^\top + UV_t^\top, \quad\quad
\mathcal L^{(t)}_A\left(Z \right) = \mathcal A \mathcal L^{(t)}\left(Z \right) = \mathcal A\left(U_tV^\top + UV_t^\top\right),
\end{align}
where $Z = \begin{psmallmatrix} U \\ V \end{psmallmatrix}$.
Note that $\mathcal L^{(t)}_A$ is the operator of the least squares problem \cref{eq:generalVariant_LSQR} common to all variants of \GNMR.
The following lemma describes a useful property of the minimal norm solution to the least squares problem, that holds regardless of the specific setting.

\begin{lemma}\label{lem:kernel_minNormSol}
Let $\begin{psmallmatrix} U_t \\ V_t \end{psmallmatrix} \in \mathbb R^{(n_1+n_2)\times r}$. The minimal norm solution $\begin{psmallmatrix} \tilde U_t \\ \tilde V_t \end{psmallmatrix}$ of \cref{eq:generalVariant_LSQR} satisfies
\begin{align}\label{eq:minNormSol_property}
\tilde U_t^\top U_t = V_t^\top \tilde V_t .
\end{align}
Further, if $U_t, V_t$ have full column rank, then
\begin{align}\label{eq:L_perp}
\left(\ker \mathcal L^{(t)}\right)^\perp = \left\{ \begin{pmatrix} U \\ V \end{pmatrix} \in \mathbb R^{(n_1+n_2)r} \,\mid\, U^\top U_t = V_t^\top V \right\} .
\end{align}

\end{lemma}

\begin{proof}
Let $\mathcal K_t = \left\{\begin{psmallmatrix} U_t R \\ -V_tR^\top\end{psmallmatrix} \,\mid\, R \in \mathbb R^{r\times r} \right\}$. Observe that
\begin{align}\label{eq:Kt_inclusion_LA_kernel}
\mathcal K_t \subseteq \ker \mathcal L^{(t)} \subseteq \ker \mathcal L_A^{(t)},
\end{align}
where the first inclusion follows since $U_t V^\top + UV_t^\top$ vanishes for any $\begin{psmallmatrix} U \\ V \end{psmallmatrix} \in \mathcal K_t$, and the second due to the linearity of $\mathcal A$. 
By definition, the minimal norm solution $\begin{psmallmatrix} \tilde U_t \\ \tilde V_t \end{psmallmatrix}$ is orthogonal to $\ker \mathcal L_A^{(t)}$, and in particular to $\mathcal K_t \subseteq \ker \mathcal L_A^{(t)}$. Equation~\eqref{eq:minNormSol_property} will thus follow if we show
\begin{align}\label{eq:Kt_perp}
\mathcal K_t^\perp = \left\{ \begin{pmatrix} U \\ V \end{pmatrix} \in \mathbb R^{(n_1+n_2)r} \,\mid\, U^\top U_t = V_t^\top V \right\} .
\end{align}
Let $\begin{psmallmatrix} U \\ V \end{psmallmatrix} \in \mathbb R^{(n_1+n_2)r}$. Then $\begin{psmallmatrix} U \\ V \end{psmallmatrix} \perp \mathcal K_t$ if and only if
\begin{align*}
0 &= \Tr\left( U^\top U_t R - V^\top V_t R^\top \right)= \Tr\left[ \left( U^\top U_t - V_t^\top V\right) R\right], \quad \forall R\in \mathbb R^{r\times r},
\end{align*}
which in turn holds if and only if $U^\top U_t = V_t^\top V$. This proves \cref{eq:Kt_perp}.

Next, we prove \cref{eq:L_perp} assuming $U_t, V_t$ have full column rank. In view of \cref{eq:Kt_perp}, it is sufficient to show that in this case $\mathcal K_t = \ker \mathcal L^{(t)}$.
Let $\mathcal M_r$ denote the manifold of $\mathbb R^{n_1\times n_2}$ matrices of rank $r$. By \cite[Proposition~2.1]{vandereycken2013low}, the range of $\mathcal L^{(t)}$ is the tangent space to $\mathcal M_r$ at $U_t Q V_t$ where $Q \in \mathbb R^{r\times r}$ is any invertible matrix. Hence its dimension is the same as that of $\mathcal M_r$ \cite[section 1.2]{guillemin2010differential}, which is $(n_1+n_2-r)r$. The dimension of $\ker \mathcal L^{(t)}$ is therefore $(n_1+n_2)r - (n_1+n_2-r)r = r^2$. Since $U_t, V_t$ have full column rank, then $\text{dim } \mathcal K_t = r^2$. Combined with $\mathcal K_t \subseteq \ker \mathcal L^{(t)}$ \cref{eq:Kt_inclusion_LA_kernel}, we conclude $\mathcal K_t = \ker \mathcal L^{(t)}$. This completes the proof.
\end{proof}

\subsection{The Q-distance}\label{sec:Q_distance}
Consider the following distance measure between pairs of factor matrices, introduced by Ma \textit{et al.} \cite{ma2021beyond}.
\begin{definition}\label{def:Q_distance}
Let $Z_i = \begin{psmallmatrix} U_i \\ V_i \end{psmallmatrix}$ where $U_i \in \mathbb R^{n_1\times r}$ and $V_i \in \mathbb R^{n_2\times r}$ for $i=1,2$. Then the Q-distance between $Z_1$ and $Z_2$ is defined as
\begin{align*} 
d_Q^2(Z_1, Z_2) = \inf \left\{ \|U_1-U_2Q\|_F^2 + \|V_1-V_2Q^{-\top}\|_F^2 \,\mid\, Q \in \mathbb R^{r\times r} \text{ is invertible} \right\} . 
\end{align*}
\end{definition}
Let us present bounds on the Q-distance.
To this end, we first introduce the definition of balanced-SVD (b-SVD), which is quite natural in light of the discussion in \cref{sec:implicit_balance}.
\begin{definition}[balanced-SVD (b-SVD)]\label{def:bSVD}
Let $X \in \mathbb R^{n_1\times n_2}$ be a matrix of rank $r$ with SVD $X = \bar U \Sigma \bar V^\top$.
Then
\begin{align*}
\text{b-SVD}(X) = \begin{pmatrix} \bar U \Sigma^\frac{1}{2} \\ \bar V \Sigma^\frac{1}{2} \end{pmatrix} \in \mathbb R^{(n_1+n_2)\times r}.
\end{align*}
\end{definition}
Note that $\begin{psmallmatrix} U \\ V \end{psmallmatrix} = \text{b-SVD}(X)$ implies $UV^\top = X$ and $U^\top U = V^\top V$.

The next lemma bounds the Q-distance between a pair of factor matrices $Z$ and the b-SVD $Z^*$ of a rank $r$ matrix $X^*$.
Since the Q-distance is asymmetric w.r.t.~its arguments, we present different bounds for $d_Q(Z^*, Z)$ and $d_Q(Z, Z^*)$.
There is a substantial difference in the difficulty of bounding each case:
As the right argument of $d_Q$ is multiplied by an invertible matrix $Q$ (\cref{def:Q_distance}), we can assume w.l.o.g.~that it is also a b-SVD, and hence $d_Q(Z^*, Z)$ is easier to analyze.
The more challenging bound of $d_Q(Z, Z^*)$ requires an additional condition as stated in the following lemma, proven in \cref{sec:technical}.

\begin{lemma}\label{lem:dQ_e_bound}
Let $Z^* = \text{b-SVD}(X^*)$ where $X^* \in \mathbb R^{n_1\times n_2}$ is of rank $r$. Then
\begin{align}\label{eq:dQ_leftbSVD_e_bound}
d_Q^2(Z^*, Z) \leq \frac{\|UV^\top - X^*\|_F^2}{(\sqrt 2 - 1)\sigma_r(X^*)}
\end{align}
for any $Z = \begin{psmallmatrix} U \\ V \end{psmallmatrix} \in \mathbb R^{(n_1+n_2)\times r}$.
Further, if
\begin{align}
\|UV^\top - X^*\|_F^2 + \frac 14 \|U^\top U - V^\top V\|_F^2 \leq \tfrac{\sqrt 2 - 1}{400}\sigma_r^2(X^*), \label{eq:dQ_e_bound_assumption}
\end{align}
then there exists an invertible matrix $Q \in \mathbb R^{r\times r}$ with $\|Q\|_2 \leq 4/3$ such that
\begin{align}\label{eq:dQ_e_bound}
d_Q^2(Z, Z^*) \leq \|U - U^*Q\|_F^2 + \|V - V^* Q^{-\top}\|_F^2 &\leq \frac{25}{4\sigma_r(X^*)} \|UV^\top - X^*\|_F^2.
\end{align}
\end{lemma}
As the Q-distance takes an infimum over a non-compact set, its corresponding optimal alignment matrix does not always exist. However, in the proof of \cref{lem:dQ_e_bound} we show that \cref{eq:dQ_e_bound_assumption} is a sufficient condition for the existence of the optimal alignment matrix.

\subsection{Proofs of \cref{lem:sens_et+1_bound,lem:minNormSol_normBound,lem:minNormSol_implicitBalance} \nopunct}\label{sec:sensingLemmas_proofs_inText}
\begin{proof}[Proof of \cref{lem:sens_et+1_bound}]
Let $F_t^2$ be the objective function of the least squares problem,
\begin{align*}
F_t^2\left(\begin{pmatrix} \Delta U \\ \Delta V \end{pmatrix}\right) = \left\| \mathcal A\left(U_tV_t^\top + U_t \Delta V^\top + \Delta U V_t^\top - X^*\right) \right\|^2.
\end{align*}
The proof consists of two parts.
First, we show that there exists $\Delta Z \in \mathbb R^{(n_1+n_2)\times r}$ such that
\begin{align}\label{eq:sens_Ft_upperBound}
F_t(\Delta Z) \leq \frac{\sqrt{1+\delta_r}}{2} \frac{25}{4\sigma_r(X^*)} \|X_t - X^*\|_F^2.
\end{align}
Second, we show that any feasible solution $\Delta Z_t = \begin{psmallmatrix} \Delta U_t \\ \Delta V_t \end{psmallmatrix}$ to the least squares problem satisfies
\begin{align}\label{eq:sens_Ft_lowerBound}
F_t(\Delta Z_t) &\geq \sqrt{1-\delta_{2r}} \|X_{t+1} - X^*\|_F - \tfrac{1}{2} \sqrt{1+\delta_r} \|\Delta Z_t\|_F^2 
\end{align}
where $X_{t+1} = (U_t + \Delta U_t)(V_t + \Delta V_t)^\top$ is the corresponding new estimate.
Since $\Delta Z_t$ minimizes $F_t$ by construction, then $F_t(\Delta Z_t) \leq F_t(\Delta Z)$ for any $\Delta Z$, from which the lemma follows.

For the first part, let $Z = \begin{psmallmatrix} U \\ V \end{psmallmatrix} \in \mathcal B^*$, where $\mathcal B^*$ is defined in \cref{eq:B*_def}, be any decomposition of the underlying matrix $X^*$, and denote $\Delta Z = \begin{psmallmatrix} \Delta U \\ \Delta V \end{psmallmatrix} = Z - Z_t$. Since $U_tV_t + U_t \Delta V^{\top} + \Delta U V_t^\top - UV^{\top} = -\Delta U \Delta V^{\top}$, then
\begin{align*}
F_t(\Delta Z) &= \|\mathcal A\left(U_tV_t + U_t \Delta V^{\top} + \Delta U V_t^\top - UV^{\top} \right)\| = \|\mathcal A\left(\Delta U \Delta V^{\top}\right)\|.
\end{align*}
Recall that $\mathcal A$ satisfies an $r$-RIP with a constant $\delta_r \leq \delta_{2r}$. Combining this with the Cauchy-Schwarz inequality and the fact that $ab \leq (a^2+b^2)/2$ we obtain
\begin{align}\label{eq:sens_Ft_UpperBound_temp}
F_t(\Delta Z) &\leq \sqrt{1 + \delta_r} \|\Delta U\|_F \|\Delta V\|_F \leq \tfrac{1}{2} \sqrt{1 + \delta_r} \|\Delta Z\|_F^2 .
\end{align}
Let us now pick a specific decomposition $Z \in \mathcal B^*$.
By \cref{eq:dP_assumption},
$Z_t$ satisfies \cref{eq:dQ_e_bound_assumption}.
\Cref{lem:dQ_e_bound} thus guarantees an invertible $Q \in \mathbb R^{r\times r}$ such that \cref{eq:dQ_e_bound} holds.
Define $Z = \begin{psmallmatrix} U^*Q \\ V^*Q^{-\top} \end{psmallmatrix}$ where $\begin{psmallmatrix} U^* \\ V^* \end{psmallmatrix} = \text{b-SVD}(X^*)$.
Then $\|\Delta Z\|_F^2 = \|U_t - U^*Q\|_F^2 + \|V_t - V^* Q^{-\top}\|_F^2$.
Hence, by \cref{eq:dQ_e_bound} of \cref{lem:dQ_e_bound}, $\|\Delta Z\|_F^2 \leq \frac{25}{4\sigma_r(X^*)} \|X_t - X^*\|_F^2$. Plugging this into \cref{eq:sens_Ft_UpperBound_temp} yields \cref{eq:sens_Ft_upperBound}.

Next, we prove \cref{eq:sens_Ft_lowerBound}. It is easy to show that any feasible solution $\Delta Z_t$ satisfies
\begin{align*}
F_t(\Delta Z_t) &= \|\mathcal A\left(U_tV_t + U_t \Delta V_t^\top + \Delta U_t V_t^\top - X^* \right)\|
= \|\mathcal A\left(X_{t+1} - \Delta U_t \Delta V_t^\top - X^*\right)\|.
\end{align*}
By the triangle inequality, $F_t(\Delta Z_t) \geq \|\mathcal A\left(X_{t+1} - X^*\right)\| - \|\mathcal A\left( \Delta U_t \Delta V_t^\top \right)\|$.
As noted before, $\mathcal A$ satisfies an $r$-RIP with a constant $\delta_r \leq \delta_{2r}$. Combining this with the Cauchy-Schwarz inequality and the fact that $ab \leq (a^2+b^2)/2$ yields \cref{eq:sens_Ft_lowerBound}.
\end{proof}

To prove \cref{lem:minNormSol_normBound} we shall use the following proposition \cite[Proposition~B.4]{sun2016guaranteed}.
\begin{proposition}\label{prop:propB4_SL16}
For any $U\in \mathbb R^{n_1\times r}$ and $V\in \mathbb R^{n_2\times r}$,
$
\|UV^\top\|_F \leq \sigma_1(U)\|V\|_F.
$
Further, if $n_1 \geq r$, then
$
\sigma_r(U) \|V\|_F \leq \|UV^\top\|_F .
$
\end{proposition}

\begin{proof}[Proof of \cref{lem:minNormSol_normBound}]
Let $E_t = X^* - X_t$ and $b_t = \mathcal A (E_t)$.
By construction, $\Delta Z_t = \mathcal L_A^{(t)\dagger} b_t$ where $\mathcal L_A^{(t)\dagger}$ is the Moore-Penrose pseudoinverse of $\mathcal L_A^{(t)}$.
In addition, by the $2r$-RIP property of $\mathcal A$, $\|b_t\|^2 \leq (1+\delta_{2r}) \|E_t\|^2_F$. Hence,
\begin{align*}
\|\Delta Z_t\|_F^2 &= \|\mathcal L_A^{(t)\dagger} b_t\|_F^2 \leq (1+\delta_{2r}) \sigma^2_1 \left(\mathcal L_A^{(t)\dagger}\right) \|E_t\|_F^2 
= (1+\delta_{2r}) {\|X_t - X^*\|_F^2} / {\sigma^{2}_\text{min}\left(\mathcal L_A^{(t)}\right)}
\end{align*}
where $\sigma_\text{min}(\mathcal L_A^{(t)})$ is the smallest nonzero singular value of $\mathcal L_A^{(t)}$.
By the $2r$-RIP of $\mathcal A$ we have $\sigma^2_\text{min}(\mathcal L_A^{(t)}) \geq (1-\delta_{2r}) \sigma^2_\text{min}(\mathcal L^{(t)})$.
Proving inequality \cref{eq:minNormSol_normBound} thus reduces to showing that
\begin{align}\label{eq:sigmaMinL_sigmaMinsUV}
\sigma^2_\text{min}\left(\mathcal L^{(t)}\right) \geq \min\{\sigma^2_r(U_t), \sigma^2_r(V_t)\}.
\end{align}
Since $U_t, V_t$ have full column rank, \cref{eq:L_perp} of \cref{lem:kernel_minNormSol} implies
\begin{align}
\sigma^{2}_\text{min}\left(\mathcal L^{(t)}\right) &= \min_{\Delta U, \Delta V}\left\{ \frac{1}{\|\Delta U\|_F^2 + \|\Delta V\|_F^2} \left\|\mathcal L^{(t)}\left( \begin{pmatrix} \Delta U \\ \Delta V \end{pmatrix} \right) \right\|_F^2 \,\mathrel{\Big|}\, \begin{pmatrix} \Delta U \\ \Delta V \end{pmatrix} \perp \ker \mathcal L^{(t)} \right\} \nonumber \\
&= \min_{\Delta U, \Delta V}\left\{ \frac{1}{\|\Delta U\|_F^2 + \|\Delta V\|_F^2} \|U_t \Delta V^\top + \Delta U V_t^\top\|_F^2 \,\mid\, U_t^\top \Delta U = \Delta V^\top V_t \right\}. \label{eq:sens_sigmaMin_tempBound}
\end{align}
Let us lower bound $\|U_t \Delta V^\top + \Delta U V_t^\top\|_F^2$ under the constraint $U_t^\top \Delta U = \Delta V^\top V_t$.
For any $\Delta U, \Delta V$ that satisfy this constraint, using the trace property $\Tr(AB) = \Tr(BA)$,
\begin{align}
\Tr\left( V_t \Delta U^\top U_t \Delta V^\top \right) &= \Tr\left(\Delta V^\top V_t \Delta U^\top U_t \right) = \Tr\left( U_t^\top \Delta U \Delta U^\top U_t \right) = \|\Delta U^\top U_t\|_F^2 \geq 0 . \label{eq:temp_trace_inequality}
\end{align}
Combining \cref{eq:temp_trace_inequality} and the second part of \cref{prop:propB4_SL16} yields the bound
\begin{align*}
\|U_t \Delta V^\top + \Delta U V_t^\top\|_F^2 &= \|U_t \Delta V^\top\|_F^2 + \|\Delta U V_t^\top\|_F^2 + 2\Tr\left( V_t \Delta U^\top U_t \Delta V^\top \right) \\
&\geq \|U_t \Delta V^\top\|_F^2 + \|V_t \Delta U^\top\|_F^2
\geq \sigma^2_r(U_t) \|\Delta V\|_F^2 + \sigma^2_r(V_t) \|\Delta U\|_F^2 \\
&\geq \min\{\sigma^2_r(U_t),\, \sigma^2_r(V_t)\} \cdot \left(\|\Delta U\|_F^2 + \|\Delta V\|_F^2\right) .
\end{align*}
Plugging this bound into \cref{eq:sens_sigmaMin_tempBound} yields \cref{eq:sigmaMinL_sigmaMinsUV}.
\end{proof}

\begin{proof}[Proof of \cref{lem:minNormSol_implicitBalance}]
By the triangle inequality,
\begin{align*}
\|U_{t+1}^\top U_{t+1} - V_{t+1}^\top V_{t+1}\|_F &\leq \|U_t^\top U_{t} - V_{t}^\top V_t\|_F + 2\|U_t^\top \Delta U_{t} - \Delta V_{t}^\top V_t\|_F \\ &+ \|\Delta U_t^\top \Delta U_{t} - \Delta V_{t}^\top \Delta V_t\|_F.
\end{align*}
The second term on the RHS vanishes due to the first part of \cref{lem:kernel_minNormSol}.
The third term can be bounded by combining the triangle and the Cauchy-Schwarz inequalities as
\begin{align*}
\|\Delta U_t^\top \Delta U_{t} - \Delta V_{t}^\top \Delta V_t\|_F &\leq \|\Delta U_t^\top \Delta U_{t}\|_F + \|\Delta V_{t}^\top \Delta V_t\|_F
\leq \|\Delta U_t\|_F^2 + \|\Delta V_t\|_F^2 .
\end{align*}
The lemma thus follows by employing \cref{lem:minNormSol_normBound}.
\end{proof}

\section{Numerical results}\label{sec:numerics}
We illustrate the performance of different variants of \GNMR via several simulations.\footnote{Additional technical details on the experimental setups appear in \cref{sec:experimental_details}. Matlab and Python implementations of \GNMR for matrix completion and matrix sensing are available at \href{https://github.com/pizilber/GNMR}{github.com/pizilber/GNMR}.}
Each experiment consists of generating a random matrix $X^* \in \mathbb R^{n_1\times n_2}$ of a given rank $r$ and singular values $\sigma_i^*$, as well as a sampling pattern $\Omega \subseteq [n_1]\times [n_2]$ of a given size.
To generate $X^*$, we construct $U\in \mathbb R^{n_1\times r}$, $V\in \mathbb R^{n_2\times r}$ with entries i.i.d.~from the standard normal distribution, orthonormalize their columns, and set $X^* = U \Sigma V^\top$ where $\Sigma \in \mathbb R^{r\times r}$ is diagonal with entries $\sigma_i^*$.
Next, we generate $\Omega$ using the procedure from \cite{kummerle2020escaping}, which samples $\Omega$ randomly without replacement, and verifies that there are at least $r$ visible entries in each column and row of $X^*$.
Since a rank-$r$ matrix $X^* \in \mathbb R^{n_1\times n_2}$ has $(n_1+n_2-r)r$ degrees of freedom, we denote the oversampling ratio $\rho = \frac{|\Omega|}{(n_1+n_2-r)r}$. As $\rho$ decreases towards the information limit value of $1$, the harder the problem becomes.

In the experiments, we compare \GNMR, as sketched in \cref{alg:GNMR}, to the following algorithms:
\texttt{LRGeomCG} \cite{vandereycken2013low},
\texttt{RTRMC} \cite{boumal2015low},
\texttt{ScaledASD} \cite{tanner2016low}, \texttt{R2RILS} \cite{bauch2021rank}, and \texttt{MatrixIRLS} \cite{kummerle2020escaping}.
We used the Matlab implementations of these algorithms with default parameters as supplied by the respective authors, with the following exceptions: (i) Following \cite{kummerle2020escaping}, we set $\lambda = 10^{-8}$ in \texttt{RTRMC}, as it allows it to handle low oversampling ratios; (ii) In \texttt{MatrixIRLS}, the tol-CG-fac parameter was modified from its default value $10^{-5}$ to $10^{-5} \kappa^{-1}$ as in the experiments in \cite{kummerle2020escaping}, leading to improved results; (iii) For fair comparison, we unified the stopping criteria of all algorithms, as detailed in \cref{sec:experimental_details}.
Finally, all algorithms were initialized by the same spectral initialization, which is also their default initialization scheme. An exception is \texttt{MatrixIRLS} which is not factorization based.

Similar to previous works \cite{tanner2016low,bauch2021rank}, we use two quantitative measures to evaluate the success of the algorithms.
The first is the median of the relative RMSE, where the latter is defined as
\begin{align}\label{eq:relRMSE}
\texttt{rel-RMSE} = \frac{\|\hat X - X^*\|_F}{\|X^*\|_F} .
\end{align}
The second is the recovery probability, defined as $\text{Pr}[\texttt{rel-RMSE} \leq 10^{-4}]$. 

We compared the performance of the algorithms via three different experiments.
The goal of the first experiment, similar to \cite{bauch2021rank,kummerle2020escaping}, is to examine the ability to recover the underlying matrix under a constraint on the runtime or number of iterations.
Specifically, the maximal number of iterations was set such that the runtimes of all algorithms are bounded by approximately one minute (see \cref{sec:experimental_details} for more details).
The target matrix $X^*$ is of size $n_1\times n_2 = 1000\times 1000$, rank $r = 5$, and condition number $\kappa = 10$ with singular values equispaced between $1$ and $\kappa$.
The oversampling ratio $\rho$ covers the range $[1.35, 2.8]$.
The results, depicted in \cref{fig:highOversampling}, show a clear performance gap between different algorithms. In particular, as noted by \cite{kummerle2020escaping}, only methods that solve an inner problem at each iteration recover the matrix at low oversampling ratios. Specifically, the setting variant \cref{eq:settingVariant} of \GNMR shows favourable performance at low oversampling ratios compared to the other algorithms.

Interesting to note in \cref{fig:highOversampling} is the clear inferiority of the updating variant of \GNMR compared to the setting and the averaging ones. This phenomenon, which repeats itself in the next results, may be at least partially explained by our theoretical findings in \cref{sec:theory_stationryPoints}, which discriminate between these variants.

\begin{figure}[t]
	\centering
	\subfloat{
		\includegraphics[width=0.495\linewidth]{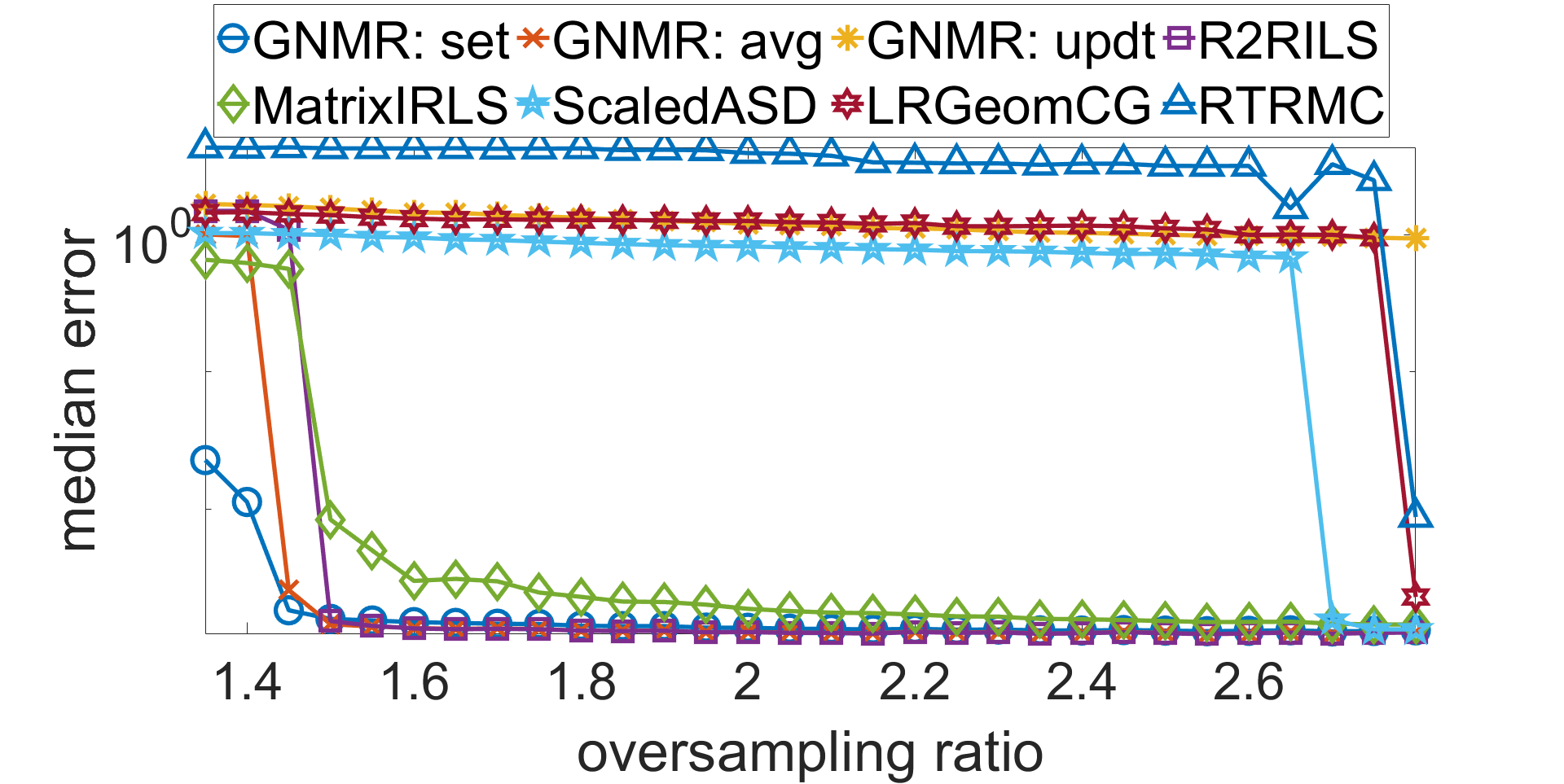}
		\label{fig:highOversampling_median}
		}
	\subfloat{
		\includegraphics[width=0.495\linewidth]{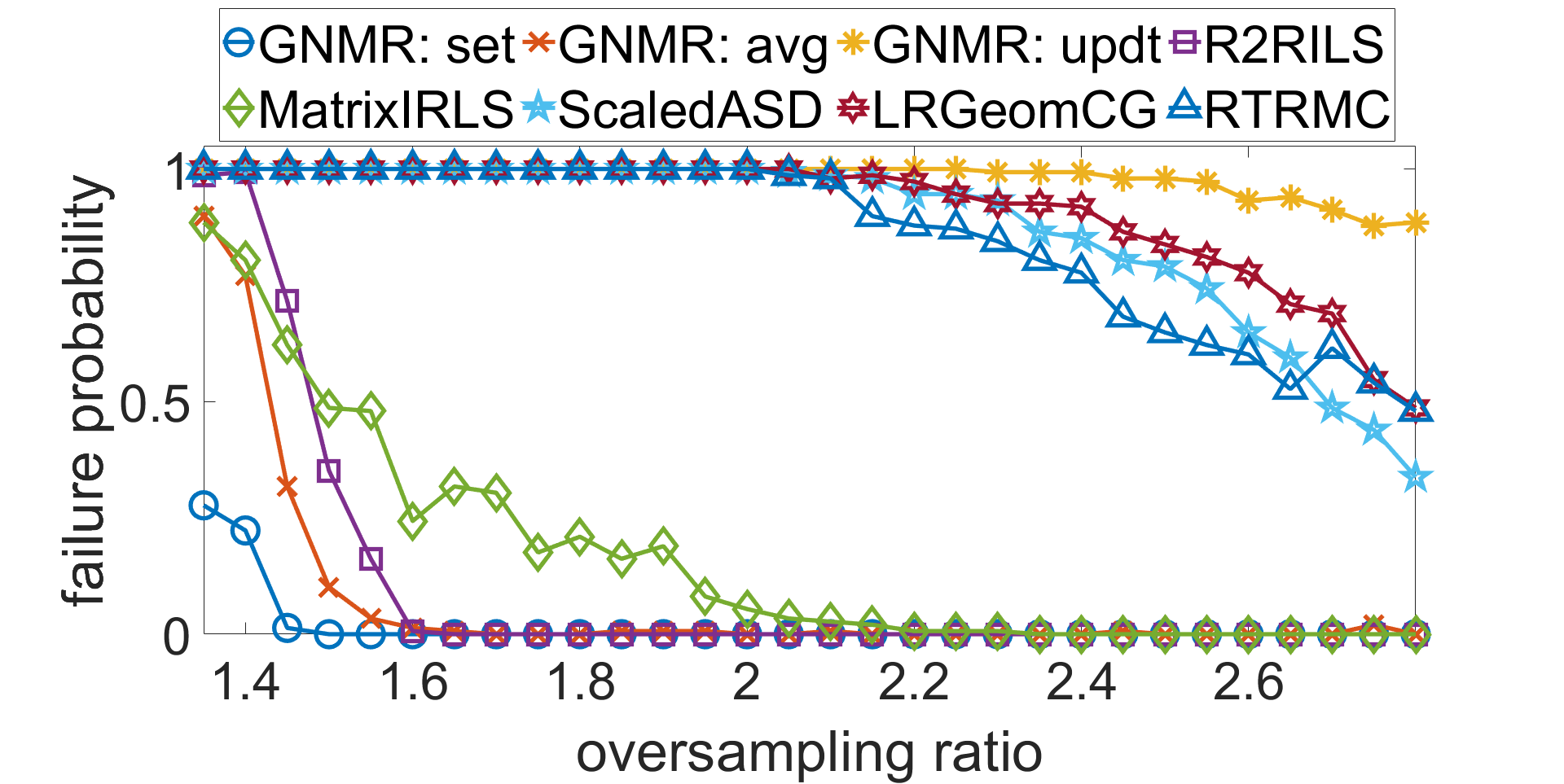}
		\label{fig:highOversampling_prob}
	}
	\caption{Comparison of several matrix completion algorithms with $X^*$ of size
		$1000\times 1000$, rank $r=5$ and condition number $\kappa=10$. Left panel: median of \texttt{rel-RMSE} \cref{eq:relRMSE}; Right panel: failure probability, defined as $\text{Pr}[\texttt{rel-RMSE} >10^{-4}]$. Each point corresponds to $150$ independent realizations.}
	\label{fig:highOversampling}
\end{figure}

The second experiment compares \texttt{R2RILS}, \texttt{MatrixIRLS} and \GNMR, which performed best in the first experiment, in a more challenging setting, where the number of observations is close to the information limit. Here $X^*$ is of size $n_1\times n_2 = 600\times 600$, rank $r = 7$, and condition number $\kappa = 100$ with singular values equispaced between $1$ and $\kappa$. The oversampling ratio $\rho$ ranges between $1.1$ and $1.5$.
Note that for $\rho$ close to one, even if $\Omega$ contains at least $r$ entries in each row and column, the solution to the completion problem may not be unique with a non-negligible probability.
Our goal in this experiment is to explore which of the algorithms can recover the matrix with essentially unlimited number of iterations. 
The results are depicted in \cref{fig:lowOversampling_median}.
Strengthening the conclusion from the previous experiment, the setting variant \cref{eq:settingVariant} of \GNMR outperforms the other algorithms, and succeeds in completing the matrix already at an oversampling ratio of $\rho = 1.1$.

\begin{figure}[t]
	\centering
	\subfloat[Challenging setting]{
		\includegraphics[width=0.495\linewidth]{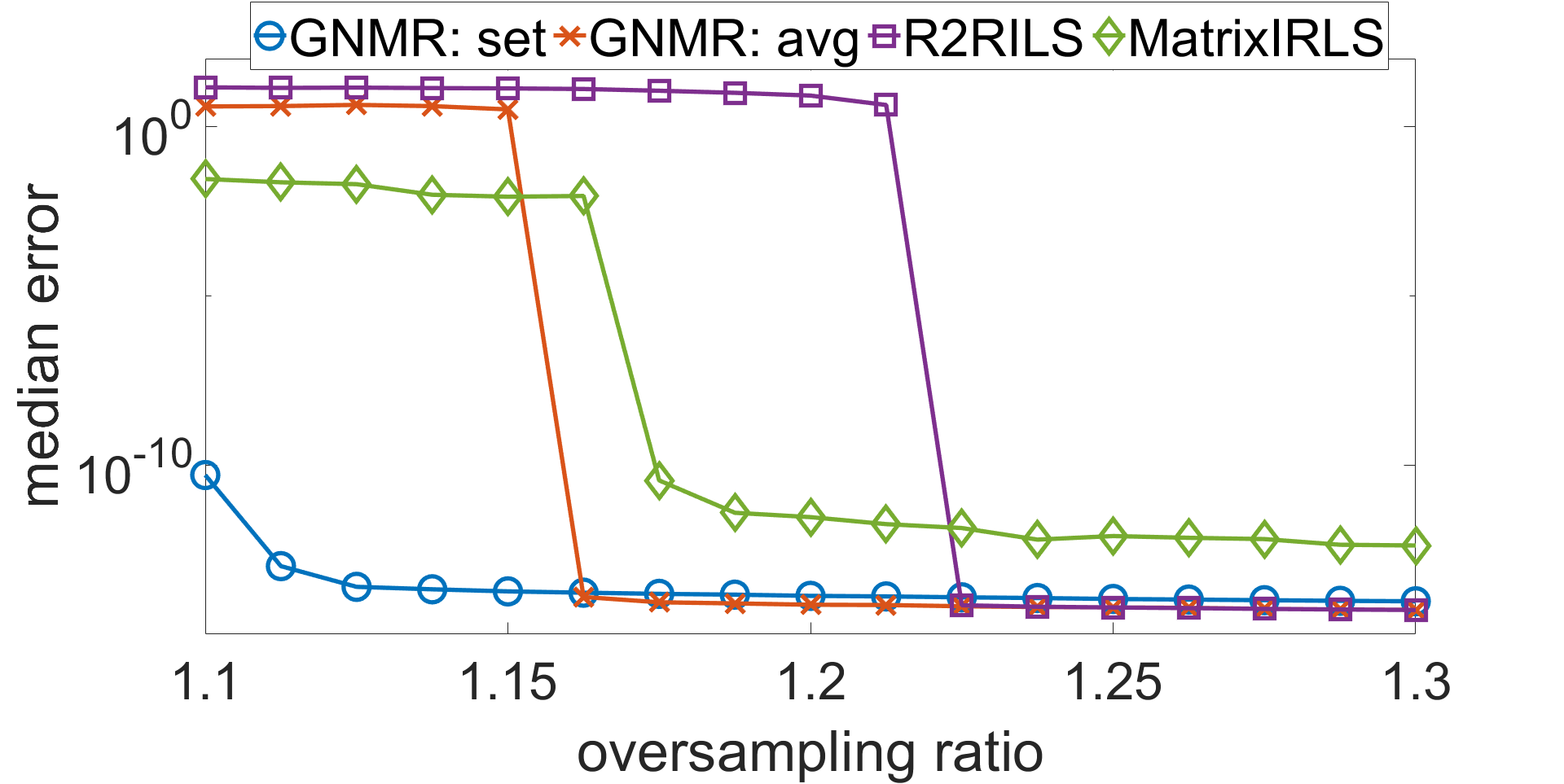}
		\label{fig:lowOversampling_median}
	}
	\subfloat[Increasing matrix dimension]{
		\includegraphics[width=0.495\linewidth]{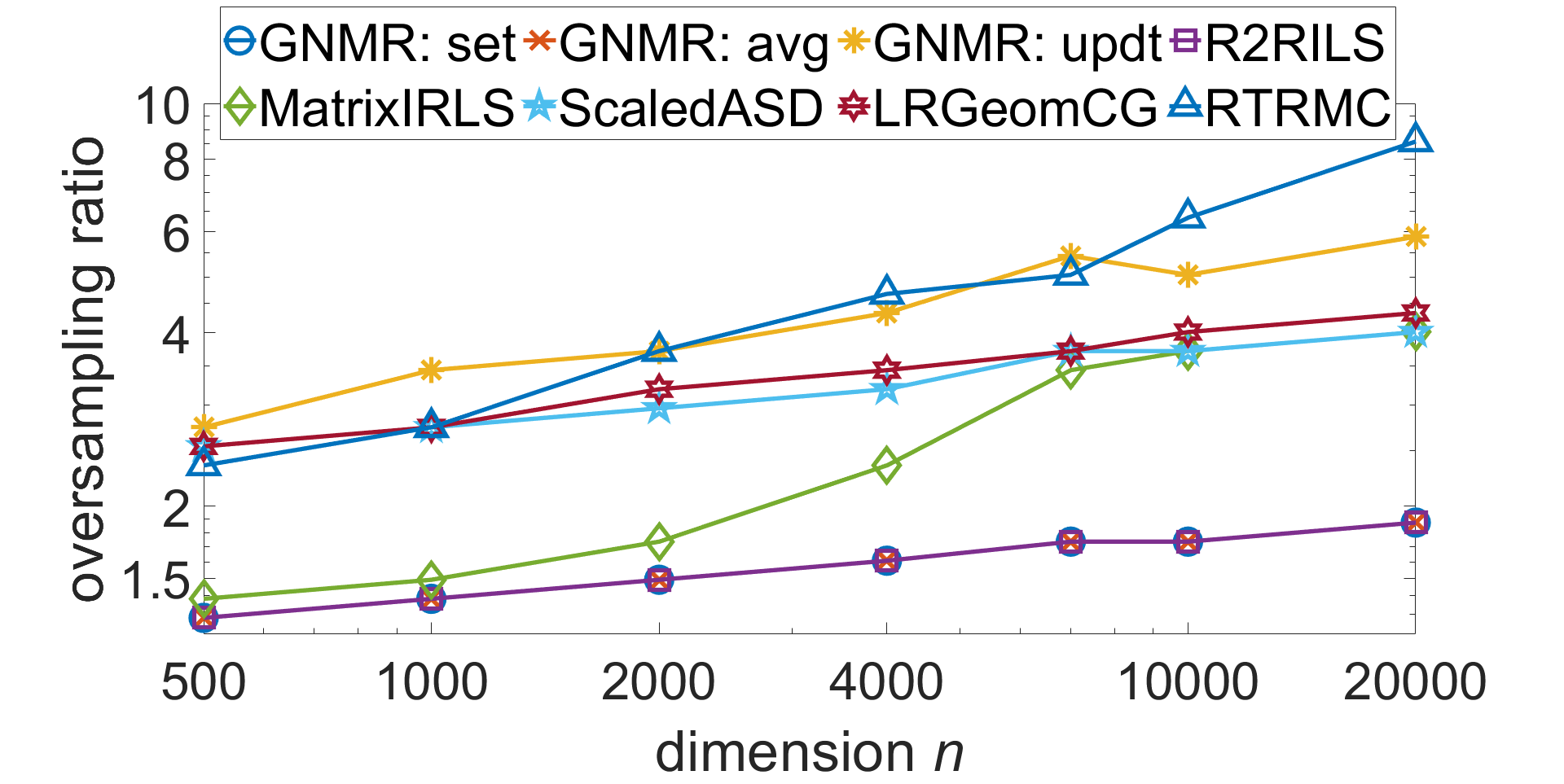}
		\label{fig:varyingDim_median}
	}
	\caption{(a) Comparison of the algorithms which succeeded in the previous experiment. Here $X^*$ is of size $600\times 600$, rank $r=7$ and condition number $\kappa=100$. Each point corresponds to 150 independent realizations.
	(b) Comparison of several matrix completion algorithms with $X^*$ of varying size
		$n\times n$, rank $r=5$ and condition number $\kappa=10$. Y-axis is the lowest oversampling ratio from which the median of \texttt{rel-RMSE} is smaller than $10^{-4}$. Each point corresponds to at least 50 independent realizations.}
	 \label{fig:lowOversampling_varyingDim_median}
\end{figure}

The previous two experiments demonstrated the recovery abilities of the algorithms for matrices of relatively small dimensions. In the third experiment, our goal is to examine how increasing the dimensions affects each of the algorithms.
Here $X^*$ is of varying size $n\times n$, fixed rank $r = 5$ and condition number $\kappa = 10$ with singular values equispaced between $1$ and $\kappa$. For each value of $n$, we report the lowest oversampling ratio from which the algorithm successfully recovers $X^*$, out of a grid of $30$ values logarithmically interpolated between $1.1$ and $4$. As seen in \cref{fig:varyingDim_median}, only \GNMR and \texttt{R2RILS} scale well with the dimension $n$. In fact, the results of these algorithms are 'optimal' in the following sense. As the dimension increases, higher oversampling ratios are required to ensure that a random subset $\Omega$ satisfies the necessary condition of $r$ observed entries in each row and column of $X^*$ with non-negligible probability. \GNMR and \texttt{R2RILS} successfully recovered $X^*$ from the lowest oversampling ratios at which this necessary condition held (see \cref{sec:experimental_details} for more details).

Next, we explore how sensitive \GNMR is to the condition number of $X^*$. As depicted in \cref{fig:OSR_vs_CN,fig:setavg_timeVsCN}, the setting and the averaging variants of \GNMR are generally robust to the condition number in two different aspects. First, the obtained error is almost unaffected by the condition number; there is only a little sensitivity to extreme condition numbers and only at very low oversampling ratios. Second, the runtime until $\texttt{rel-RMSE} \leq 10^{-4}$ shows only little sensitivity to the condition number. The updating variant of \GNMR, on the other hand, is sensitive to the condition number even at small values and at relatively large oversampling ratios.

\begin{figure}[t]
	\centering
	\subfloat{
		\includegraphics[width=0.495\linewidth]{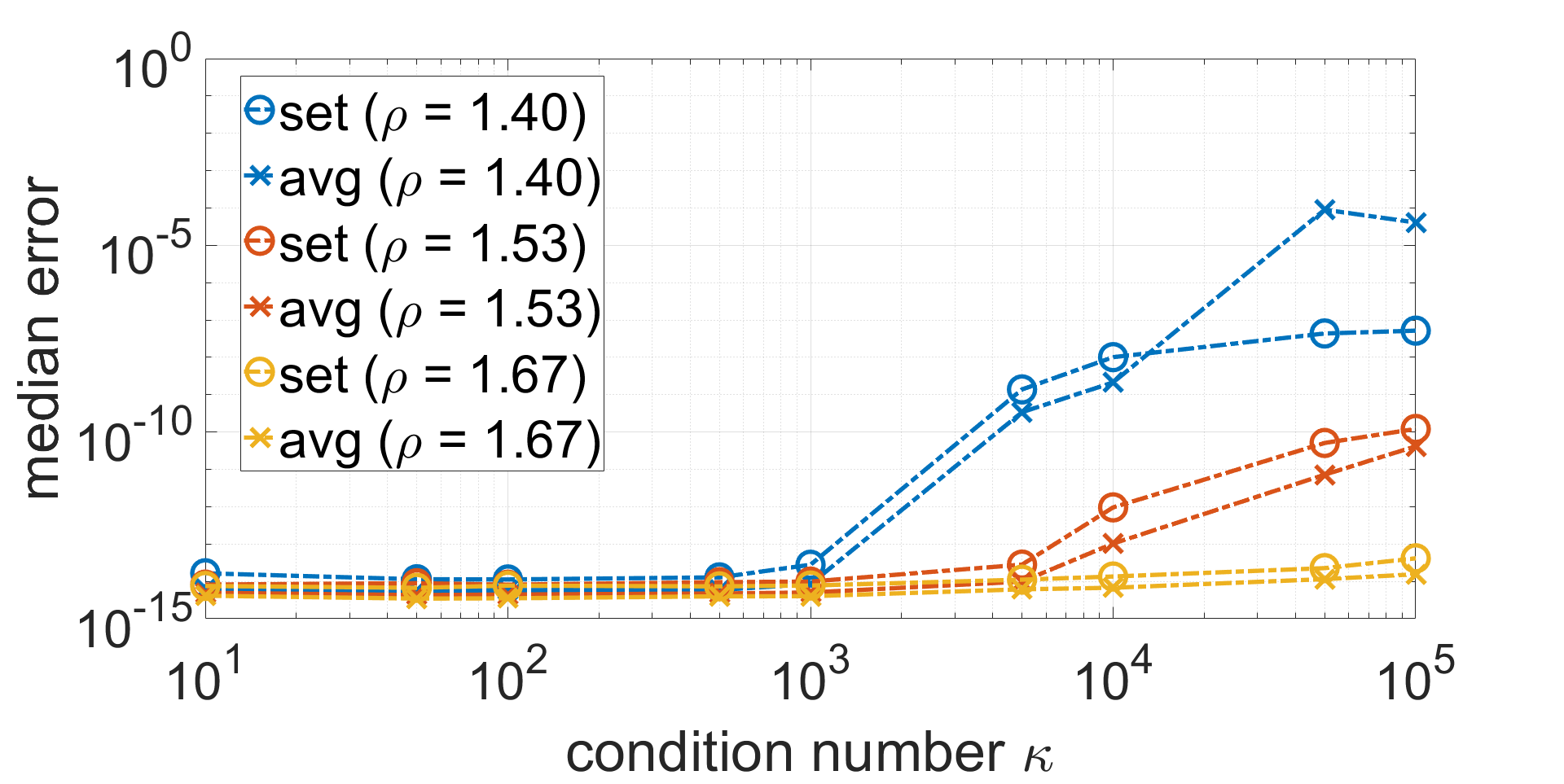}
	}
	\subfloat{
		\includegraphics[width=0.495\linewidth]{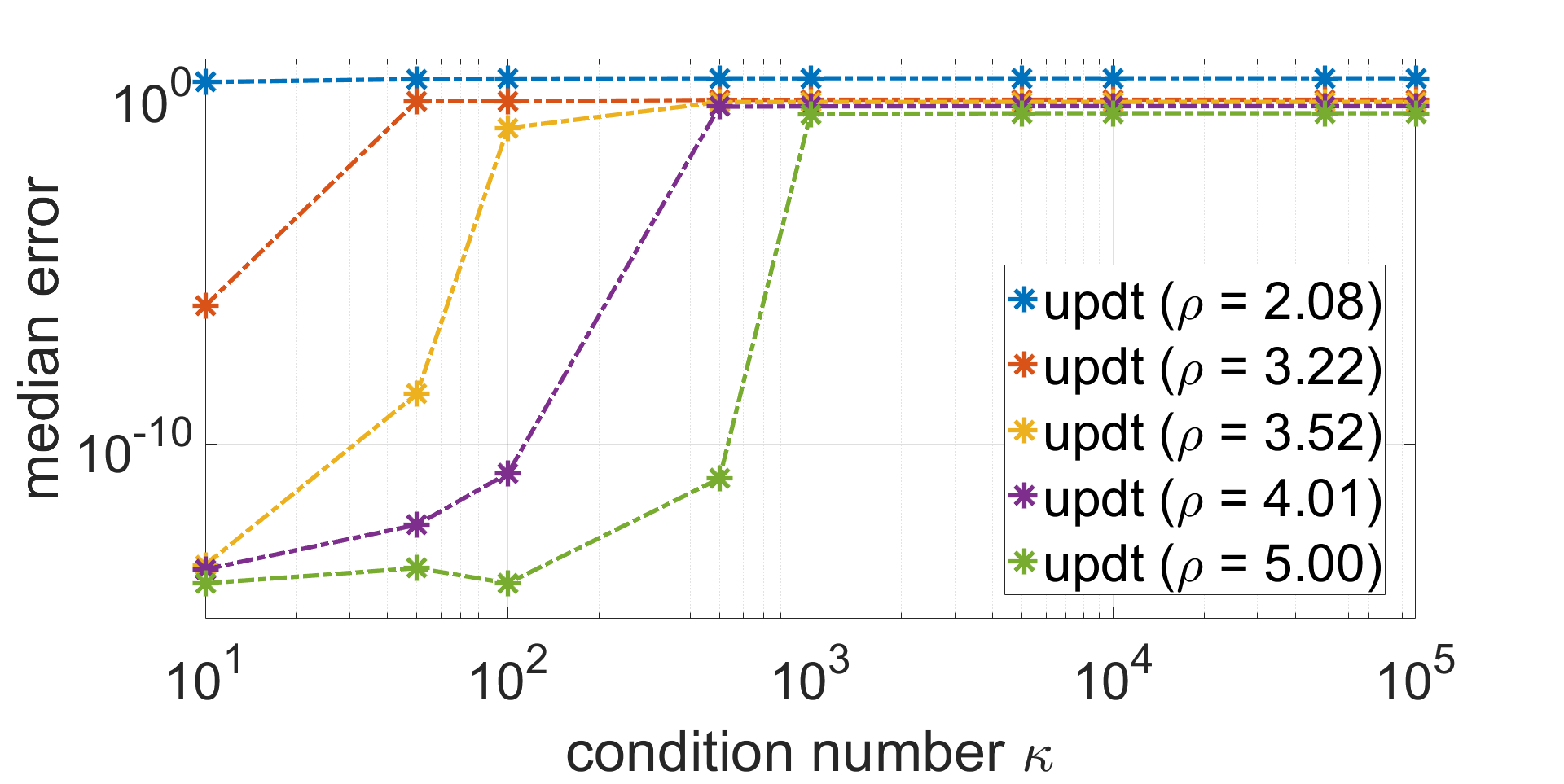}
	}
	\caption{Median \texttt{rel-RMSE} of \GNMR for a $1000\times 1000$ matrix $X^*$ of rank 5 as function of its condition number, for few values of the oversampling ratio $\rho$. Each point corresponds to 50 independent realizations.}
	\label{fig:OSR_vs_CN}
\end{figure}

Finally, \cref{fig:setavg_noise_median} illustrates the stability to noise of \GNMR. In this experiment, the observed entries are corrupted by additive white Gaussian noise of standard deviation $\sigma$. As seen in the figure, both the setting \cref{eq:settingVariant} and the averaging \cref{eq:averagingVariant} variants of \GNMR are robust to low noise levels, but interestingly, now it is the latter which performs better at higher noise levels.

\begin{figure}[t]
	\centering
	\subfloat[\GNMR runtime as function of $\kappa$]{
		\includegraphics[width=0.495\linewidth]{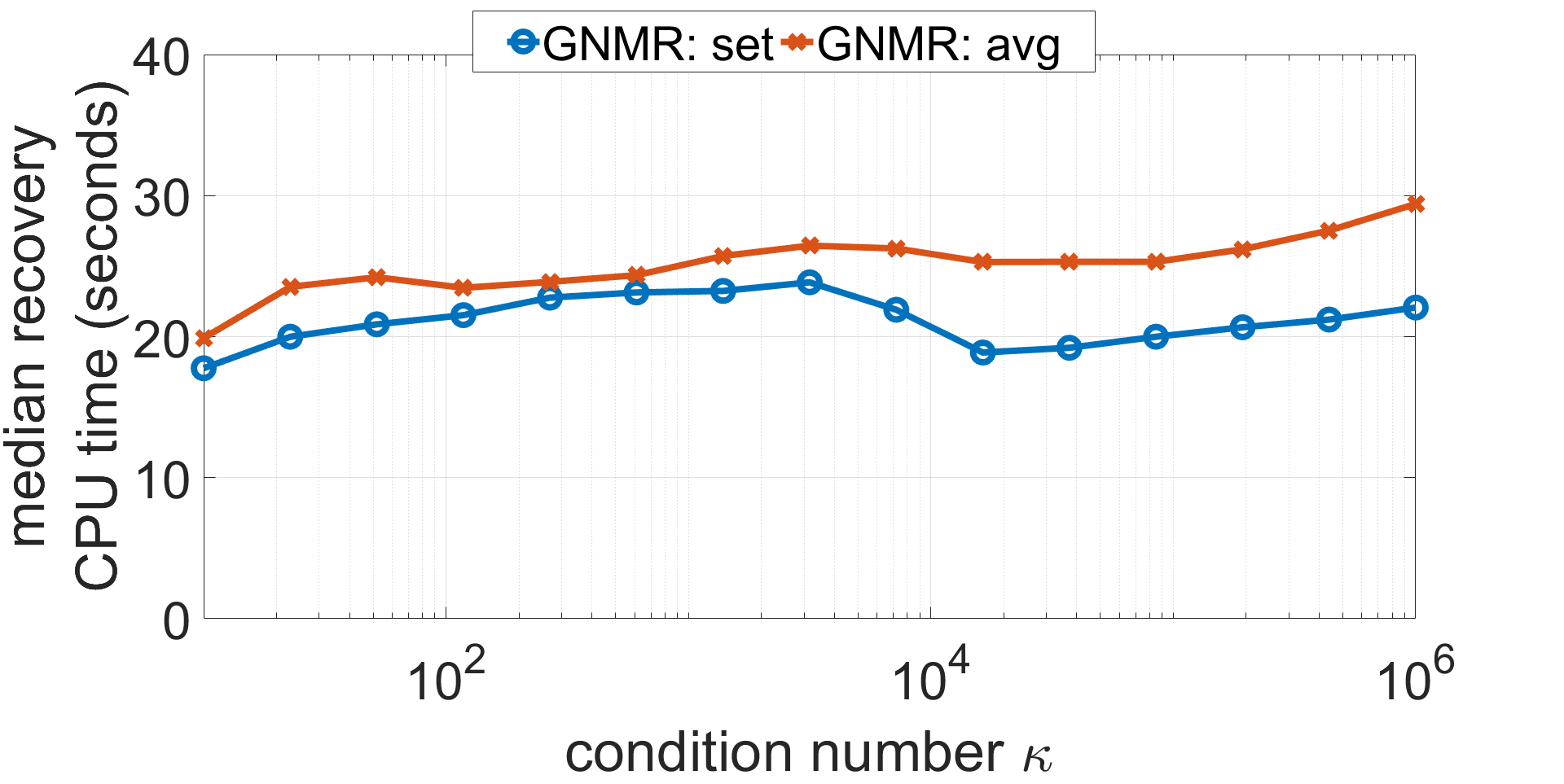}
		\label{fig:setavg_timeVsCN}
	}
	\subfloat[Noisy setting]{
		\includegraphics[width=0.495\linewidth]{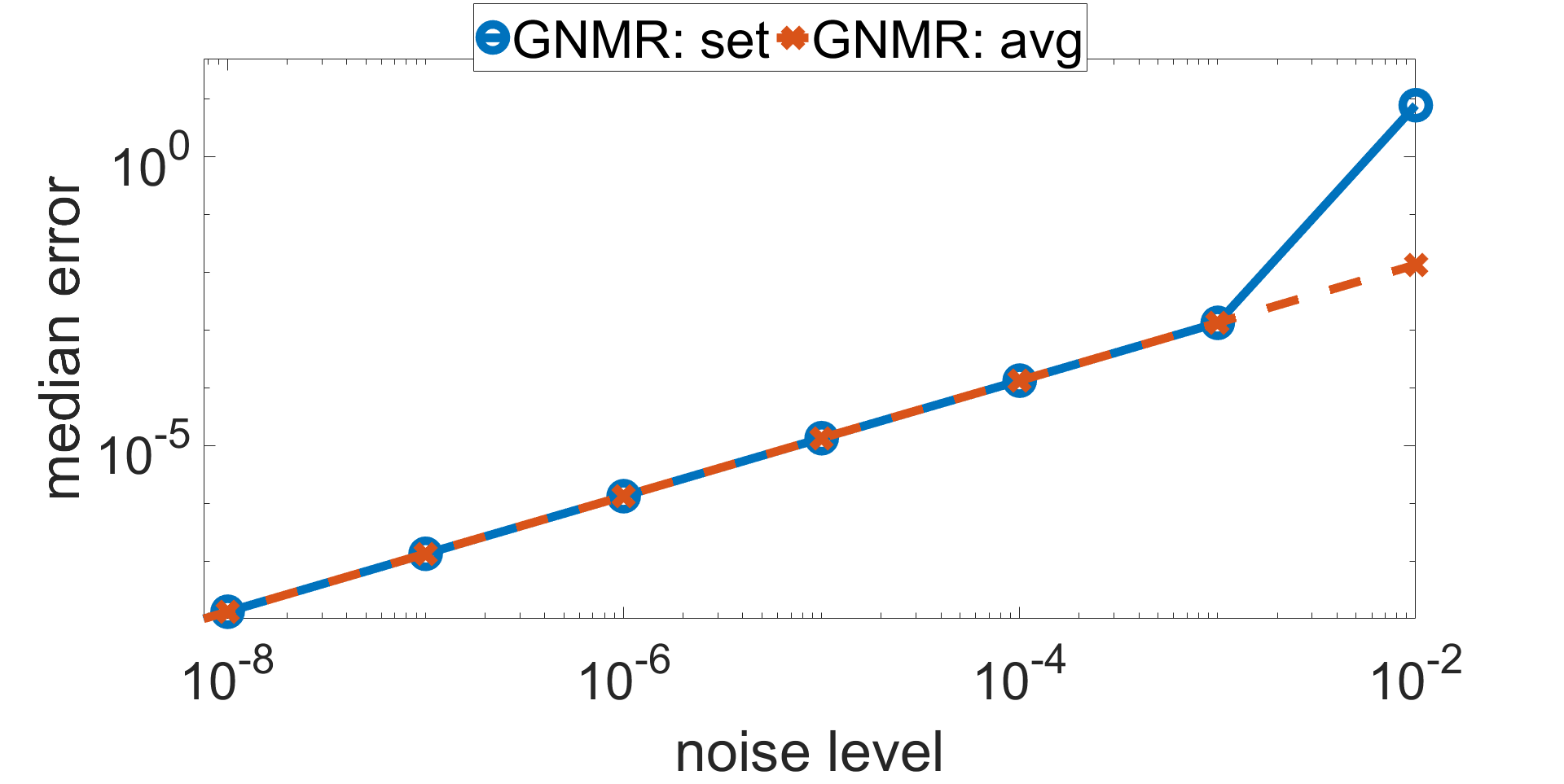}
		\label{fig:setavg_noise_median}
	}
	\caption{(a) Runtime of \GNMR as function of the condition number (with $\rho = 3$). (b) Stability of \GNMR to additive white Gaussian noise (with $\rho = 1.5$). The noise level on the x-axis corresponds to the standard deviation of the noise, and the y-axis to the median \texttt{rel-RMSE}. In both panels $X^*$ is a $1000\times 1000$ matrix of rank $5$, and each point corresponds to 150 independent realizations.}
	\label{fig:setavg_timeVsCN_noise}
\end{figure}


\section{Discussion and future work}
We proposed an extremely simple Gauss-Newton algorithm, \GNMR, to solve the matrix recovery problem \cref{eq:matrixRecovery_problem}.
We derived theoretical guarantees for our method, and demonstrated its state of the art empirical performance in matrix completion.
In our analysis, we showed that due to the choice of the minimal norm solution to a degenerate least squares problem, the iterates of \GNMR enjoy an implicit balance regularization.
%
Similarly, we proved that the stationary points of \GNMR are perfectly balanced.


The simplicity of \GNMR opens several future research directions. One is related to a current gap in the literature between available guarantees for factorization-based methods and their performance in practice: Nearly all available guarantees, including ours, scale at least quadratically with the condition number $\kappa$; our simulations, however, show that \GNMR is able to recover matrices with very little sensitivity to $\kappa$.
As far as we are aware, the only works with $\kappa$-independent (or logarithmically scaled) guarantees are \cite{hardt2014fast,cherapanamjeri2017nearly} and \cite{jain2015fast}. The latter is not factorization based, but its computational complexity is similar to factorization-based methods. However, these works employed sample splitting in their algorithm, which is never used in practice. This raises the question: is the quadratic dependence on $\kappa$ necessary for factorization-based methods that do not employ sample splitting? Our novel \GNMR method, which is both simple and empirically insensitive to $\kappa$, may help in providing a negative answer to this question, possibly via a leave-one-out perturbation analysis as in \cite{ma2019implicit}.

Another research direction is exploiting the simplicity of \GNMR to develop application-specific variants, which use additional prior knowledge or incorporate suitable regularizations. For example, in an ongoing work \cite{zilber2022inductive} we developed a variant of \GNMR for the inductive matrix completion problem \cite{jain2013provable,xu2013speedup}, in which one has prior knowledge that the rows and columns of $X^*$ belong to certain subspaces of $\mathbb R^{n_2}$ and $\mathbb R^{n_1}$, namely that $X^* = AM^*B^\top$ for some known matrices $A \in \mathbb R^{n_1\times d_1}$, $B\in \mathbb R^{n_2\times d_2}$.
Another example is an observed matrix corrupted by outliers, in which case the current form of \GNMR is unsuitable. However, an appealing property of our Gauss-Newton framework is that the inner problem solved in each iteration is convex for any convex loss function. Hence, a robust variant of \GNMR may be obtained by replacing the $\ell_2$ norm in \cref{alg:GNMR} by a robust one.
Finally, it may be beneficial to improve the runtime of \GNMR, so it will be able to handle large scale matrices.

\section*{Acknowledgments}
The research of P.Z.~was partially supported by a fellowship for data science from the Israeli Council for Higher Education (CHE).
B.N.~is the incumbent of the William Petschek professorial chair of mathematics.
We thank Yuval Kluger, Nati Srebro, Eric Chi, Yuejie Chi, Tian Tong, Tal Amir, Christian K\"{u}mmerle and Claudio Verdun for interesting discussions.

\bibliographystyle{alpha}
\bibliography{./GNMR}

\newcommand{\etalchar}[1]{$^{#1}$}
\begin{thebibliography}{YZMCS19}

\bibitem[ABG07]{absil2007trust}
P-A Absil, Christopher~G Baker, and Kyle~A Gallivan.
\newblock Trust-region methods on riemannian manifolds.
\newblock {\em Foundations of Computational Mathematics}, 7(3):303--330, 2007.

\bibitem[AKKS12]{avron2012efficient}
Haim Avron, Satyen Kale, Shiva~Prasad Kasiviswanathan, and Vikas Sindhwani.
\newblock Efficient and practical stochastic subgradient descent for nuclear
  norm regularization.
\newblock In {\em Proceedings of the 29th International Conference on Machine
  Learning}, pages 323--–330, Madison, WI, USA, 2012. Omnipress.

\bibitem[AMS09]{absil2009optimization}
P-A Absil, Robert Mahony, and Rodolphe Sepulchre.
\newblock {\em Optimization algorithms on matrix manifolds}.
\newblock Princeton University Press, 2009.

\bibitem[BA15]{boumal2015low}
Nicolas Boumal and P-A Absil.
\newblock Low-rank matrix completion via preconditioned optimization on the
  grassmann manifold.
\newblock {\em Linear Algebra and its Applications}, 475:200--239, 2015.

\bibitem[BF05]{buchanan2005damped}
Aeron~M Buchanan and Andrew~W Fitzgibbon.
\newblock Damped newton algorithms for matrix factorization with missing data.
\newblock In {\em Conference on Computer Vision and Pattern Recognition
  (CVPR)}, volume~2, pages 316--322. IEEE, 2005.

\bibitem[BNZ21]{bauch2021rank}
Jonathan Bauch, Boaz Nadler, and Pini Zilber.
\newblock Rank 2r iterative least squares: efficient recovery of
  ill-conditioned low rank matrices from few entries.
\newblock {\em SIAM Journal on Mathematics of Data Science}, 3(1):439--465,
  2021.

\bibitem[BTW15]{blanchard2015cgiht}
Jeffrey~D Blanchard, Jared Tanner, and Ke~Wei.
\newblock {CGIHT}: conjugate gradient iterative hard thresholding for
  compressed sensing and matrix completion.
\newblock {\em Information and Inference: A Journal of the IMA}, 4(4):289--327,
  2015.

\bibitem[Can08]{candes2008restricted}
Emmanuel~J Candes.
\newblock The restricted isometry property and its implications for compressed
  sensing.
\newblock {\em Comptes rendus mathematique}, 346(9-10):589--592, 2008.

\bibitem[CBSW15]{chen2015completing}
Yudong Chen, Srinadh Bhojanapalli, Sujay Sanghavi, and Rachel Ward.
\newblock Completing any low-rank matrix, provably.
\newblock {\em The Journal of Machine Learning Research}, 16(1):2999--3034,
  2015.

\bibitem[CCD{\etalchar{+}}21]{charisopoulos2021low}
Vasileios Charisopoulos, Yudong Chen, Damek Davis, Mateo D{\'\i}az, Lijun Ding,
  and Dmitriy Drusvyatskiy.
\newblock Low-rank matrix recovery with composite optimization: good
  conditioning and rapid convergence.
\newblock {\em Foundations of Computational Mathematics}, pages 1--89, 2021.

\bibitem[CCF{\etalchar{+}}20]{chen2020noisy}
Yuxin Chen, Yuejie Chi, Jianqing Fan, Cong Ma, and Yuling Yan.
\newblock Noisy matrix completion: Understanding statistical guarantees for
  convex relaxation via nonconvex optimization.
\newblock {\em SIAM Journal on Optimization}, 30(4):3098--3121, 2020.

\bibitem[CCS10]{cai2010singular}
Jian-Feng Cai, Emmanuel~J Cand{\`e}s, and Zuowei Shen.
\newblock A singular value thresholding algorithm for matrix completion.
\newblock {\em SIAM Journal on Optimization}, 20(4):1956--1982, 2010.

\bibitem[CFMY21]{chen2021bridging}
Yuxin Chen, Jianqing Fan, Cong Ma, and Yuling Yan.
\newblock Bridging convex and nonconvex optimization in robust pca: Noise,
  outliers and missing data.
\newblock {\em The Annals of Statistics}, 49(5):2948--2971, 2021.

\bibitem[CGJ17]{cherapanamjeri2017nearly}
Yeshwanth Cherapanamjeri, Kartik Gupta, and Prateek Jain.
\newblock Nearly optimal robust matrix completion.
\newblock In {\em International Conference on Machine Learning}, pages
  797--805. PMLR, 2017.

\bibitem[Che15]{chen2015incoherence}
Yudong Chen.
\newblock Incoherence-optimal matrix completion.
\newblock {\em IEEE Transactions on Information Theory}, 61(5):2909--2923,
  2015.

\bibitem[CL19]{chi2019matrix}
Eric~C Chi and Tianxi Li.
\newblock Matrix completion from a computational statistics perspective.
\newblock {\em Wiley Interdisciplinary Reviews: Computational Statistics},
  11(5):e1469, 2019.

\bibitem[CLC19]{chi2019nonconvex}
Yuejie Chi, Yue~M Lu, and Yuxin Chen.
\newblock Nonconvex optimization meets low-rank matrix factorization: An
  overview.
\newblock {\em IEEE Transactions on Signal Processing}, 67(20):5239--5269,
  2019.

\bibitem[CLL20]{chen2020nonconvex}
Ji~Chen, Dekai Liu, and Xiaodong Li.
\newblock Nonconvex rectangular matrix completion via gradient descent without
  $\ell_{2\infty}$ regularization.
\newblock {\em IEEE Transactions on Information Theory}, 66(9):5806--5841,
  2020.

\bibitem[CP10]{candes2010matrix}
Emmanuel~J Candes and Yaniv Plan.
\newblock Matrix completion with noise.
\newblock {\em Proceedings of the IEEE}, 98(6):925--936, 2010.

\bibitem[CR09]{candes2009exact}
Emmanuel~J Cand{\`e}s and Benjamin Recht.
\newblock Exact matrix completion via convex optimization.
\newblock {\em Foundations of Computational mathematics}, 9(6):717, 2009.

\bibitem[CT10]{candes2010power}
Emmanuel~J Cand{\`e}s and Terence Tao.
\newblock The power of convex relaxation: Near-optimal matrix completion.
\newblock {\em IEEE Transactions on Information Theory}, 56(5):2053--2080,
  2010.

\bibitem[DC20]{ding2020leave}
Lijun Ding and Yudong Chen.
\newblock Leave-one-out approach for matrix completion: Primal and dual
  analysis.
\newblock {\em IEEE Transactions on Information Theory}, 66(11):7274--7301,
  2020.

\bibitem[DR16]{davenport2016overview}
Mark~A Davenport and Justin Romberg.
\newblock An overview of low-rank matrix recovery from incomplete observations.
\newblock {\em IEEE Journal of Selected Topics in Signal Processing},
  10(4):608--622, 2016.

\bibitem[FHB{\etalchar{+}}01]{fazel2001rank}
Maryam Fazel, Haitham Hindi, Stephen~P Boyd, et~al.
\newblock A rank minimization heuristic with application to minimum order
  system approximation.
\newblock {\em Proceedings of the American control conference}, 6:4734--4739,
  2001.

\bibitem[FO05]{feige2005spectral}
Uriel Feige and Eran Ofek.
\newblock Spectral techniques applied to sparse random graphs.
\newblock {\em Random Structures \& Algorithms}, 27(2):251--275, 2005.

\bibitem[FRW11]{fornasier2011low}
Massimo Fornasier, Holger Rauhut, and Rachel Ward.
\newblock Low-rank matrix recovery via iteratively reweighted least squares
  minimization.
\newblock {\em SIAM Journal on Optimization}, 21(4):1614--1640, 2011.

\bibitem[GJZ17]{ge2017no}
Rong Ge, Chi Jin, and Yi~Zheng.
\newblock No spurious local minima in nonconvex low rank problems: A unified
  geometric analysis.
\newblock In {\em International Conference on Machine Learning}, pages
  1233--1242. PMLR, 2017.

\bibitem[GLM16]{ge2016matrix}
Rong Ge, Jason~D Lee, and Tengyu Ma.
\newblock Matrix completion has no spurious local minimum.
\newblock In {\em Advances in Neural Information Processing Systems}, pages
  2973--2981, 2016.

\bibitem[GP03]{golub2003separable}
Gene Golub and Victor Pereyra.
\newblock Separable nonlinear least squares: the variable projection method and
  its applications.
\newblock {\em Inverse problems}, 19(2):R1, 2003.

\bibitem[GP10]{guillemin2010differential}
Victor Guillemin and Alan Pollack.
\newblock {\em Differential topology}, volume 370.
\newblock American Mathematical Soc., 2010.

\bibitem[Gro11]{gross2011recovering}
David Gross.
\newblock Recovering low-rank matrices from few coefficients in any basis.
\newblock {\em IEEE Transactions on Information Theory}, 57(3):1548--1566,
  2011.

\bibitem[Har14]{hardt2014understanding}
Moritz Hardt.
\newblock Understanding alternating minimization for matrix completion.
\newblock In {\em 2014 IEEE 55th Annual Symposium on Foundations of Computer
  Science}, pages 651--660. IEEE, 2014.

\bibitem[HH09]{haldar2009rank}
Justin~P Haldar and Diego Hernando.
\newblock Rank-constrained solutions to linear matrix equations using
  powerfactorization.
\newblock {\em IEEE Signal Processing Letters}, 16(7):584--587, 2009.

\bibitem[HW14]{hardt2014fast}
Moritz Hardt and Mary Wootters.
\newblock Fast matrix completion without the condition number.
\newblock In {\em Conference on learning theory}, pages 638--678. PMLR, 2014.

\bibitem[JD13]{jain2013provable}
Prateek Jain and Inderjit~S Dhillon.
\newblock Provable inductive matrix completion.
\newblock {\em arXiv preprint arXiv:1306.0626}, 2013.

\bibitem[JMD10]{jain2010guaranteed}
Prateek Jain, Raghu Meka, and Inderjit Dhillon.
\newblock Guaranteed rank minimization via singular value projection.
\newblock In {\em Proceedings of the 23rd International Conference on Neural
  Information Processing Systems-Volume 1}, pages 937--945, 2010.

\bibitem[JN15]{jain2015fast}
Prateek Jain and Praneeth Netrapalli.
\newblock Fast exact matrix completion with finite samples.
\newblock In {\em Conference on Learning Theory}, pages 1007--1034, 2015.

\bibitem[JNS13]{jain2013low}
Prateek Jain, Praneeth Netrapalli, and Sujay Sanghavi.
\newblock Low-rank matrix completion using alternating minimization.
\newblock In {\em Proceedings of the forty-fifth annual ACM symposium on Theory
  of computing}, pages 665--674. ACM, 2013.

\bibitem[JY09]{ji2009accelerated}
Shuiwang Ji and Jieping Ye.
\newblock An accelerated gradient method for trace norm minimization.
\newblock In {\em Proceedings of the 26th annual international conference on
  machine learning}, pages 457--464. ACM, 2009.

\bibitem[KC14]{kyrillidis2014matrix}
Anastasios Kyrillidis and Volkan Cevher.
\newblock Matrix recipes for hard thresholding methods.
\newblock {\em Journal of mathematical imaging and vision}, 48(2):235--265,
  2014.

\bibitem[Kes12]{keshavan2012efficient}
Raghunandan~Hulikal Keshavan.
\newblock {\em Efficient algorithms for collaborative filtering}.
\newblock Stanford University, 2012.

\bibitem[KMO10]{keshavan2010matrix}
Raghunandan~H Keshavan, Andrea Montanari, and Sewoong Oh.
\newblock Matrix completion from a few entries.
\newblock {\em IEEE transactions on Information Theory}, 56(6):2980--2998,
  2010.

\bibitem[KS18]{kummerle2018harmonic}
Christian K{\"u}mmerle and Juliane Sigl.
\newblock Harmonic mean iteratively reweighted least squares for low-rank
  matrix recovery.
\newblock {\em The Journal of Machine Learning Research}, 19(1):1815--1863,
  2018.

\bibitem[KV20]{kummerle2020escaping}
Christian K{\"u}mmerle and Claudio~M Verdun.
\newblock Escaping saddle points in ill-conditioned matrix completion with a
  scalable second order method.
\newblock In {\em Workshop on Beyond First Order Methods in ML Systems at the
  $37^{th}$ International Conference on Machine Learning}, 2020.

\bibitem[KV21]{kummerle2021scalable}
Christian K{\"u}mmerle and Claudio~M Verdun.
\newblock A scalable second order method for ill-conditioned matrix completion
  from few samples.
\newblock In {\em International Conference on Machine Learning (ICML)}, 2021.

\bibitem[LCZL20]{li2020non}
Yuanxin Li, Yuejie Chi, Huishuai Zhang, and Yingbin Liang.
\newblock Non-convex low-rank matrix recovery with arbitrary outliers via
  median-truncated gradient descent.
\newblock {\em Information and Inference: A Journal of the IMA}, 9(2):289--325,
  2020.

\bibitem[LHLZ20]{luo2020recursive}
Yuetian Luo, Wen Huang, Xudong Li, and Anru~R Zhang.
\newblock Recursive importance sketching for rank constrained least squares:
  Algorithms and high-order convergence.
\newblock {\em arXiv preprint arXiv:2011.08360}, 2020.

\bibitem[LLA{\etalchar{+}}19]{li2019symmetry}
Xingguo Li, Junwei Lu, Raman Arora, Jarvis Haupt, Han Liu, Zhaoran Wang, and
  Tuo Zhao.
\newblock Symmetry, saddle points, and global optimization landscape of
  nonconvex matrix factorization.
\newblock {\em IEEE Transactions on Information Theory}, 65(6):3489--3514,
  2019.

\bibitem[LLZ{\etalchar{+}}20]{li2020global}
Shuang Li, Qiuwei Li, Zhihui Zhu, Gongguo Tang, and Michael~B Wakin.
\newblock The global geometry of centralized and distributed low-rank matrix
  recovery without regularization.
\newblock {\em IEEE Signal Processing Letters}, 27:1400--1404, 2020.

\bibitem[MGC11]{ma2011fixed}
Shiqian Ma, Donald Goldfarb, and Lifeng Chen.
\newblock Fixed point and {B}regman iterative methods for matrix rank
  minimization.
\newblock {\em Mathematical Programming}, 128(1-2):321--353, 2011.

\bibitem[MHT10]{mazumder2010spectral}
Rahul Mazumder, Trevor Hastie, and Robert Tibshirani.
\newblock Spectral regularization algorithms for learning large incomplete
  matrices.
\newblock {\em Journal of machine learning research}, 11(Aug):2287--2322, 2010.

\bibitem[MLC21]{ma2021beyond}
Cong Ma, Yuanxin Li, and Yuejie Chi.
\newblock Beyond procrustes: Balancing-free gradient descent for asymmetric
  low-rank matrix sensing.
\newblock {\em IEEE Transactions on Signal Processing}, 69:867--877, 2021.

\bibitem[MMBS13]{mishra2013low}
Bamdev Mishra, Gilles Meyer, Francis Bach, and Rodolphe Sepulchre.
\newblock Low-rank optimization with trace norm penalty.
\newblock {\em SIAM Journal on Optimization}, 23(4):2124--2149, 2013.

\bibitem[MMBS14]{mishra2014fixed}
Bamdev Mishra, Gilles Meyer, Silv{\`e}re Bonnabel, and Rodolphe Sepulchre.
\newblock {Fixed-rank matrix factorizations and Riemannian low-rank
  optimization}.
\newblock {\em Computational Statistics}, 29(3-4):591--621, 2014.

\bibitem[MS12]{marjanovic2012l_q}
Goran Marjanovic and Victor Solo.
\newblock On $\ell_q$ optimization and matrix completion.
\newblock {\em IEEE Transactions on signal processing}, 60(11):5714--5724,
  2012.

\bibitem[MS14]{mishra2014r3mc}
Bamdev Mishra and Rodolphe Sepulchre.
\newblock {R3MC: A Riemannian three-factor algorithm for low-rank matrix
  completion}.
\newblock In {\em 53rd IEEE Conference on Decision and Control}, pages
  1137--1142. IEEE, 2014.

\bibitem[MWCC19]{ma2019implicit}
Cong Ma, Kaizheng Wang, Yuejie Chi, and Yuxin Chen.
\newblock Implicit regularization in nonconvex statistical estimation: Gradient
  descent converges linearly for phase retrieval, matrix completion, and blind
  deconvolution.
\newblock {\em Foundations of Computational Mathematics}, 2019.

\bibitem[NS12]{ngo2012scaled}
Thanh Ngo and Yousef Saad.
\newblock Scaled gradients on grassmann manifolds for matrix completion.
\newblock In {\em Advances in Neural Information Processing Systems}, pages
  1412--1420, 2012.

\bibitem[OD07]{okatani2007wiberg}
Takayuki Okatani and Koichiro Deguchi.
\newblock On the wiberg algorithm for matrix factorization in the presence of
  missing components.
\newblock {\em International Journal of Computer Vision}, 72(3):329--337, 2007.

\bibitem[OYD11]{okatani2011efficient}
Takayuki Okatani, Takahiro Yoshida, and Koichiro Deguchi.
\newblock Efficient algorithm for low-rank matrix factorization with missing
  components and performance comparison of latest algorithms.
\newblock In {\em International Conference on Computer Vision}, pages 842--849.
  IEEE, 2011.

\bibitem[PABN16]{pimentel2016characterization}
Daniel~L Pimentel-Alarc{\'o}n, Nigel Boston, and Robert~D Nowak.
\newblock A characterization of deterministic sampling patterns for low-rank
  matrix completion.
\newblock {\em IEEE Journal of Selected Topics in Signal Processing},
  10(4):623--636, 2016.

\bibitem[PKCS18]{park2018finding}
Dohyung Park, Anastasios Kyrillidis, Constantine Caramanis, and Sujay Sanghavi.
\newblock Finding low-rank solutions via nonconvex matrix factorization,
  efficiently and provably.
\newblock {\em SIAM Journal on Imaging Sciences}, 11(4):2165--2204, 2018.

\bibitem[PS82]{paige1982lsqr}
Christopher~C Paige and Michael~A Saunders.
\newblock {LSQR: An algorithm for sparse linear equations and sparse least
  squares}.
\newblock {\em ACM Transactions on Mathematical Software (TOMS)}, 8(1):43--71,
  1982.

\bibitem[PT94]{paatero1994positive}
Pentti Paatero and Unto Tapper.
\newblock Positive matrix factorization: A non-negative factor model with
  optimal utilization of error estimates of data values.
\newblock {\em Environmetrics}, 5(2):111--126, 1994.

\bibitem[Rec11]{recht2011simpler}
Benjamin Recht.
\newblock A simpler approach to matrix completion.
\newblock {\em Journal of Machine Learning Research}, 12(Dec):3413--3430, 2011.

\bibitem[RFP10]{recht2010guaranteed}
Benjamin Recht, Maryam Fazel, and Pablo~A Parrilo.
\newblock Guaranteed minimum-rank solutions of linear matrix equations via
  nuclear norm minimization.
\newblock {\em SIAM review}, 52(3):471--501, 2010.

\bibitem[RS05]{rennie2005fast}
Jasson~DM Rennie and Nathan Srebro.
\newblock Fast maximum margin matrix factorization for collaborative
  prediction.
\newblock In {\em Proceedings of the 22nd international conference on Machine
  learning}, pages 713--719. ACM, 2005.

\bibitem[RW80]{ruhe1980algorithms}
Axel Ruhe and Per~{\AA}ke Wedin.
\newblock Algorithms for separable nonlinear least squares problems.
\newblock {\em SIAM review}, 22(3):318--337, 1980.

\bibitem[SC10]{singer2010uniqueness}
Amit Singer and Mihai Cucuringu.
\newblock Uniqueness of low-rank matrix completion by rigidity theory.
\newblock {\em SIAM Journal on Matrix Analysis and Applications},
  31(4):1621--1641, 2010.

\bibitem[SL16]{sun2016guaranteed}
Ruoyu Sun and Zhi-Quan Luo.
\newblock Guaranteed matrix completion via non-convex factorization.
\newblock {\em IEEE Transactions on Information Theory}, 62(11):6535--6579,
  2016.

\bibitem[TBS{\etalchar{+}}16]{tu2016low}
Stephen Tu, Ross Boczar, Max Simchowitz, Mahdi Soltanolkotabi, and Ben Recht.
\newblock Low-rank solutions of linear matrix equations via procrustes flow.
\newblock In {\em International Conference on Machine Learning}, pages
  964--973. PMLR, 2016.

\bibitem[TMC21a]{tong2021accelerating}
Tian Tong, Cong Ma, and Yuejie Chi.
\newblock Accelerating ill-conditioned low-rank matrix estimation via scaled
  gradient descent.
\newblock {\em Journal of Machine Learning Research}, 22(150):1--63, 2021.

\bibitem[TMC21b]{tong2021low}
Tian Tong, Cong Ma, and Yuejie Chi.
\newblock Low-rank matrix recovery with scaled subgradient methods: Fast and
  robust convergence without the condition number.
\newblock {\em IEEE Transactions on Signal Processing}, 69:2396--2409, 2021.

\bibitem[TW13]{tanner2013normalized}
Jared Tanner and Ke~Wei.
\newblock Normalized iterative hard thresholding for matrix completion.
\newblock {\em SIAM Journal on Scientific Computing}, 35(5):S104--S125, 2013.

\bibitem[TW16]{tanner2016low}
Jared Tanner and Ke~Wei.
\newblock Low rank matrix completion by alternating steepest descent methods.
\newblock {\em Applied and Computational Harmonic Analysis}, 40(2):417--429,
  2016.

\bibitem[TY10]{toh2010accelerated}
Kim-Chuan Toh and Sangwoon Yun.
\newblock An accelerated proximal gradient algorithm for nuclear norm
  regularized linear least squares problems.
\newblock {\em Pacific Journal of optimization}, 6(615-640):15, 2010.

\bibitem[Van13]{vandereycken2013low}
Bart Vandereycken.
\newblock Low-rank matrix completion by {R}iemannian optimization.
\newblock {\em SIAM Journal on Optimization}, 23(2):1214--1236, 2013.

\bibitem[Wah65]{wahba1965least}
Grace Wahba.
\newblock A least squares estimate of satellite attitude.
\newblock {\em SIAM review}, 7(3):409--409, 1965.

\bibitem[WCCL16]{wei2016guarantees}
Ke~Wei, Jian-Feng Cai, Tony~F Chan, and Shingyu Leung.
\newblock {Guarantees of Riemannian optimization for low rank matrix recovery}.
\newblock {\em SIAM Journal on Matrix Analysis and Applications},
  37(3):1198--1222, 2016.

\bibitem[WCZT21]{wang2021large}
Yuqing Wang, Minshuo Chen, Tuo Zhao, and Molei Tao.
\newblock Large learning rate tames homogeneity: Convergence and balancing
  effect.
\newblock {\em arXiv preprint arXiv:2110.03677}, 2021.

\bibitem[Wib76]{wiberg1976}
T~Wiberg.
\newblock Computation of principal components when data are missing.
\newblock In {\em Proc. Second Symp. Computational Statistics}, pages 229--236,
  1976.

\bibitem[WYZ12]{wen2012solving}
Zaiwen Wen, Wotao Yin, and Yin Zhang.
\newblock Solving a low-rank factorization model for matrix completion by a
  nonlinear successive over-relaxation algorithm.
\newblock {\em Mathematical Programming Computation}, 4(4):333--361, 2012.

\bibitem[XJZ13]{xu2013speedup}
Miao Xu, Rong Jin, and Zhi-Hua Zhou.
\newblock Speedup matrix completion with side information: Application to
  multi-label learning.
\newblock In {\em Advances in neural information processing systems}, pages
  2301--2309, 2013.

\bibitem[YD21]{ye2021global}
Tian Ye and Simon~S Du.
\newblock Global convergence of gradient descent for asymmetric low-rank matrix
  factorization.
\newblock {\em Advances in Neural Information Processing Systems}, 34, 2021.

\bibitem[YPCC16]{yi2016fast}
Xinyang Yi, Dohyung Park, Yudong Chen, and Constantine Caramanis.
\newblock Fast algorithms for robust pca via gradient descent.
\newblock In {\em Proceedings of the 30th International Conference on Neural
  Information Processing Systems}, pages 4159--4167, 2016.

\bibitem[YZMCS19]{yue2019quadratic}
Man-Chung Yue, Zirui Zhou, and Anthony Man-Cho~So.
\newblock On the quadratic convergence of the cubic regularization method under
  a local error bound condition.
\newblock {\em SIAM Journal on Optimization}, 29(1):904--932, 2019.

\bibitem[ZDG18]{zhang2018fast}
Xiao Zhang, Simon Du, and Quanquan Gu.
\newblock Fast and sample efficient inductive matrix completion via multi-phase
  procrustes flow.
\newblock In {\em International Conference on Machine Learning}, pages
  5756--5765. PMLR, 2018.

\bibitem[ZL15]{zheng2015convergent}
Qinqing Zheng and John Lafferty.
\newblock A convergent gradient descent algorithm for rank minimization and
  semidefinite programming from random linear measurements.
\newblock In {\em Proceedings of the 28th International Conference on Neural
  Information Processing Systems-Volume 1}, pages 109--117, 2015.

\bibitem[ZL16]{zheng2016convergence}
Qinqing Zheng and John Lafferty.
\newblock Convergence analysis for rectangular matrix completion using
  burer-monteiro factorization and gradient descent.
\newblock {\em arXiv preprint arXiv:1605.07051}, 2016.

\bibitem[ZLTW18]{zhu2018global}
Zhihui Zhu, Qiuwei Li, Gongguo Tang, and Michael~B Wakin.
\newblock Global optimality in low-rank matrix optimization.
\newblock {\em IEEE Transactions on Signal Processing}, 66(13):3614--3628,
  2018.

\bibitem[ZN22]{zilber2022inductive}
Pini Zilber and Boaz Nadler.
\newblock Inductive matrix completion: No bad local minima and a fast
  algorithm.
\newblock {\em arXiv preprint arXiv:2201.13052}, 2022.

\end{thebibliography}

\appendix

\section{Comparison of \GNMR to Wiberg's method, \texttt{PMF} and \texttt{R2RILS}}\label{sec:comparison}
In this section we compare \GNMR to three other iterative matrix completion methods. 
In the 1970's, several authors devised schemes to efficiently solve separable non-linear least squares problems, whose unknown variables are not necessarily matrices,
see \cite{ruhe1980algorithms,golub2003separable} and references therein.
The idea is to divide the optimization variables to two subsets, such that solving the problem for one subset while keeping the other fixed is easy. The remaining problem for the other subset is then of reduced dimensionality. 
Wiberg \cite{wiberg1976} adapted this idea to matrix completion 
as follows: Denote by 
$\tilde U(V) = \argmin_U \|\mathcal P_\Omega(UV^\top) - b\|^2$
the solution of the factorized objective \cref{eq:matrixRecovery_factorizedObjective} for $U$ given $V$. Then \cref{eq:matrixRecovery_factorizedObjective} can be written as an optimization problem over a single matrix $V$:
\begin{align} \label{eq:wiberg}
\min_{V} \left\|\mathcal P_\Omega \left[\tilde U(V) V^\top \right] - b\right\|^2.
\end{align}
At iteration $t$, Wiberg's algorithm approximately solves \cref{eq:wiberg} by the Gauss-Newton method, namely by 
linearizing \cref{eq:wiberg} around the current estimate $V_t$. 
Similar to \GNMR, the resulting least-squares problem is rank deficient, and the solution with minimal norm $\|V\|_F$ is chosen. Denoting this solution by 
${V}_{t+1}$, Wiberg's method then updates $U_{t+1} = \tilde U(V_{t+1})$.
{A regularized version of Wiberg's method, which avoids the rank deficiency, became popular in the computer vision community \cite{okatani2007wiberg,okatani2011efficient}.}

Wiberg's algorithm is similar to \GNMR, but differs from it in how $(U,V)$ are updated.
Specifically, Wiberg's algorithm applies the Gauss-Newton method to only one of the variables. Hence, {in particular} it treats the factor matrices $U,V$ in an asymmetric way. 
Empirically, in our simulations Wiberg's method performs worse than \GNMR. 
Also, to the best of our knowledge, no theoretical recovery guarantees have been derived for {it}.

The Gauss-Newton approximation for matrix completion was employed in yet another algorithm, named \texttt{PMF} \cite{paatero1994positive}. However, the setting in \cite{paatero1994positive} is slightly different from ours: instead of \cref{eq:matrixRecovery_factorizedObjective}, their goal is to minimize a weighted objective $\sum_{(i,j)\in \Omega} \left( \sum_{k=1}^r U_{ik}V_{jk} - X^*_{ij}\right)^2/\sigma_{ij}^2$ for some known weights $\sigma_{ij}$. As a result, the iterative Gauss-Newton approximation yields a full-rank least squares problem with a unique solution, and there is no need to choose a specific solution as in \GNMR. In addition, to the best of our knowledge, no theoretical recovery guarantees have been derived for this algorithm either.

Finally, we compare \GNMR to the \texttt{R2RILS} algorithm \cite{bauch2021rank}. 
Given an estimate $(U_t, V_t)$, 
the first step of \texttt{R2RILS} computes the minimal norm solution $(\tilde U, \tilde V)$ of \cref{eq:averagingVariant_LSQR} as in the averaging variant of \GNMR. 
However, instead of the update \cref{eq:averagingVariant_update}, it performs two column normalizations as follows: 
\begin{align}
\begin{pmatrix} U_{t+1} \\ V_{t+1} \end{pmatrix} &= \begin{pmatrix} \text{ColNorm}\left[U_t + \text{ColNorm}[\tilde U]\right] \\ \text{ColNorm}\left[V_t + \text{ColNorm}[\tilde V]\right] \end{pmatrix},
	\label{eq:R2RILS_update}
\end{align}
where $\text{ColNorm}[A]$ normalizes the columns of $A$ to have unit norm.
To understand the relation between the averaging variant of \GNMR
and \texttt{R2RILS}, it is instructive to analyze the latter near convergence.
As \texttt{R2RILS} converges, $U_{t+1}\approx U_t$, which implies that 
$\text{ColNorm}[\tilde U]\approx U_t$. Hence, 
the update in \cref{eq:R2RILS_update} can approximately be written
as $U_{t+1}\approx \tfrac 12 (U_t + \text{ColNorm}[\tilde U])$, 
which bears resemblance to \cref{eq:averagingVariant_update}.
Empirically, the setting and the averaging variants of \GNMR achieve superior performance over \texttt{R2RILS}, see \cref{fig:highOversampling,fig:lowOversampling_median,fig:lowOversampling_varyingDim_prob}.
In addition, the column normalizations make the theoretical analysis of \texttt{R2RILS} more difficult, and currently there are no recovery guarantees for it.


\section{Technical results}\label{sec:technical}
In this section we present some useful definitions and few technical results.
We start by recalling the classical Weyl's inequality, which states that for any two matrices $A,B$ of the same dimensions,
\begin{align}\label{eq:Weyl}
|\sigma_i(A) - \sigma_i(B)| \leq \|A-B\|_2 \quad \forall i.
\end{align}

\subsection*{Properties of balanced SVD \nopunct}\label{sec:bSVD}
\begin{lemma}\label{lem:bSVD_properties}
Let $X \in \mathbb R^{n_1\times n_2}$ be a matrix of rank $r$.
Let $Z = \begin{psmallmatrix} U \\ V \end{psmallmatrix} = \text{b-SVD}(X)$ and $\Sigma \in \mathbb R^{r\times r}$ be the diagonal matrix with the singular values of $X$. Then $Z$ is also of rank $r$ and 
\begin{equation}
\sigma_r^2(Z) = 2 \sigma_r(X). \label{eq:bSVD_sigma_r}
\end{equation}
In addition, 
\begin{equation}
U^\top U = V^\top V = \Sigma. \label{eq:bSVD_is_balanced}
\end{equation}
Finally, if $X$ is $\mu$-incoherent (see \cref{def:incoherence}), then
\begin{align}
\|U\|_{2,\infty} \leq \sqrt{\mu r \sigma_1(X)/n_1}, \quad
\|V\|_{2,\infty} \leq \sqrt{\mu r \sigma_1(X)/n_2}. \label{eq:bSVD_row_norms}
\end{align}
\end{lemma}
\begin{proof}
Denote by $\bar U \Sigma \bar V^\top$ the SVD of $X$.
Since $\bar U^\top U = \bar V^\top V = I$, then
\begin{align*}
U^\top U &= \Sigma^{\frac 12} \bar U^{\top} \bar U \Sigma^{\frac 12} = \Sigma = \Sigma^{\frac 12} \bar V^{\top} \bar V \Sigma^{\frac 12} = V^{\top} V .
\end{align*}
This proves \cref{eq:bSVD_is_balanced}.
This also implies \cref{eq:bSVD_sigma_r}, since  
\begin{align*}
\sigma^2_r(Z) &= \sigma_r(Z^\top Z) = \sigma_r(U^\top U + V^\top V) = 2\sigma_r(\Sigma) = 2\sigma_r(X).
\end{align*}
Finally, using $\|AB\|_{2,\infty} \leq \|A\|_{2,\infty} \|B\|_2$ and the $\mu$-incoherence assumption,
\begin{align*}
\|U\|_{2,\infty} = \|\bar U \Sigma^\frac{1}{2}\|_{2,\infty} \leq \sqrt{\sigma_1(X)} \|\bar U\|_{2,\infty} \leq \sqrt{\mu r \sigma_1(X)/n_1},
\end{align*}
and similarly $\|V\|_{2,\infty} \leq \sqrt{\mu r \sigma_1(X)/n_2}$. This proves \cref{eq:bSVD_row_norms}.
\end{proof}

\subsection*{A novel result on the Procrustes distance}\label{sec:Procrustes}
In \cref{sec:Q_distance} we presented the Q-distance between factor matrices.
Another distance measure, which was used in several previous works \cite{zheng2015convergent,chen2015completing,tu2016low,yi2016fast,zheng2016convergence}, is the Procrustes distance.
In what follows we shall use it to bound the Q-distance.
\begin{definition}\label{def:Procrustes_distance}
The Procrustes distance between $Z_1, Z_2 \in \mathbb R^{n\times r}$ is defined as
\begin{align*}
d_P(Z_1, Z_2) &= \min \left\{ \|Z_1 - Z_2 P\|_F \,\mid\, P \in \mathbb R^{r\times r} \text{ is orthogonal} \right\} .
\end{align*}
\end{definition}
This distance is closely related to the Wahba's problem \cite{wahba1965least}; the latter, however, allows different weights to the column norms of the difference $Z_1 - Z_2P$ instead of the (uniform) Frobenius norm, but on the other hand constrains $P$ to be a rotation matrix with unit determinant.
In contrast to the Q-distance, the Procrustes distance is symmetric in its arguments, and its minimizer always exists. Moreover, it can be explicitly written in terms of the SVD of $Z_1^\top Z_2$: The minimizer of $d_P(Z_1,Z_2)$ is $\bar U \bar V^\top$ where $\bar U \Sigma \bar V$ is the SVD of $Z_1^\top Z_2$. 
A simple yet useful inequality which involves the Procrustes distance is the following one.
\begin{proposition}\label{prop:sigmaMin_U}
Let $Z_i = \begin{psmallmatrix} U_i \\ V_i \end{psmallmatrix}$ where $U_i \in \mathbb R^{n_1\times r}$, $V_i \in \mathbb R^{n_2\times r}$ and $r\leq \min\{n_1,n_2\}$ for $i=1,2$.
Then
\begin{align*}
\min\left\{\left|\sigma_i(U_1) - \sigma_i(U_2)\right|, \left|\sigma_i(V_1) - \sigma_i(V_2)\right|\right\} &\leq d_P(Z_1, Z_2) \quad \forall i \in [r].
\end{align*}
\end{proposition}
\begin{proof}
Let $P \in \mathbb R^{r\times r}$ be the (orthogonal) minimizer of the Procrustes distance between $U_1$ and $U_2$.
Then Weyl's inequality \cref{eq:Weyl} implies
\begin{align*}
\left|\sigma_i(U_1) - \sigma_i(U_2)\right| &= \left|\sigma_i(U_1) - \sigma_i(U_2 P)\right| \leq \|U_1 - U_2 P \|_2 \leq \|U_1 - U_2 P\|_F \leq d_P(Z_1, Z_2),
\end{align*}
and similarly for $|\sigma_i(V_1) - \sigma_i(V_2)|$.
\end{proof}

The next result bounds the Procrustes distance between a general pair of factor matrices and a b-SVD. Note that the following is a stronger version of \cite[Lemma~5.14]{tu2016low}.
\begin{lemma}[Procrustes distance bound]\label{lem:dP_e_bound}
Let $X^* \in \mathbb R^{n_1\times n_2}$ be a matrix of rank $r$, and denote $Z^* = \text{b-SVD}(X^*)$.
Then for any $Z = \begin{psmallmatrix} U \\ V \end{psmallmatrix} \in \mathbb R^{(n_1+n_2)\times r}$,
\begin{align*}
d_P^2(Z, Z^*) \leq \frac{1}{(\sqrt 2 - 1)\sigma_r\left(X^*\right)} \left(\|UV^\top - X^*\|_F^2 + \frac 14 \|U^\top U - V^\top V\|_F^2\right) . 
\end{align*}
\end{lemma}

To prove \cref{lem:dP_e_bound} we shall use the following auxiliary lemma \cite[Lemma~5.4]{tu2016low}.
\begin{lemma}\label{lem:lemma5.4_inTBSSR16}
For any $Z, Z^* \in \mathbb R^{(n_1+n_2)\times r}$,
\begin{align}\label{eq:lemma5.4_inTBSSR16}
d^2_P(Z, Z^*) \leq \frac{1}{2(\sqrt 2 - 1) \sigma^2_r(Z^*)} \|ZZ^\top - Z^* Z^{*\top}\|^2_F .
\end{align}
\end{lemma}

\begin{proof}[Proof of \cref{lem:dP_e_bound}]
Let $Z^* = \text{b-SVD}(X^*)$. By \cref{eq:bSVD_sigma_r} of \cref{lem:bSVD_properties}, we have $\sigma^2_r(Z^*) = 2\sigma_r(X^*)$.
In view of \cref{eq:lemma5.4_inTBSSR16} of \cref{lem:lemma5.4_inTBSSR16}, it thus suffices to show that
\begin{equation}
\|ZZ^\top - Z^* Z^{*\top}\|^2_F \leq 4\|X-X^*\|^2_F + \|U^\top U - V^\top V\|_F^2 \label{eq:ZZ_XX_bound}
\end{equation}
where $X = UV^\top$. Note that \cref{eq:ZZ_XX_bound},
which we shall now prove, is a stronger version of \cite[Lemma~4]{zheng2016convergence}. 
Denote $\begin{psmallmatrix} U^* \\ V^* \end{psmallmatrix} = Z^*$. Then, by definition,
\begin{align*}
\|ZZ^\top - Z^* Z^{*\top}\|^2_F = \|UU^\top - U^* U^{*\top}\|^2_F + \|VV^\top- V^*V^{*\top}\|^2_F + 2\|X - X^*\|_F^2.
\end{align*}
We first simplify some of the terms above.
By $\|AA^\top\|^2_F = \|A^\top A\|^2_F$ and \cref{eq:bSVD_is_balanced} of \cref{lem:bSVD_properties}, 
\begin{subequations}\begin{align}
\|UU^\top - U^* U^{*\top}\|^2_F &= \|UU^\top\|^2_F + \|U^* U^{*\top}\|^2_F - 2\Tr\left(U^* U^{*\top} UU^\top\right), \nonumber \\
&=  \|U^\top U\|^2_F + \|\Sigma^*\|^2_F - 2\Tr\left(U^* U^{*\top} UU^\top\right), \label{eq:UUT_U*U*T}
\end{align}
and similarly
\begin{align}
\|VV^\top - V^*V^{*\top}\|^2_F &= \|V^\top V\|^2_F + \|\Sigma^*\|^2_F - 2\Tr\left(V^* V^{*\top} VV^\top\right). \label{eq:VVT_V*V*T}
\end{align}\end{subequations}
Let $\mathcal E = U^\top U - V^\top V$. By its symmetry, $\mathcal E$ satisfies $\Tr(\mathcal E^2) = \Tr(\mathcal E^\top \mathcal E) = \|\mathcal E\|_F^2$.
Using the trace property $\Tr(AB) = \Tr(BA)$ for square matrices $A,B$, the first term on the RHS of \cref{eq:UUT_U*U*T} can be rewritten as
\begin{align*}
\|U^\top U\|^2_F &= \Tr\left((V^\top V + \mathcal E)(V^\top V + \mathcal E) \right)
= \|V^\top V\|^2_F + 2\Tr\left( V^\top V \mathcal E\right) + \|\mathcal E\|_F^2.
\end{align*}
Combining the above three equations gives that
\begin{align}
\|ZZ^\top - Z^*Z^{*\top}\|^2_F & =
2 \| V^\top V\|_F^2 +2 \Tr(V^\top V\mathcal E) + 
\|\mathcal E\|_F^2 + 2 \|\Sigma^*\|_F^2 + 2 \|X-X^*\|_F^2 \nonumber \\
&- 2 \Tr\left(U^* U^{*\top} UU^\top\right) -
2 \Tr\left(V^* V^{*\top} VV^\top\right) . \label{eq:ZZ_minus_Z*Z*}
\end{align}
Next, by the trace property $\Tr(AB) = \Tr(BA)$ for $A\in \mathbb R^{n\times r}$, $B \in \mathbb R^{r\times n}$ we have
\begin{align*}
\|X\|_F^2 &= \Tr\left(VU^\top UV^\top\right)
= \Tr\left(V^\top V U^\top U\right) 
= \Tr\left(V^\top V (V^\top V + \mathcal E)\right) \\
= \|V^\top V\|_F^2 + \Tr\left( V^\top V \mathcal E \right),
\end{align*}
Since $\|X-X^*\|^2_F = \|X\|_F^2 + \|X^*\|_F^2 - 2\Tr(X^{*\top}X) = \|X\|_F^2 + \|\Sigma^*\|_F^2 - 2\Tr(V^* U^{*\top} UV^\top)$, we obtain that  
\begin{align*}
\|V^\top V\|_F^2 = \|X-X^*\|_F^2 - \Tr\left( V^\top V \mathcal E \right) - \|\Sigma^*\|_F^2 + 2\Tr\left( V^* U^{*\top} UV^\top\right) .
\end{align*}
Inserting this into \cref{eq:ZZ_minus_Z*Z*} gives that
\begin{align*}
\|ZZ^\top - Z^*Z^{*\top}\|^2_F &= 4 \|X-X^*\|_F^2 + \|\mathcal E\|_F^2 \\
&+ 2\Tr\left( 2V^* U^{*\top} UV^\top - U^* U^{*\top} UU^\top - V^* V^{*\top} VV^\top\right).
\end{align*}
Finally, using again the trace property $\Tr(AB) = \Tr(BA)$ for $A\in \mathbb R^{n\times r}$, $B \in \mathbb R^{r\times n}$, the lemma follows since
\begin{align*}
&\Tr\left( U^* U^{*\top} UU^\top + V^* V^{*\top} VV^\top - 2 V^* U^{*\top} UV^\top \right) \\
&= \Tr\left( U^{*\top} UU^\top U^* + V^{*\top} VV^\top V^* - 2 U^{*\top} UV^\top V^* \right) = \|U^\top U^* - V^\top V^*\|_F^2 \geq 0.
\end{align*}
\end{proof}

\subsection*{Proof of \cref{lem:dQ_e_bound} (bounds on the Q-distance)}
To prove \cref{lem:dQ_e_bound} we shall use the following auxiliary lemma \cite[Lemma~1]{ma2021beyond}.
\begin{lemma}\label{lem:lemma1_MLC21}
Let $Z^* = \begin{psmallmatrix} U^* \\ V^* \end{psmallmatrix} = \text{b-SVD}(X^*)$ where $X^* \in \mathbb R^{n_1\times n_2}$ is of rank $r$.
Let $Z = \begin{psmallmatrix} U \\ V \end{psmallmatrix} \in \mathbb R^{(n_1+n_2)\times r}$, and suppose that there exists an invertible matrix $P \in \mathbb R^{r\times r}$ with $1/2 \leq \sigma_r(P) \leq \sigma_1(P) \leq 3/2$ such that
\begin{align}\label{eq:lemma1_MLC21_assumption}
\max\{\|U^* - UP\|_F, \|V^* - VP^{-\top}\|_F\} \leq \tfrac{1}{20} \sqrt{\sigma_r(X^*)}.
\end{align}
Then the optimal alignment matrix $Q \in \mathbb R^{r\times r}$ that minimizes the Q-distance $d_Q(Z^*, Z)$, exists and satisfies
\begin{align*}
\|P-Q\|_2 \leq \|P - Q\|_F \leq \frac{5}{\sqrt{\sigma_r(X^*)}} \max\{\|U^* - UP\|_F, \|V^* - VP^{-\top}\|_F\}.
\end{align*}
\end{lemma}
We remark that the original version of \cite[Lemma~1]{ma2021beyond} required $\max\{\|U^* - UP\|_F, \|V^* - VP^{-\top}\|_F\} \leq {\sqrt{\sigma_r(X^*)}}/{80}$. However, by tracing its proof, it is straightforward to see that our version with a sharper constant also holds.

As discussed in the main text, bounding $d_Q(Z, Z^*)$ is more challenging than bounding $d_Q(Z^*, Z)$.
Our strategy to prove the bound \cref{eq:dQ_e_bound} of the first quantity is as follows.
\cref{lem:lemma1_MLC21} states that when the Procrustes distance $d_P(Z,Z^*)$ is not too large, the optimal alignment matrix between $Z$ and $Z^*$ is nearly orthogonal. Intuitively, this implies that $d_Q(Z, Z^*)$ is close to $d_Q(Z^*, Z)$, and thus similarly bounded. In the second part of the following proof we formalize this argument.

\begin{proof}[Proof of \cref{lem:dQ_e_bound}]
Let $\tilde Z = \begin{psmallmatrix} \tilde U \\ \tilde V \end{psmallmatrix} = \text{b-SVD}(UV^\top)$.
First, we show that $d_Q(Z^*,Z) = d_Q(Z^*,\tilde Z)$.
Indeed, since $\tilde U \tilde V^\top = UV^\top$, there exists an invertible matrix $Q\in\mathbb R^{r\times r}$ such that $U = \tilde UQ$ and $V = \tilde VQ^{-\top}$. By definition of the Q-distance (\cref{def:Q_distance}), this implies $d_Q(Z^*, Z) = d_Q(Z^*, \tilde Z)$.

Next, we bound the Q-distance $d_Q^2(Z^*, \tilde Z)$ via the Procrustes distance.
Note that for any fixed $Z_1, Z_2$, it holds that $d_Q(Z_1, Z_2) \leq d_P(Z_1, Z_2)$ since the former involves minimization over any invertible matrix $Q$, whereas the latter involves minimization over a smaller subset of orthogonal matrices $P$.
In particular, $d_Q(Z^*, \tilde Z) \leq d_P(Z^*, \tilde Z) = d_P(\tilde Z, Z^*)$. Invoking \cref{lem:dP_e_bound} thus yields
\begin{align*}
d_Q^2(Z^*, Z) = d_Q^2(Z^*, \tilde Z) \leq \frac{1}{(\sqrt 2 - 1)\sigma_r^*} \left(\|UV^\top - X^*\|_F^2 + \frac 14 \|\tilde U^\top \tilde U - \tilde V^\top \tilde V\|_F^2 \right) .
\end{align*}
Equation~\cref{eq:dQ_leftbSVD_e_bound} of the lemma follows since the second term on the RHS vanishes due to \cref{eq:bSVD_is_balanced} of \cref{lem:bSVD_properties}.

Next, assume \cref{eq:dQ_e_bound_assumption} holds.
Combining \cref{eq:dQ_e_bound_assumption} and \cref{lem:dP_e_bound} yields that there exists an orthogonal $P \in \mathbb R^{r\times r}$ such that
\begin{align}\label{eq:dP_bound}
\|U^* - UP\|_F^2 + \|V^* - VP^{-\top}\|_F^2 = \|U^* - UP\|_F^2 + \|V^* - VP\|_F^2 \leq \frac{\sigma_r^*}{400}.
\end{align}
Hence $P$, whose all singular values are $1$, satisfies \cref{eq:lemma1_MLC21_assumption}. Invoking \cref{lem:lemma1_MLC21} implies that the optimal alignment matrix $\tilde{Q}$ between $Z^*$ and $Z$ exists and satisfies $\|\tilde{Q} - P\|_2 \leq 1/4$. By the unitarity of $P$ this implies
\begin{align}\label{eq:Q_norm_bound}
\|\tilde{Q}\|_2 \leq \|\tilde{Q}-P\|_2 + \|P\|_2 \leq \frac 54.
\end{align}
Next, we bound $\|\tilde{Q}^{-1}\|_2$.
By Weyl's inequality \cref{eq:Weyl}, $|\sigma_r(\tilde{Q}) - \sigma_r(P)| \leq \|\tilde{Q} - P\|_2 \leq 1/4$. Since $\sigma_r(P) = 1$,
\begin{align}\label{eq:Qinverse_norm_bound}
\|\tilde{Q}^{-1}\|_2 = \frac{1}{\sigma_r(\tilde{Q})} \leq \frac{1}{1 - 1/4} = \frac 43 .
\end{align}
Finally, let $\mathcal E_U = U^* - U \tilde{Q}$, $\mathcal E_V = V^* - V \tilde{Q}^{-\top}$ and $X = UV^\top$. Then, by the first part of the lemma \cref{eq:dQ_leftbSVD_e_bound},
\begin{align}\label{eq:EU_EV_bound}
\max\{\|\mathcal E_U\|^2_F,\, \|\mathcal E_V\|^2_F\} \leq \frac{\|X-X^*\|_F^2}{(\sqrt{2} - 1)\sigma_r^*} .
\end{align}
Let $Q = \tilde{Q}^{-1}$. Then $\|Q\|_2 \leq 4/3$ by \cref{eq:Qinverse_norm_bound}. In addition, putting everything together yields
\begin{align*}
\|U - U^*Q\|_F^2 + \|V - V^* Q^{-\top}\|_F^2 &= \|\mathcal E_U Q^{-1}\|_F^2 + \|\mathcal E_V Q^\top\|_F^2 \\
&\stackrel{(a)}{\leq} \|Q^{-1}\|^2_2 \|\mathcal E_U\|_F^2 + \|Q\|^2_2 \|\mathcal E_V\|_F^2 \\
&\stackrel{(b)}{\leq} \frac{\|X - X^*\|_F^2}{(\sqrt{2} - 1)\sigma_r^*} \left(\frac 43 + \frac 54 \right)  \\
&\leq \frac{25}{4} \frac{\|X - X^*\|_F^2}{\sigma_r^*},
\end{align*}
where (a) follows from the first part of \cref{prop:propB4_SL16} and (b) from \cref{eq:Q_norm_bound}, \cref{eq:Qinverse_norm_bound} and \cref{eq:EU_EV_bound}.
\end{proof}

\subsection*{Bounds for pairs of factor matrices}
Given a pair of factor matrices $\begin{psmallmatrix} U \\ V \end{psmallmatrix}$, the following lemma bounds the balance of a new pair $\begin{psmallmatrix} U' \\ V' \end{psmallmatrix}$ and the distance of its corresponding matrix $U'V'^\top$ from $UV^\top$, in terms of the Procrustes distance.

\begin{lemma}\label{lem:update_balance_distance_bounds}
Let $Z = \begin{psmallmatrix} U \\ V \end{psmallmatrix}, Z' = \begin{psmallmatrix} U' \\ V' \end{psmallmatrix} \in \mathbb R^{(n_1+n_1)\times r}$. Denote $d = d_P(Z,Z')$ and
\begin{align*}
a = \sqrt 2 \max\{\sigma_1(U), \sigma_1(V)\} d + \tfrac 12 d^2 .
\end{align*}
Then
\begin{subequations}\label{eq:update_balance_distance_bounds}\begin{align}
\|U'^\top U' - V'^\top V'\|_F &\leq \|U^\top U - V^\top V\|_F + 2a,  \label{eq:update_balance_bound} \\
\|U'V'^\top - UV^\top\|_F &\leq a \label{eq:update_distance_bound} .
\end{align}\end{subequations}
\end{lemma}
\begin{proof}
Let $P \in \mathbb R^{r\times r}$ be the minimizer of the Procrustes distance between $Z$ and $Z'$, and denote $\Delta U = U'P - U$ and $\Delta V = V'P - V$. Then $\|\Delta U\|_F^2 + \|\Delta V\|_F^2 = d^2$. In addition, since $P$ is unitary, $U'V'^\top = (U'P)(V'P)^\top = (U+\Delta U)(V+\Delta V)^\top$.
Equation~\cref{eq:update_distance_bound} holds since
\begin{align*}
\|U'V'^\top - UV^\top\|_F &\leq \|U\Delta V^\top\|_F + \|\Delta U V^\top\|_F + \|\Delta U \Delta V^\top\|_F \\
&\stackrel{(a)}{\leq} \sigma_1(U) \|\Delta V\|_F + \sigma_1(V) \|\Delta U\|_F + \|\Delta U \|_F \|\Delta V \|_F \\
&\stackrel{(b)}{\leq} \max\{\sigma_1(U), \sigma_1(V)\} \left( \|\Delta V\|_F + \|\Delta U\|_F\right) + \tfrac 12 d^2 \\
&\stackrel{(c)}{\leq} \sqrt 2 \max\{\sigma_1(U), \sigma_1(V)\} d + \tfrac 12 d^2,
\end{align*} 
where (a) follows from the first part of \cref{prop:propB4_SL16} and the Cauchy-Schwarz inequality, (b) from the inequality $ab \leq (a^2 + b^2)/2$, and (c) from the inequality $a + b \leq \sqrt{2(a^2+b^2)}$.
Next, by the triangle inequality,
\begin{align}\label{eq:update_balance_temp_bound}
\|U'^\top U' - V'^\top V'\|_F &= \|P^\top \left(U'^\top U' - V'^\top V'\right)P \|_F \nonumber \\
&= \|(U+\Delta U)^\top (U+\Delta U) - (V+\Delta V)^\top (V+\Delta V) \|_F \nonumber \\
&\leq \|U^\top U - V^\top V\|_F 
+ 2\|U^\top \Delta U - \Delta V^\top V\|_F + \|\Delta U^\top \Delta U - \Delta V^\top \Delta V\|_F.
\end{align}
The last term of the RHS above can be bounded by the Cauchy-Schwarz inequality as
\begin{align*}
\|\Delta U^\top \Delta U\|_F + \|\Delta V^\top \Delta V\|_F \leq \|\Delta U\|_F^2 + \|\Delta V\|_F^2 = d^2 .
\end{align*}
As for the second term, by combining \cref{prop:propB4_SL16} and the inequality $a + b \leq \sqrt{2(a^2+b^2)}$,
\begin{align*}
\|U^\top \Delta U - \Delta V^\top V\|_F &\leq \sigma_1(U)\|\Delta U\|_F + \sigma_1(V)\|\Delta V\|_F \\
&\leq \max\{\sigma_1(U), \sigma_1(V)\} \left(\|\Delta U\|_F + \|\Delta V\|_F \right) \\
&\leq \sqrt 2 \max\{\sigma_1(U), \sigma_1(V)\} d.
\end{align*}
Inserting these bounds into \cref{eq:update_balance_temp_bound} yields \cref{eq:update_balance_bound}.
\end{proof}

\section{Proof of \cref{thm:sens_quadConvergence} (matrix sensing)}\label{sec:sensing_proof_SM}
The proof is based on the following lemma which considers a single iteration of \cref{alg:GNMR}.
\begin{lemma}\label{lem:sens_singleIteration}
Let $\delta$ be a positive constant strictly smaller than one.
Let $c_e=c_e(\delta)$ be sufficiently large.
Assume the sensing operator $\mathcal A$ satisfies a $2r$-RIP with a constant $\delta_{2r} \leq \delta$.
Let $X^* \in \mathbb R^{n_1\times n_2}$ be a matrix of rank $r$, and denote $\gamma = {c_e}/{(2\sigma_r^*)}$.
Let $Z_t = \begin{psmallmatrix} U_t \\ V_t \end{psmallmatrix}$ be the current iterate, and denote the estimation error $e_t = \|X_t - X^*\|_F$. Also denote the minimal singular value of the factor matrices $s_t = \min\{\sigma_r(U_t),\sigma_r(V_t)\}$ and the imbalance $l_t = \|U_t^\top U_t - V_t^\top V_t\|_F$.
Assume that the current iterate satisfies the following three conditions:
\begin{subequations}\label{eq:sens_assumptions}\begin{align}
e_t &\leq \frac{\sigma_r^*}{c_e}, \label{eq:sens_et_assumption}\\
s_t &\geq \frac 12 \sqrt{\sigma_r^*} + \sqrt\frac{1+\delta}{1-\delta} \frac{4e_t}{\sqrt{\sigma_r^*}}, \label{eq:sens_st_assumption} \\
l_t &\leq \frac{1}{c_e} \sigma_r^* - \frac{1+\delta}{1-\delta} \frac{8e_t^2}{\sigma_r^*} . \label{eq:sens_lt_assumption}
\end{align}\end{subequations}
Then the next iterate of \cref{alg:GNMR} with $\alpha=-1$ satisfies
\begin{subequations}\begin{align}
e_{t+1} &\leq \gamma e^2_t, \label{eq:sens_et+1_guarantee} \\
s_{t+1} &\geq \frac 12 \sqrt{\sigma_r^*} + \sqrt\frac{1+\delta}{1-\delta} \frac{4e_{t+1}}{\sqrt{\sigma_r^*}}, \label{eq:sens_st+1_guarantee} \\
l_{t+1} &\leq \frac{1}{c_e} \sigma_r^* - \frac{1+\delta}{1-\delta} \frac{8e_{t+1}^2}{\sigma_r^*}. \label{eq:sens_lt+1_guarantee}
\end{align}\end{subequations}
\end{lemma}

\begin{proof}[Proof of \cref{thm:sens_quadConvergence}]
Let $Z_0$ be an initial guess which satisfies the conditions of the theorem.
Let us show that it satisfies assumptions \cref{eq:sens_assumptions} of \cref{lem:sens_singleIteration} at $t=0$.
First of all, \cref{eq:sens_et_assumption} holds by the assumption $Z_0 \in \mathcal B_\text{err}(1/c_e)$.
Next, since $Z_0 \in \mathcal B_\text{err}(1/c_e) \cap \mathcal B_\text{bln}(1/(2c_e))$, \cref{lem:dP_e_bound} implies
\begin{align*}
d^2_P\left(Z_0, Z^* \right) \leq \frac{1}{(\sqrt 2 - 1) \sigma_r^*} \left(e^2_0 + \frac 14 \frac{\sigma_r^{*2}}{4c_e^2} \right) \leq \frac{3\sigma_r^*}{c_e^2} .
\end{align*}
Hence $d^2_P(U_0, U^*) \leq {3\sigma_r^*}/{c_e^2}$.
Combining this with \cref{prop:sigmaMin_U} yields
\begin{align*}
\sigma_r(U_0) &\geq \sigma_r(U^*) - d_P(Z_0, Z^*) \geq \sqrt{\sigma_r^*} - \frac{\sqrt{3\sigma_r^*}}{c_e} .
\end{align*}
Together with a similar bound for $\sigma_r(V_0)$ we obtain
\begin{align*}
s_0 \geq \left(1 - \frac{\sqrt 3}{c_e} \right) \sqrt{\sigma_r^*} .
\end{align*}
In addition, $e_0 \leq \frac{\sigma_r^*}{c_e}$ by the assumption $Z_0 \in \mathcal B_\text{err}(1/c_e)$. Hence, for \cref{eq:sens_st_assumption} to hold at $t=0$, we need $1 - \frac{\sqrt 3}{c_e} \geq \frac 12 + \frac{4}{c_e}\sqrt\frac{1+\delta}{1-\delta}$. For any fixed $\delta < 1$, this holds for large enough $c_e \equiv c_e(\delta)$.

Finally, we prove \cref{eq:sens_lt_assumption} at $t=0$. By the assumption $Z_0 \in \mathcal B_\text{bln}(1/(2c_e))$, the LHS of \cref{eq:sens_lt_assumption}, $l_0$, is upper bounded by $\sigma_r^*/(2c_e)$. In addition, by $Z_0 \in \mathcal B_\text{err}(1/c_e)$, the RHS of \cref{eq:sens_lt_assumption} is lower bounded by
\begin{align*}
\frac{1}{c_e}\sigma_r^* - \frac{1+\delta}{1-\delta}\frac{8e_0^2}{\sigma_r^*} \geq \left(1 - \frac{1+\delta}{1-\delta} \frac{8}{c_e}\right)\frac{\sigma_r^*}{c_e} \geq \frac 12 \frac{\sigma_r^*}{c_e}
\end{align*}
where the last inequality follows for large enough $c_e \equiv c_e(\delta)$.
Hence \cref{eq:sens_lt_assumption} holds at $t=0$.
The theorem thus follows by iteratively applying \cref{lem:sens_singleIteration}.
\end{proof}

\begin{proof}[Proof of \cref{lem:sens_singleIteration}]
Let us begin with \cref{eq:sens_et+1_guarantee}.
Combining assumptions \cref{eq:sens_et_assumption} and \cref{eq:sens_lt_assumption} yields
\begin{align*}
e_t^2 + \frac 14 \|U_t^\top U_t - V_t^\top V_t\|_F^2
&\leq \frac{5\sigma_r^{*2}}{4c_e^2}.
\end{align*}
Hence, for large enough $c_e$, the current iterate $Z_t$ satisfies condition \cref{eq:dP_assumption} of \cref{lem:sens_et+1_bound}.
By \cref{lem:sens_et+1_bound},
\begin{align*}
e_{t+1} \leq \frac 12 \sqrt\frac{1+\delta}{1-\delta} \left(\frac{25}{4 \sigma_r^{*}} e_t^2 + \Delta_t^2\right),
\end{align*}
where $\Delta^2_t = \|U_{t+1} - U_t\|^2_F + \|V_{t+1} - V_t\|^2_F$.
By assumption \cref{eq:sens_st_assumption}, $U_t$ and $V_t$ have full column rank. Hence, $\Delta^2_t$ can be bounded by combining \cref{lem:minNormSol_normBound} and \cref{eq:sens_st_assumption} as
\begin{align}
\Delta_t^2 \leq \frac{1+\delta}{1-\delta} \frac{e_t^2}{s_t^2} \leq \frac{1+\delta}{1-\delta} \frac{4e_t^2}{ \sigma_r^*}. \label{eq:Deltat_bound}
\end{align}
We thus conclude
\begin{align*}
e_{t+1} &\leq \frac 12 \sqrt\frac{1+\delta}{1-\delta}\left[\frac{25}{4} + 4\frac{1+\delta}{1-\delta} \right] \frac{e_t^2}{\sigma_r^{*}} .
\end{align*}
Hence \cref{eq:sens_et+1_guarantee} holds for a large enough $c_e$.

Next, we prove \cref{eq:sens_st+1_guarantee}.
By \cref{eq:Deltat_bound} we have $\|U_{t+1} - U_t\|_2 \leq \Delta_t \leq \sqrt\frac{1+\delta}{1-\delta} \frac{2e_t}{\sqrt{\sigma_r^*}}$. Combined with Weyl's inequality \cref{eq:Weyl} and assumption \cref{eq:sens_st_assumption}, this implies
\begin{align*}
\sigma_r(U_{t+1}) &\geq \sigma_r(U_{t}) - \sqrt\frac{1+\delta}{1-\delta} \frac{2e_t}{\sqrt{\sigma_r^*}} \geq \frac 12 \sqrt{\sigma_r^*} + \sqrt\frac{1+\delta}{1-\delta} \frac{2e_t}{\sqrt{\sigma_r^*}} .
\end{align*}
Together with a similar bound on $\sigma_r(V_{t+1})$ we obtain
\begin{align}\label{eq:sens_st_temp_bound}
s_{t+1} \geq \frac 12 \sqrt{\sigma_r^*} + \sqrt\frac{1+\delta}{1-\delta} \frac{2e_t}{\sqrt{\sigma_r^*}}.
\end{align}
In addition, combining \cref{eq:sens_et+1_guarantee} and \cref{eq:sens_et_assumption} yields
\begin{align} \label{eq:et+1_temp_etBound}
e_{t+1} \leq \tfrac 12 e_t .
\end{align}
Inequality \cref{eq:sens_st+1_guarantee} follows by inserting \cref{eq:et+1_temp_etBound} into \cref{eq:sens_st_temp_bound}.

Finally, we prove \cref{eq:sens_lt+1_guarantee}.
By assumption \cref{eq:sens_st_assumption}, $U_t$ and $V_t$ have full column rank.
Invoking \cref{lem:minNormSol_implicitBalance} thus gives that
\begin{align*}
l_{t+1} = \|U_{t+1}^\top U_{t+1} - V_{t+1}^\top V_{t+1}\|_F &\leq \|U_t^\top U_t - V_t^\top V_t\|_F + \frac{1+\delta}{1-\delta} \frac{e_t^2}{s_t^2} .
\end{align*}
Bounding the first term on the RHS by assumption \cref{eq:sens_lt_assumption} and the second term by \cref{eq:Deltat_bound} yields
\begin{align*}
l_{t+1} &\leq \frac{\sigma_r^*}{c_e} - \frac{1+\delta}{1-\delta}\frac{8e_t^2}{\sigma_r^*} + \frac{1+\delta}{1-\delta} \frac{4e_t^2}{\sigma_r^*} = \frac{\sigma_r^*}{c_e} - \frac{1+\delta}{1-\delta} \frac{4e_t^2}{\sigma_r^*} .
\end{align*}
Inequality \cref{eq:sens_lt+1_guarantee} follows by combining this with \cref{eq:et+1_temp_etBound}.
\end{proof}

\section{Proof of \cref{thm:sens_noisy} (noisy matrix sensing)}
The proof is based on the following lemma which considers a single iteration of \cref{alg:GNMR_bSVD}.
\begin{lemma}\label{lem:sens_noisy_singleIteration}
Let $\delta$ be a positive constant strictly smaller than one,
and denote $c=7(1+\delta)^\frac{3}{2}/(1-\delta)^\frac{3}{2}$.
Assume the sensing operator $\mathcal A$ satisfies a $2r$-RIP with a constant $\delta_{2r} \leq \delta$.
Let $b = \mathcal A(X^*) + \xi$ where $X^* \in \mathbb R^{n_1\times n_2}$ is of rank $r$ and $\xi \in \mathbb R^m$ satisfies
\begin{align}\label{eq:sens_noisy_xi_assumption}
\|\xi\| \leq \frac{\sigma_r^*\sqrt{1-\delta}}{6c} .
\end{align}
Denote $\gamma = c/(4\sigma_r^*)$.
Let $Z_t = \begin{psmallmatrix} U_t \\ V_t \end{psmallmatrix}$ be the current iterate. Assume that its estimation error $e_t = \|X_t - X^*\|_F$ satisfies
\begin{align}
e_t &\leq \frac{\sigma_r^*}{c} + \frac{3\|\xi\|}{\sqrt{1-\delta}}. \label{eq:sens_noisy_et_assumption}
\end{align}
Then the next iterate of \cref{alg:GNMR_bSVD} with $\alpha=-1$ satisfies
\begin{align}
e_{t+1} &\leq \gamma e^2_t + \frac{3\|\xi\|}{\sqrt{1-\delta}} . \label{eq:sens_noisy_et+1_guarantee}
\end{align}
\end{lemma}

\begin{proof}[Proof of \cref{thm:sens_noisy}]
By assumption, at $t=0$ the error satisfies
\begin{align}\label{eq:sens_noisy_initAssumption}
\|X_0 - X^*\|_F \leq \frac{\sigma_r^*}{c}.
\end{align}
Hence \cref{eq:sens_noisy_et_assumption} holds at $t=0$.
The proof follows by induction:
We show that if \cref{eq:sens_noisy_et_assumption} holds at iteration $t$, then it holds at $t$. By combining \cref{eq:sens_noisy_xi_assumption,eq:sens_noisy_et_assumption},
\begin{align*}
\gamma e^2_{t} &\leq \frac{c}{4\sigma_r^*} \left(\frac{\sigma_r^*}{c} + \frac{3\|\xi\|}{\sqrt{1-\delta}}\right)^2 \leq \left(1 + \frac{3}{6}\right)^2 \frac{\sigma_r^*}{4c} \leq \frac{\sigma_r^*}{c} .
\end{align*}
Plugging this into \cref{eq:sens_noisy_et+1_guarantee} of \cref{lem:sens_noisy_singleIteration} yields $e_{t+1} \leq \frac{\sigma_r^*}{c} + \frac{3\|\xi\|}{\sqrt{1-\delta}}$, namely \cref{eq:sens_noisy_et_assumption} holds at iteration $t+1$.
Equation~\cref{eq:sens_noisy} of the theorem follows by iteratively applying \cref{lem:sens_noisy_singleIteration}.

Next, let $x = 3\|\xi\|/\sqrt{1-\delta}$.
We shall prove by induction that
\begin{align}\label{eq:sens_noisy_induction}
\|X_{t} - X^*\|_F \leq \frac{1}{4^{2^t-1}} \frac{\sigma_r^*}{c} + 2x.
\end{align}
At $t=0$, \cref{eq:sens_noisy_induction} follows by the initialization assumption \cref{eq:sens_noisy_initAssumption}. By combining \cref{eq:sens_noisy_initAssumption,lem:sens_noisy_singleIteration}, \cref{eq:sens_noisy_induction} holds also at $t=1$, since
\begin{align*}
\|X_1 - X^*\|_F \leq \gamma \|X_0-X^*\|_F^2 + x \leq \frac{\sigma_r^*}{4c} + x \leq \frac{\sigma_r^*}{4c} + 2x .
\end{align*}
Next, assume \cref{eq:sens_noisy_induction} holds at some $t\geq 1$. Then
\begin{align*}
\|X_{t+1} - X^*\|_F &\leq \gamma \|X_{t} - X^*\|_F^2 + x \leq \frac{c}{4\sigma_r^*} \left(\frac{1}{4^{2^{t}-1}} \frac{\sigma_r^*}{c} + 2x\right)^2 + x \\
&= \frac{1}{4^{2^{{t}+1}-1}} \frac{\sigma_r^*}{c} + \frac{1}{4^{2^{t}-1}} x + \frac{c}{\sigma_r^*}x^2 + x
\leq \frac{1}{4^{2^{{t}+1}-1}} \frac{\sigma_r^*}{c} + \left(\frac{1}{4^{2^{t}-1}} + \frac{3}{2}\right) x,
\end{align*}
where the last inequality follows by the assumption $x \leq \sigma_r^*/(2c)$ \cref{eq:sens_noisy_xi_assumption}. Since $\frac{1}{4^{2^{t}-1}} + \frac{3}{2} \leq 2$ for any $t\geq 1$, \cref{eq:sens_noisy_induction} holds at $t+1$. This completes the proof.
\end{proof}

\begin{proof}[Proof of \cref{lem:sens_noisy_singleIteration}]
The proof consists of two parts. First, we show that the next error is bounded as
\begin{align}\label{eq:sens_noisy_et+1_bound}
e_{t+1} &\leq \frac{1}{2(\sqrt 2 - 1)} \sqrt\frac{1+\delta}{1-\delta} \frac{e_t^2}{\sigma_r^*} + \frac 12 \sqrt\frac{1+\delta}{1-\delta} \|\Delta Z_t\|_F^2+ \frac{2}{\sqrt{1-\delta}}\|\xi\|
\end{align}
where $\Delta Z_t$ is the minimal norm solution to the least squares problem of \cref{eq:updatingVariant_LSQR}.
Second, we show that
\begin{align}\label{eq:sens_noisy_DeltaZt}
\|\Delta Z_t\|_F^2 \leq \frac{1+\delta}{1-\delta} \frac{e_t^2}{\sigma_r^*} + \frac{1}{2\sqrt{1+\delta}} \|\xi\|.
\end{align}
Combining these bounds gives
\begin{align*}
e_{t+1} &\leq \frac{1}{2(\sqrt 2 - 1)} \sqrt\frac{1+\delta}{1-\delta} \frac{e_t^2}{\sigma_r^*} + \frac 12 \sqrt\frac{1+\delta}{1-\delta} \left( \frac{1+\delta}{1-\delta} \frac{e_t^2}{\sigma_r^*} + \frac{1}{2\sqrt{1+\delta}} \|\xi\| \right) + \frac{2}{\sqrt{1-\delta}} \|\xi\| \\
&= \frac 12 \sqrt\frac{1+\delta}{1-\delta} \left(\frac{1}{\sqrt 2 - 1} + \frac{1+\delta}{1-\delta}\right) \frac{e_t^2}{\sigma_r^*} + \frac{9/4}{\sqrt{1-\delta}}\|\xi\|.
\end{align*}
Plugging the definitions $c=7(1+\delta)^\frac{3}{2}/(1-\delta)^\frac{3}{2}$ and $\gamma = c/(4\sigma_r^*)$ yields \cref{eq:sens_noisy_et+1_guarantee}, as
\begin{align*}
e_{t+1} &\leq \frac{c}{14} \left(\frac{1}{\sqrt 2 - 1} + 1\right) \frac{e_t^2}{\sigma_r^*} + \frac{9/4}{\sqrt{1-\delta}}\|\xi\| 
\leq \gamma e_t^2 + \frac{3}{\sqrt{1-\delta}}\|\xi\| .
\end{align*}

The proof of \cref{eq:sens_noisy_et+1_bound} follows the lines of the proof of \cref{lem:sens_et+1_bound}, but uses the Procrustes distance instead of the Q-distance.
Let $P$ be the minimizer of the Procrustes distance between $Z_t$ and $Z^*$, and denote $Z = \begin{psmallmatrix} U \\ V \end{psmallmatrix} = \begin{psmallmatrix} U^*P \\ V^*P \end{psmallmatrix}$ where $\begin{psmallmatrix} U^* \\ V^* \end{psmallmatrix} = \text{b-SVD}(X^*)$. Then, by \cref{lem:dP_e_bound}, $\Delta Z = Z_t - Z$ satisfies
\begin{align}
\|\Delta Z\|_F^2 \leq \frac{1}{(\sqrt 2 - 1)\sigma_r^*} \left(e_t^2 + \frac 14 \|U_t^\top U_t - V_t^\top V_t\|_F^2 \right) = \frac{e_t^2}{(\sqrt 2 - 1)\sigma_r^*} , \label{eq:sens_noisy_dP_bound}
\end{align}
where the equality follows by the fact that $U_t, V_t$ are balanced due to the additional SVD step of the algorithm, see \cref{eq:bSVD_is_balanced} of \cref{lem:bSVD_properties}.
Let $F_t^2$ be the objective function of the least squares problem. Denote $\begin{psmallmatrix} \Delta U \\ \Delta V \end{psmallmatrix} = \Delta Z$. Using $UV^\top = X^*$,
\begin{align*}
F_t\left(\Delta Z\right) &= \left\| \mathcal A\left(U_tV_t^\top + U_t \Delta V^\top + \Delta U V_t^\top \right) - b \right\| \\
&= \left\| \mathcal A\left(U_tV_t^\top + U_t \Delta V^\top + \Delta U V_t^\top - UV^\top \right) - \xi \right\|.
\end{align*}
Since $U_tV_t + U_t \Delta V^{\top} + \Delta U V_t^\top - UV^{\top} = -\Delta U \Delta V^{\top}$,
\begin{align}\label{eq:noisy_Ft_temp_bound}
F_t(\Delta Z) &= \|\mathcal A\left(\Delta U \Delta V^{\top} \right) + \xi\| 
\leq \|\mathcal A\left(\Delta U \Delta V^{\top}\right)\| + \|\xi\|
\leq \sqrt{1+\delta_r} \|\Delta U \Delta V^\top\|_F + \|\xi\|, 
\end{align}
where the last inequality follows by the fact that $\mathcal A$ satisfies an r-RIP with a constant $\delta_r$. By the Cauchy-Schwarz inequality and the fact that $ab \leq (a^2+b^2)/2$, we have $\|\Delta U \Delta V^\top\|_F \leq (\|\Delta U\|_F^2+\|\Delta V\|_F^2)/2 = \|\Delta Z\|_F^2/2$. Combining this with \cref{eq:sens_noisy_dP_bound,eq:noisy_Ft_temp_bound} yields
\begin{equation}\label{eq:noisy_Ft_upper_bound}
F_t(\Delta Z) \leq \frac{\sqrt{1+\delta_r}}{2(\sqrt 2 - 1)\sigma_r^*}e_t^2 + \|\xi\| .
\end{equation}

Next, we lower bound $F_t$ at the minimal norm solution $\Delta Z_t$. Similar to the proof of \cref{eq:sens_Ft_lowerBound}, any feasible solution to the least squares problem, including $\Delta Z_t$, satisfies
\begin{align*}
F_t(\Delta Z_t) &= \|\mathcal A\left(U_tV_t + U_t \Delta V_t^\top + \Delta U_t V_t^\top - X^*\right) - \xi\|
= \|\mathcal A\left(X_{t+1} - \Delta U_t \Delta V_t^\top - X^*\right) - \xi\| \\
&\geq \|\mathcal A(X_{t+1} - X^*)\| - \|\mathcal A(\Delta U_t \Delta V_t^\top)\| - \|\xi\| .
\end{align*}
Using again the RIP of $\mathcal A$, the Cauchy-Schwarz inequality and $ab \leq (a^2+b^2)/2$ yields
\begin{align}\label{eq:noisy_Ft_lowerBound}
F_t(\Delta Z_t) &\geq \sqrt{1-\delta_{2r}} e_{t+1} - \frac{\sqrt{1+\delta_r}}{2} \|\Delta Z_t\|_F^2 - \|\xi\| .
\end{align}
Since $\Delta Z_t$ minimizes $F_t$ by construction, in particular $F_t(\Delta Z_t) \leq F_t(\Delta Z)$. Hence, combining \cref{eq:noisy_Ft_upper_bound,eq:noisy_Ft_lowerBound} with the assumption $\delta_{2r} \leq \delta$ yields \cref{eq:sens_noisy_et+1_bound}.

Next, we prove \cref{eq:sens_noisy_DeltaZt}.
By tracing the proof of \cref{lem:minNormSol_normBound}, it is easy to verify that in the noisy case,
\begin{align}\label{eq:sens_noisy_DeltaZt_temp}
\|\Delta Z_t\|_F^2 &\leq \frac{1+\delta_{2r}}{1-\delta_{2r}} \frac{(e_t+\|\xi\|)^2}{\min\{\sigma_r^2(U_t), \sigma_r^2(V_t)\}} \leq \frac{1+\delta}{1-\delta} \frac{(e_t+\|\xi\|)^2}{\min\{\sigma_r^2(U_t), \sigma_r^2(V_t)\}} .
\end{align}
By combining \cref{prop:sigmaMin_U} with \cref{eq:sens_noisy_dP_bound} we obtain
\begin{align*}
\sigma_r(U_t) \geq \sigma_r(U^*) - d_P(U_t, U^*) \geq \sqrt{\sigma_r^*} - \|\Delta Z\|_F \geq \sqrt{\sigma_r^*} - \frac{e_t}{\sqrt{(\sqrt 2 - 1)\sigma_r^*}} .
\end{align*}
Employing assumptions \cref{eq:sens_noisy_xi_assumption,eq:sens_noisy_et_assumption} yields
\begin{align*}
\sigma_r(U_t) \geq \left(1 - \frac{1}{c} - \frac{3}{6c}\right) \frac{\sqrt{\sigma_r^*}}{\sqrt{\sqrt 2 - 1}} \geq \frac{\sqrt{\sigma_r^*}}{\sqrt{2(\sqrt 2 - 1)}} ,
\end{align*}
where the last inequality follows since $c = 7(1+\delta)^\frac{3}{2}/(1-\delta)^\frac{3}{2} \geq 7$.
Together with a similar bound on $\sigma_r(V_t)$ and employing again assumptions \cref{eq:sens_noisy_xi_assumption,eq:sens_noisy_et_assumption}, \cref{eq:sens_noisy_DeltaZt_temp} gives that
\begin{align*}
\|\Delta Z_t\|_F^2 &\leq 2(\sqrt 2 - 1) \frac{1+\delta}{1-\delta} \frac{(e_t+\|\xi\|)^2}{\sigma_r^*} = 2(\sqrt 2 - 1) \frac{1+\delta}{1-\delta} \left( \frac{e_t^2}{\sigma_r^*} + \frac{2e_t + \|\xi\|}{\sigma_r^*}\|\xi\| \right) \\
&\leq 2(\sqrt 2 - 1) \frac{1+\delta}{1-\delta} \left( \frac{e_t^2}{\sigma_r^*} + \frac{2\cdot (1+3/6) + 1/6}{c}\|\xi\| \right) \leq \frac{1+\delta}{1-\delta} \frac{e_t^2}{\sigma_r^*} + \frac{1}{2\sqrt{1+\delta}} \|\xi\| .
\end{align*}
This completes the proof of \cref{eq:sens_noisy_DeltaZt}.
\end{proof}

\section{Proof of \cref{thm:RIP}\label{sec:RIP_proof} (uniform RIP for matrix completion)}
Before we present the proofs for the matrix completion setting, we make the following two remarks.
\begin{remark}\label{rem:probability_model}
Our results for matrix completion (\cref{thm:comp_linearConvergence,thm:comp_quadConvergence,thm:RIP}) are stated with respect to a uniform random model of the sampling pattern, namely $\Omega$ is drawn uniformly from $2^{[n_1]\times [n_2]}$ with fixed size $|\Omega|$.
Similar to previous works, our analysis assumes the more convenient Bernoulli model, in which each entry of $X^*$ is observed with probability $p$.
These models are equivalent in the sense that under the Bernoulli model, $pn_1n_2 - C \sqrt{n_2\log n_2} \leq |\Omega| \leq pn_1n_2 + C \sqrt{n_2\log n_2}$ with probability~$1-1/n^{10}$ where $C$ is a constant \cite[section I.D]{keshavan2010matrix}.
Consequently, a result that holds w.p.~$1-\mathcal O(1/n^c)$ for some $c\leq 10$ under the uniform random model with a certain $|\Omega|$, holds with similar probability under the Bernoulli model with $p \equiv \frac{|\Omega|}{n_1 n_2}$.
\end{remark}

\begin{remark}\label{rem:absolute_constants}
In our results, if an argument holds for some constants (such as $C, c_e, c_l$, etc.), then it also holds for larger constants.
As a consequence, if, for example, Lemma A claims that there exists a constant $c_1$ such that argument A$(c_1)$ holds, and Lemma B claims that there exists an a constant $c_2$ such that argument B$(c_2)$ holds, then there exists a constant $c_3$ such that arguments A$(c_3) \land$B$(c_3)$ hold, as we can always choose $c_3 \geq \max\{c_1,c_2\}$.
\end{remark}

The proof of \cref{thm:RIP} relies on the following two technical lemmas.
The first lemma provides bounds on two distance measures between $Z \in \mathcal B_\textnormal{err}({\epsilon}/{c_e}) \cap \mathcal B_\textnormal{bln}({1}/{c_l}) \cap \mathcal B_\mu$ and a nearby $Z^* \in \mathcal \mathcal B^* \cap \mathcal B_\mu$.
The second lemma states that theses bounds are sufficient for the RIP \cref{eq:RIP} to hold.
These lemmas are also used in the proof of \cref{thm:comp_linearConvergence,thm:comp_quadConvergence}.
\begin{lemma}\label{lem:delta_bounds}
There exist constants $C, c_l, c_e$ such that the following holds.
Let $X^* \in \mathcal M(n_1, n_2, r, \mu, \kappa)$ and $\epsilon \in (0,1)$, and assume $\Omega \subseteq [n_1]\times [n_2]$ is randomly sampled with $np \geq C \max\{\log n, {\mu^2 r^2 \kappa^2}/{\epsilon^4}\}$.
Then w.p.~at least $1 - {2}/{n^5}$, for any $\begin{psmallmatrix} U \\ V \end{psmallmatrix} \in \mathcal B_\textnormal{err}({\epsilon}/{c_e}) \cap \mathcal B_\textnormal{bln}({1}/{c_l}) \cap \mathcal B_\mu$ with $X = UV^\top$ there exists $\begin{psmallmatrix} U^* \\ V^* \end{psmallmatrix} \in \mathcal B^* \cap \mathcal B_\mu$ such that
\begin{subequations}\label{eq:delta_bounds}\begin{align}
\|U-U^*\|^2_F + \|V-V^*\|^2_F &\leq \frac{\epsilon}{2} \|X-X^*\|_F, \label{eq:delta_F_bound} \\
\frac{1}{\sqrt p}\|(U - U^*)(V - V^*)^\top\|_{F(\Omega)} &\leq \frac{\epsilon}{6} \|X-X^*\|_F . \label{eq:deltaProduct_FOmega_bound}
\end{align}\end{subequations}
\end{lemma}

In what follows, we denote by $\mathcal M(n_1, n_2, r, \mu)$ the set of $n_1\times n_2$ $\mu$-incoherent matrices of rank $r$, without specifying the condition number.
\begin{lemma}\label{lem:RIP_fromDeltaProductBounds}
There exists a constant $C$ such that the following holds.
Let $X^* \in \mathcal M(n_1, n_2, r, \mu)$ and $\epsilon \in (0,1)$, and assume $\Omega \subseteq [n_1]\times [n_2]$ is randomly sampled with $np \geq \frac{C}{\epsilon^2}\mu r \log n$.
Then w.p.~at least $1 - {3}/{n^3}$, for any $\begin{psmallmatrix} U \\ V \end{psmallmatrix} \in \mathbb R^{(n_1+n_2)\times r}$ for which there exists $\begin{psmallmatrix} U^* \\ V^* \end{psmallmatrix} \in \mathcal B^*$ that satisfies \cref{eq:delta_bounds}, the RIP \cref{eq:RIP} holds w.r.t.~$X = UV^\top$.
\end{lemma}

\begin{proof}[Proof of \cref{thm:RIP}]
By \cref{lem:delta_bounds}, there exists $\begin{psmallmatrix} U^* \\ V^* \end{psmallmatrix} \in \mathcal B^*$ such that \cref{eq:delta_bounds} holds.
The theorem thus follows by \cref{lem:RIP_fromDeltaProductBounds}.
\end{proof}

\subsection*{Proofs of \cref{lem:delta_bounds,lem:RIP_fromDeltaProductBounds}}
Let us begin with two auxiliary lemmas. 
The first lemma provides a (deterministic) bound on the Frobenius distance between $Z \in \mathcal B_\textnormal{err}({1}/{c_e}) \cap \mathcal B_\textnormal{bln}({1}/{c_l})$ and a nearby $Z^* \in \mathcal B^* \cap \mathcal B_\mu$.

\begin{lemma}\label{lem:deterministic_nearby_delta_bound}
There exist constants $c_l, c_e$ such that the following holds.
Let $X^* \in \mathcal M(n_1, n_2, r, \mu)$.
Then for any $\begin{psmallmatrix} U \\ V \end{psmallmatrix} \in \mathcal B_\textnormal{err}({1}/{c_e}) \cap \mathcal B_\textnormal{bln}({1}/{c_l})$ there exists $\begin{psmallmatrix} U^* \\ V^* \end{psmallmatrix} \in \mathcal B^* \cap \mathcal B_\mu$ such that
\begin{align}
\|U - U^*\|_F^2 + \|V - V^*\|_F^2 &\leq \frac{25}{4 \sigma_r^*} \|UV^\top - X^*\|_F^2 . \label{eq:deterministic_nearby_delta_bound}
\end{align}
\end{lemma}

The second lemma is a direct consequence of \cite[Lemma~7.1]{keshavan2010matrix}.
\begin{lemma}\label{lem:KMO10_lemma71_consequence}
There exist constants $C,c$ such that the following holds for any $\mu, t, \epsilon > 0$.
Assume $\Omega \subseteq [n_1]\times [n_2]$ is randomly sampled with $np \geq C \max\{\log n, \mu^2 r^2/\epsilon^4\}$. Then w.p.~at least $1 - {2}/{n^5}$, for any $\begin{psmallmatrix} U \\ V \end{psmallmatrix} \in \mathbb R^{(n_1+n_2)\times r}$ such that
\begin{align}\label{eq:KMO10_lemma71_consequence_assumption}
\|U\|_{2,\infty} \leq 4\sqrt{\mu r t/n_1}, \quad
\|V\|_{2,\infty} \leq 4\sqrt{\mu r t/n_2},
\end{align}
we have
\begin{align}\label{eq:KMO10_lemma71_consequence}
\frac 1p \|U V^\top\|^2_{F(\Omega)} \leq \frac{\|U\|_F^2 + \|V\|_F^2}{2} \left( c \left(\|U\|_F^2 + \|V\|_F^2\right) + t \epsilon^2 \right) .
\end{align}
\end{lemma}

\begin{proof}[Proof of \cref{lem:delta_bounds}]
Given $Z = \begin{psmallmatrix} U \\ V \end{psmallmatrix}$, let $Z^* = \begin{psmallmatrix} U^* \\ V^* \end{psmallmatrix} \in \mathcal B^* \cap \mathcal B_\mu$ be the corresponding factor matrices from \cref{lem:deterministic_nearby_delta_bound}.
We shall prove that $Z^*$ satisfies \cref{eq:delta_bounds} w.p.~at least $1 - {1}/{n^5}$.
First, by $Z \in \mathcal B_\textnormal{err}({\epsilon}/{c_e})$ we have $\|X - X^*\|_F \leq {\epsilon \sigma_r^*}/{c_e}$.
Equation~\cref{eq:delta_F_bound} follows for large enough $c_e$ by combining this with \cref{eq:deterministic_nearby_delta_bound} of \cref{lem:deterministic_nearby_delta_bound}.

Next, we prove \cref{eq:deltaProduct_FOmega_bound}.
Since both $Z, Z^* \in \mathcal B_\mu$, the difference $\Delta U^* = U - U^*$ satisfies
\begin{align*}
\|\Delta U^*\|_{2,\infty} &\leq \|U\|_{2,\infty} + \|U^*\|_{2,\infty} \leq 2\sqrt{3\mu r \sigma_1^*/n_1},
\end{align*}
and similarly $\|\Delta V^*\|_{2,\infty} \leq 2\sqrt{3\mu r \sigma_1^*/n_2}$ where $\Delta V^* = V - V^*$.
Invoking \cref{lem:KMO10_lemma71_consequence} with $t\to \sigma_1^*$ and $\epsilon \to \epsilon/(11\sqrt\kappa)$ yields
\begin{align}\label{eq:deltaProduct_p_boundTemp}
\frac 1p \|\Delta U^* \Delta V^{*\top}\|^2_{F(\Omega)} &\leq \frac{\|\Delta U^*\|_F^2 + \|\Delta V^*\|_F^2}{2} \left(c\left(\|\Delta U^*\|_F^2 + \|\Delta V^*\|_F^2\right) + \frac{\epsilon^2 \sigma_r^*}{121} \right) .
\end{align}
Next, we bound $\|\Delta U^*\|_F^2 + \|\Delta V^*\|_F^2$. To this end, recall that $Z^*$ are the factor matrices given by \cref{lem:deterministic_nearby_delta_bound}. Combining \cref{eq:deterministic_nearby_delta_bound} and the assumption $Z \in \mathcal B_\textnormal{err}({\epsilon}/{c_e})$ gives
\begin{align*}
\|\Delta U^*\|_F^2 + \|\Delta V^*\|_F^2 &\leq \frac{25}{8\sigma_r^*}\|UV^\top - X^*\|_F^2 \leq \frac{25\epsilon^2 \sigma_r^*}{8c_e^2}.
\end{align*}
By plugging this result into \cref{eq:deltaProduct_p_boundTemp} we obtain
\begin{align*}
\frac 1p \|\Delta U^* \Delta V^{*\top}\|^2_{F(\Omega)} &\leq \frac{25}{16\sigma_r^*} \|UV^\top - X^*\|_F^2 \left(\frac{25c\epsilon^2 \sigma_r^*}{8c_e^2} + \frac{\epsilon^2 \sigma_r^*}{120}\right) \\
&= \frac{25}{64} \left(\frac{25c}{c_e^2} + \frac{1}{15}\right)\epsilon^2 \|UV^\top - X^*\|_F^2,
\end{align*}
from which \cref{eq:deltaProduct_FOmega_bound} follows for large enough $c_e$.
\end{proof}

To prove \cref{lem:RIP_fromDeltaProductBounds} we shall use the following auxiliary result (see \cite[Lemma~9]{yi2016fast}, \cite[Lemma~10]{zheng2016convergence}), based on \cite[Theorem~4.1]{candes2009exact}. The lemma exploits the fact that $\Omega$ is random and independent of $X^*$ to \textit{uniformly} bound w.h.p.~the $F(\Omega)$-magnitude of matrices with the same row space and column space as $X^*$.
\begin{lemma}\label{lem:YPCC16_lemma9}
There exists a constant $C$ such that the following holds.
Let $X^* \in \mathcal M(n_1, n_2, r, \mu)$, and denote its SVD $X^* =  \bar U \Sigma \bar V^{\top}$.
Define the subspace $\mathcal T \subset \mathbb R^{n_1\times n_2}$ as
\begin{align*}
\mathcal T = \left\{\bar U V^\top + U \bar V^\top \,\mid\, U\in \mathbb R^{n_1\times r}, V \in \mathbb R^{n_2\times r} \right\}.
\end{align*}
Let $\epsilon \in (0,1)$, and assume $\Omega \subseteq [n_1]\times [n_2]$ is randomly sampled with $np \geq \frac{C}{\epsilon^2} \mu r \log n$.
Then w.p.~at least $1 - 2/n^3$, any $Z \in \mathcal T$ satisfies the RIP
\begin{align*}\begin{aligned}
(1-\epsilon) \|Z\|_F^2 &\leq \frac 1p \|Z\|^2_{F(\Omega)} \leq (1+\epsilon) \|Z\|^2_F .
\end{aligned}\end{align*}
\end{lemma}

\begin{proof}[Proof of \cref{lem:RIP_fromDeltaProductBounds}]
Let $Z = \begin{psmallmatrix} U \\ V \end{psmallmatrix} \in \mathbb R^{(n_1+n_2)\times r}$ be such that there exists $Z^* = \begin{psmallmatrix} U^* \\ V^* \end{psmallmatrix} \in \mathcal B^*$ that satisfies \cref{eq:delta_bounds} w.r.t.~$X = UV^\top$.
To bound the norm of $X - X^*$, we decompose it as $X - X^* = A + B$ where
\begin{align*}\begin{aligned}
A &= U^* (V - V^*)^\top + (U - U^*)V^{*\top}, \quad
B = (U - U^*)(V - V^*)^\top.
\end{aligned}\end{align*}
Let $e_F = \|X - X^*\|_F$ and $e_{F(\Omega)} = \|X - X^*\|_{F(\Omega)}$.
By \cref{eq:deltaProduct_FOmega_bound}, $\frac{1}{\sqrt p}\|B\|_{F(\Omega)} \leq \frac{\epsilon}{2}e_F$. Since $e_{F(\Omega)} = \|A+B\|_{F(\Omega)}$, this implies
\begin{align*}
\frac{1}{\sqrt p}\|A\|_{F(\Omega)} - \frac{\epsilon}{2}e_F \leq \frac{1}{\sqrt p}e_{F(\Omega)} \leq  \frac{1}{\sqrt p}\|A\|_{F(\Omega)} + \frac{\epsilon}{2}e_F .
\end{align*}
Hence, the RIP \cref{eq:RIP} will follow from the bounds
\begin{align}
\left(1 - \frac{\epsilon}{2}\right) e_F &\leq \frac{1}{\sqrt p} \|A\|_{F(\Omega)} \leq \left(1 + \frac{\epsilon}{2}\right) e_F. \label{eq:a_bounds}
\end{align}
We prove \cref{eq:a_bounds} by first bounding $\|A\|_F$, and then invoking \cref{lem:YPCC16_lemma9} to bound $\|A\|_{F(\Omega)}$.
Since $e_F = \|A + B\|_F$,
\begin{align}
e_F - \|B\|_F \leq \|A\|_F \leq e_F + \|B\|_F. \label{eq:aF_eF_b_bound}
\end{align}
Combining the Cauchy-Schwarz inequality, the fact $ab \leq (a^2+b^2)/2$ and \cref{eq:delta_F_bound} gives that
\begin{align*}
\|B\|_F &\leq \|U-U^*\|_F \|V-V^*\|_F \leq \frac 12 \left(\|U-U^*\|_F^2 + \|V-V^*\|_F^2 \right) \leq \frac{\epsilon}{4} e_F.
\end{align*}
Plugging this into \cref{eq:aF_eF_b_bound} yields
\begin{align*}
\left(1 - \frac{\epsilon}{4}\right) e_F \leq \|A\|_F \leq \left(1 + \frac{\epsilon}{4}\right)e_F.
\end{align*}
Since $\epsilon \in (0,1)$, it is easy to verify that
$
\frac{1 - \epsilon/2}{1 - \epsilon/6}
\leq 1 - \frac{\epsilon}{4}
$
and
$
\frac{1 + \epsilon/2}{1 + \epsilon/6}
\geq 1 + \frac{\epsilon}{4}.
$ Hence
\begin{align*}
\frac{1 - \epsilon/2}{1 - \epsilon/6} e_F \leq \|A\|_F \leq \frac{1 + \epsilon/2}{1 + \epsilon/6} e_F.
\end{align*}
In addition, since $A \in \mathcal T$, invoking \cref{lem:YPCC16_lemma9} implies that for large enough $C$,
\begin{align*}
\left(1 - \frac{\epsilon}{6}\right)\|A\|_F \leq \frac{1}{\sqrt p}\|A\|_{F(\Omega)} \leq \left(1 + \frac{\epsilon}{6}\right) \|A\|_F .
\end{align*}
Combining the last two equations yields \cref{eq:a_bounds}.
\end{proof}

\subsection*{Proofs of auxiliary \cref{lem:deterministic_nearby_delta_bound,lem:KMO10_lemma71_consequence}\nopunct}
\begin{proof}[Proof of \cref{lem:deterministic_nearby_delta_bound}]
Let $\tilde Z = \begin{psmallmatrix} \tilde U \\ \tilde V \end{psmallmatrix} = \text{b-SVD}(X^*)$.
By the assumption $\begin{psmallmatrix} U \\ V \end{psmallmatrix} \in \mathcal B_\textnormal{err}(\frac{1}{c_e})$ $\cap \mathcal B_\textnormal{bln}(\frac{1}{c_l})$ we have $\|UV^\top - X^*\|_F \leq {\sigma_r^*}/{c_e}$ and $\|U^\top U - V^\top V\|_F \leq {\sigma_r^*}/{c_l}$.
Hence, for large enough constants $c_e, c_l$, condition \cref{eq:dQ_e_bound_assumption} of \cref{lem:dQ_e_bound} holds, which implies the existence of an invertible matrix $Q\in \mathbb R^{r\times r}$ that satisfies $\|Q\|_2 \leq  4/3$ and
\begin{align}\label{eq:delta_temp_bound}
\|U - \tilde U Q\|_F^2 + \|V - \tilde V Q^{-\top}\|_F^2 \leq \frac{25}{4\sigma_r^*} \|UV^\top - X^*\|_F .
\end{align}
Let $Z^* = \begin{psmallmatrix} U^* \\ V^* \end{psmallmatrix} = \begin{psmallmatrix} \tilde U Q \\ \tilde V Q^{-\top} \end{psmallmatrix}$.
Clearly $Z^* \in \mathcal B^*$, and \cref{eq:deterministic_nearby_delta_bound} of the lemma follows from \cref{eq:delta_temp_bound}. It thus remains to show that $Z^* \in \mathcal B_\mu$.
Combining $\|AB\|_{2,\infty} \leq \|A\|_{2,\infty} \|B\|_2$, the bound $\|Q\|_2 \leq 4/3$ and \cref{eq:bSVD_row_norms} of \cref{lem:bSVD_properties} gives
\begin{align*}
\|U^*\|_{2,\infty} &= \|\tilde U Q\|_{2,\infty} \leq \frac 43 \sqrt\frac{\mu r \sigma_1(X^*)}{n_1}, 
\end{align*}
and similarly  $\|V^*\|_{2,\infty} \leq \frac 43 \sqrt\frac{\mu r \sigma_1(X^*)}{n_2}$.
This completes the proof.
\end{proof}

To prove \cref{lem:KMO10_lemma71_consequence}, we shall use the following result \cite[Lemma~7.1]{keshavan2010matrix}, based on a random graph lemma due to Feige and Ofek \cite{feige2005spectral}.
\begin{lemma}\label{lem:KMO10_lemma71}
There exist constants $C, c$ such that the following holds.
Assume $\Omega \subseteq [n_1]\times [n_2]$ is randomly sampled with $np \geq C \log n$. Then w.p.~at least $1 - {1}/{n^5}$, for any $x \in \mathbb R^{n_1}, y \in \mathbb R^{n_2}$,
\begin{align*} 
\sum_{(i,j)\in \Omega} x_i y_j \leq c \left( p \|x\|_1 \|y\|_1 + \sqrt{np} \|x\|_2\|y\|_2 \right) .
\end{align*}
\end{lemma}

\begin{proof}[Proof of \cref{lem:KMO10_lemma71_consequence}]
Define two vectors $x\in \mathbb R^{n_1}$ and $y\in \mathbb R^{n_2}$ as follows:
For $i\in [n_1]$ and $j\in [n_2]$, let $x_i = \|U^{(i)}\|^2$ and $y_j = \|V^{(j)}\|^2$ where $U^{(i)}$ denotes the $i$-th row of $U$.
By the Cauchy-Schwarz inequality and \cref{lem:KMO10_lemma71},
\begin{align*}
\frac 1p \|U V^\top\|^2_{F(\Omega)} &= \frac 1p \sum_{(i,j)\in \Omega} [UV^\top]_{ij}^2 \leq \frac 1p \sum_{(i,j)\in \Omega} x_i y_j \leq c \left(\|x\|_1 \|y\|_1 + \sqrt\frac{n}{p} \|x\|_2\|y\|_2 \right) .
\end{align*}
Let us bound each of the two terms on the RHS.
First, since $ab \leq (a^2 + b^2)/2$, then
\begin{align*}
\|x\|_1 \|y\|_1 &= \|U\|_F^2 \|V\|_F^2 \leq \frac 12 \left(\|U\|_F^2 + \|V\|_F^2\right)^2 .
\end{align*}
Next, let us bound $\|x\|_2\|y\|_2$. Observe that
\begin{align*}
\|x\|_2^2 &= \sum_i \|U^{(i)}\|^4 \leq \max_i \|U^{(i)}\|^2 \sum_i \|U^{(i)}\|^2 = \|U\|_{2,\infty}^2 \|U\|_F^2,
\end{align*}
and similarly $\|y\|_2^2 \leq \|V\|_{2,\infty}^2 \|V\|_F^2$.
Again by the inequality $ab \leq (a^2 + b^2)/2$ we obtain
\begin{align*}
\|x\|_2 \|y\|_2 &\leq \frac 12 \|U\|_{2,\infty} \|V\|_{2,\infty} \left(\|U\|_F^2 + \|V\|_F^2\right).
\end{align*}
In addition, by the assumptions $np \geq C\mu^2 r^2/\epsilon^4$ and \cref{eq:KMO10_lemma71_consequence_assumption},
\begin{align*}
\sqrt\frac{n}{p} \leq \frac{n \epsilon^2}{\sqrt{C} \mu r} \leq \frac{16n}{\sqrt{C n_1n_2}} \frac{t \epsilon^2}{\|U\|_{2,\infty} \|V\|_{2,\infty}} \leq \frac{1}{c} \frac{t \epsilon^2}{\|U\|_{2,\infty} \|V\|_{2,\infty}} ,
\end{align*}
where the last inequality holds for large enough $C$.
Putting everything together completes the proof.
\end{proof}

\subsection*{Proof of \cref{corollary:RIP_incoherent}\nopunct}
\begin{proof}
Let $Z = \text{b-SVD}(X)$.
In view of \cref{thm:RIP}, it suffices to show that $Z \in \mathcal B_\textnormal{err}({\epsilon}/{c_e}) \cap \mathcal B_\textnormal{bln}({1}/{c_l}) \cap \mathcal B_\mu$.
The first condition, $Z \in \mathcal B_\textnormal{err}({\epsilon}/{c_e})$, follows from the assumption $\|X - X^*\|_F \leq {\epsilon \sigma_r^*}/{c_e}$.
Next, since $Z$ is a b-SVD of a matrix, it is perfectly balanced, see \cref{eq:bSVD_is_balanced} of \cref{lem:bSVD_properties}. Hence $Z \in \mathcal B_\textnormal{bln}({1}/{c_l})$ for any $c_l>0$.
It is thus left to show that $Z \in \mathcal B_\mu$.
Since $X$ is ${3\mu}/{2}$-incoherent, by \cref{eq:bSVD_row_norms} of \cref{lem:bSVD_properties} we have
\begin{align}\label{eq:comp_UV2infty_temp_bound}
\|U\|_{2,\infty} \leq \sqrt{3\mu r \sigma_1(X)/(2n_1)}, \quad
\|V\|_{2,\infty} \leq \sqrt{3\mu r \sigma_1(X)/(2n_2)}.
\end{align}
Let us bound $\sigma_1(X)$. W.l.o.g.~we can assume that $c_e \geq 1$ (see \cref{rem:absolute_constants}). Thus, by assumption, $\|X - X^*\|_2 \leq \|X-X^*\|_F \leq {\epsilon \sigma_r^*}/{c_e} \leq \sigma_1^*$.
Hence
\begin{align*}
\sigma_1(X) &\leq \sigma_1^* + \|X - X^*\|_2 \leq 2\sigma_1^* .
\end{align*}
Together with \cref{eq:comp_UV2infty_temp_bound} we obtain $\|U\|_{2,\infty} \leq \sqrt{3\mu r \sigma_1^*/n_1}$, and similarly $\|V\|_{2,\infty} \leq \sqrt{3\mu r \sigma^*/n_2}$, which implies $Z \in \mathcal B_\mu$.
\end{proof}

\section{Proof of \cref{thm:comp_linearConvergence}\label{sec:comp_linearConvergence_proof} (matrix completion, linear convergence)}
The proof of \cref{thm:comp_linearConvergence} relies on the property that the iterates $U_t, V_t$ remain approximately balanced. In particular, their largest singular value remains bounded. To this end, we introduce the following definition of subset of factor matrices with bounded \textit{largest singular value},
\begin{align}\label{eq:B_s_def}
\mathcal B_\textnormal{lsv}(\nu) &= \left\{\begin{pmatrix} U \\ V \end{pmatrix} \in \mathbb R^{(n_1+n_2)\times r} \,\mid\, \max\{\sigma_1(U), \sigma_1(V)\} \leq \nu \sqrt{\sigma_1^*} \right\} .
\end{align}

Denote the current and next iterates of \cref{alg:reg_GNMR} by $Z_t = \begin{psmallmatrix} U_t \\ V_t \end{psmallmatrix}$ and $Z_{t+1} = \begin{psmallmatrix} U_{t+1} \\ V_{t+1} \end{psmallmatrix}$, respectively, and let $X_t = U_tV_t^\top$ and $X_{t+1} = U_{t+1}V_{t+1}^\top$ be their corresponding estimates.
The following lemma states that the estimation error contracts geometrically at each iteration, while the iterates remain balanced, with bounded row norms, and with bounded largest singular value.

\begin{lemma}\label{lem:comp_linearConvergence_singleIteration}
There exist constants $C, c_e, c_l$ such that the following holds.
Let $X^* \in \mathcal M(n_1, n_2, r, \mu, \kappa)$.
Define
\begin{align}
\epsilon_t &= \frac{1}{2^t c_e \sqrt\kappa}, \quad
\delta_t = \frac{1}{c_l} + \frac{36(1-2^{-t})}{c_e}, \quad
\nu_t = 2 + \frac{6(1-2^{-t})}{c_e} , \label{eq:epsilon_delta_xi_t}
\end{align}
and
\begin{align*}
\mathcal B(t) = \mathcal B_\textnormal{err}\left(\epsilon_t\right) \cap \mathcal B_\textnormal{bln}\left(\delta_t\right) \cap \mathcal B_\mu \cap \mathcal B_\text{lsv}\left(\nu_t\right).
\end{align*}
Assume $\Omega \subseteq [n_1]\times [n_2]$ is randomly sampled with $np \geq C \mu r \max\{\log n, \mu r \kappa^2\}$.
Further assume that at some iteration $t$, 
\begin{align}
Z_t \in \mathcal B(t). \label{eq:comp_linearConvergence_singleIteration_assumption_B}
\end{align}
Then w.p.~at least $1 - {3}/{n^3}$, for all iterates $t' \geq t$,
\begin{subequations}\label{eq:comp_linearConvergence_singleIteration_results}\begin{align}
Z_{t'+1} &\in\, \mathcal B(t'+1), \label{eq:comp_linearConvergence_singleIteration_result_B} \\
\|X_{t'+1} - X^*\|_F &\leq \tfrac 12 \|X_{t'} - X^*\|_F . \label{eq:comp_linearConvergence_singleIteration_result_e}
\end{align}\end{subequations}
\end{lemma}

\begin{proof}[Proof of \cref{thm:comp_linearConvergence}]
Let $Z_0 = \begin{psmallmatrix} U_0 \\ V_0 \end{psmallmatrix}$ be an initial guess which satisfies the conditions of the theorem.
Let us show that it satisfies conditions \cref{eq:comp_linearConvergence_singleIteration_assumption_B} of \cref{lem:comp_linearConvergence_singleIteration} at $t=0$.
Since $Z_0 \in \mathcal B_\textnormal{err}({1}/{(c_e\sqrt\kappa)}) \cap \mathcal B_\textnormal{bln}({1}/{c_l}) \cap \mathcal B_\mu$, we only need to show that $Z_0 \in \mathcal B_\textnormal{lsv}(\nu_0)$ with $\nu_0 = 2$.
Let $Z^* = \text{b-SVD}(X^*)$.
We then have
\begin{align*}
d_P^2(Z_0, Z^*) &\stackrel{(a)}{\leq} \frac{1}{(\sqrt 2 - 1) \sigma_r^*} \left(\|U_0V_0^\top - X^*\|_F^2 + \frac 14 \|U_0^\top U_0 - V_0^\top V_0\|_F\right) \\ 
&\stackrel{(b)}{\leq} \frac{1}{(\sqrt 2 - 1) \sigma_r^*} \left( \frac{\sigma_r^{*2}}{c_e^2 \kappa} + \frac{\sigma_r^{*2}}{4c_l^2} \right) 
= \frac{1}{(\sqrt 2 - 1) \kappa} \left( \frac{1}{c_e^2 \kappa} + \frac{1}{4c_l^2} \right) \sigma_1^* \stackrel{(c)}{\leq} \sigma_1^*,
\end{align*}
where (a) follows by \cref{lem:dP_e_bound}, (b) by the assumption $Z_0 \in \mathcal B_\textnormal{err}({1}/{(c_e \sqrt\kappa)}) \cap \mathcal B_\textnormal{bln}({1}/{c_l})$, and (c) follows for large enough $c_e, c_l$.
In addition, $d_P(U_0, U^*) \leq d_P(Z_0, Z^*)$ by definition.
\Cref{prop:sigmaMin_U} thus implies
\begin{align*}
\sigma_1(U_0) &\leq \sigma_1(U^*) + d_P(Z_0, Z^*) \leq 2\sqrt{\sigma_1^*}.
\end{align*}
Similarly, $\sigma_1(V_0) \leq 2\sqrt{\sigma_1^*}$. Together we obtain $Z_0 \in \mathcal B_\textnormal{lsv}(\nu_0)$, which completes the proof of \cref{eq:comp_linearConvergence_singleIteration_assumption_B} at $t=0$.
The theorem now follows by applying \cref{lem:comp_linearConvergence_singleIteration} at $t=0$.
\end{proof}

\subsection*{Proof of \cref{lem:comp_linearConvergence_singleIteration}}
To prove the lemma we shall use the following auxiliary result, which is a partial (deterministic) version of \cref{lem:comp_linearConvergence_singleIteration}.
It is presented as a separate lemma as it is also used in the context of quadratic convergence.

\begin{lemma}\label{lem:comp_common_singleIteration}
Let $X^* \in \mathcal M(n_1, n_2, r, \mu, \kappa)$.
Denote the current and next iterates of \cref{alg:reg_GNMR} by $Z_t = \begin{psmallmatrix} U_t \\ V_t \end{psmallmatrix}$ and $Z_{t+1} = \begin{psmallmatrix} U_{t+1} \\ V_{t+1} \end{psmallmatrix}$, respectively.
Let $\Omega \subseteq [n_1]\times [n_2]$ be such that the current iterate satisfies an RIP under $\mathcal P_\Omega$,
\begin{align}
\frac 78 e_t^2 &\leq \frac 1p e_{\Omega,t}^2 \leq \frac 98 e_t^2, \label{eq:comp_common_Xt_RIP}
\end{align}
where $e_t = \|U_tV_t^\top - X^*\|_F$ and $e_{\Omega,t} = \|U_tV_t^\top - X^*\|_{F(\Omega)}$ are the true and observed errors at iteration $t$, respectively.
Further, assume the current iterate satisfies \cref{eq:comp_linearConvergence_singleIteration_assumption_B} with constants $\epsilon_t, \delta_t, \nu_t$ given in \cref{eq:epsilon_delta_xi_t}.
Then the next iterate satisfies
\begin{subequations}\label{eq:comp_common_singleIteration_results}\begin{align}
Z_{t+1} \in \mathcal B_\textnormal{err}\left({10}/{(2^t c_e)}\right) \cap \mathcal B_\textnormal{bln}(\delta_{t+1}) \cap \mathcal B_\mu \cap \mathcal B_\textnormal{lsv}(\nu_{t+1}), \label{eq:comp_common_singleIteration_result_B} \\
\|U_{t+1} - U_t\|_F^2 + \|V_{t+1} - V_t\|_F^2 \leq \frac{9}{\sigma_r^*}e_t^2 \leq \frac{9 \sigma_r^*}{2^{2t} c_e^2 \kappa}. \label{eq:comp_common_singleItertion_result_Delta}
\end{align}\end{subequations}
In addition, if $\mathcal B^* \cap \mathcal B_\mu \cap \mathcal C^{(t)} \neq \emptyset$ where $\mathcal C^{(t)}$ is defined in \cref{eq:Ct_def}, then for any $\begin{psmallmatrix} U^* \\ V^* \end{psmallmatrix} \in \mathcal B^* \cap \mathcal B_\mu \cap \mathcal C^{(t)}$,
\begin{align}
e_{\Omega,t+1} &\leq \|(U^*-U_t)(V^*-V_t)^\top\|_{F(\Omega)} + \|(U_{t+1} - U_t)(V_{t+1} - V_t)^\top\|_{F(\Omega)} . \label{eq:comp_common_singleIteration_result_e}
\end{align}
\end{lemma}


\begin{proof}[Proof of \cref{lem:comp_linearConvergence_singleIteration}]
In the following, we prove that if a certain random event occurs, then \cref{eq:comp_linearConvergence_singleIteration_results} holds for $t'=t$.
Since this event does not depend on $t$ and occurs w.p.~at least $1 - {3}/{n^3}$, the lemma follows for any $t'\geq t$ by induction.

By assumption \cref{eq:comp_linearConvergence_singleIteration_assumption_B} for large enough $c_e$, $Z_t$ satisfies the conditions of \cref{thm:RIP} with $\epsilon = 1/8$. This guarantees an RIP for the current estimate \cref{eq:comp_common_Xt_RIP}.
In conjunction with \cref{eq:comp_linearConvergence_singleIteration_assumption_B}, the conditions of \cref{lem:comp_common_singleIteration} hold.
The next iterate $Z_{t+1}$ thus satisfies \cref{eq:comp_common_singleIteration_result_B} of \cref{lem:comp_common_singleIteration}.
Hence, in order to complete the proof, we need to show that $Z_{t+1} \in \mathcal B_\textnormal{err}(\epsilon_{t+1})$ and that \cref{eq:comp_linearConvergence_singleIteration_result_e} holds.
However, since $Z_t \in \mathcal B(\epsilon_t)$ by assumption \cref{eq:comp_linearConvergence_singleIteration_assumption_B}, the required $Z_{t+1} \in \mathcal B_\textnormal{err}(\epsilon_{t+1})$ will follow from \cref{eq:comp_linearConvergence_singleIteration_result_e}.
It is thus sufficient to prove \cref{eq:comp_linearConvergence_singleIteration_result_e}.

To use \cref{eq:comp_common_singleIteration_result_e} of \cref{lem:comp_common_singleIteration}, we need to find some $Z^* \in \mathcal B^* \cap \mathcal B_\mu \cap \mathcal C^{(t)}$.
Assumption \cref{eq:comp_linearConvergence_singleIteration_assumption_B} with large enough $c_e$ implies that \cref{lem:delta_bounds} holds w.r.t.~$\begin{psmallmatrix} U \\ V \end{psmallmatrix} \to Z_t$ and $\epsilon = 1/4$. 
Let $Z^* = \begin{psmallmatrix} U^* \\ V^* \end{psmallmatrix} \in \mathcal B^* \cap \mathcal B_\mu$ be the corresponding matrix given by \cref{lem:delta_bounds}.
By \cref{eq:deterministic_nearby_delta_bound}, which is established during the proof of \cref{lem:delta_bounds}, $Z^*$ satisfies
\begin{align*}
\|U^* - U_t\|_F^2 &+ \|V^*-V_t\|_F^2 \leq \frac{25}{4\sigma_r^*} e^2_t.
\end{align*}
Combining this with the RIP lower bound of the current estimate \cref{eq:comp_common_Xt_RIP} yields that $Z^* \in \mathcal B^* \cap \mathcal B_\mu \cap \mathcal C^{(t)}$.
\cref{lem:comp_common_singleIteration} thus guarantees that $Z^*$ satisfies \cref{eq:comp_common_singleIteration_result_e}.
In the following, we shall prove that the LHS of \cref{eq:comp_common_singleIteration_result_e} is lower bounded by $\tfrac{9}{10}{\sqrt p} e_{t+1}$, and that its RHS is upper bounded by $ \tfrac{9}{20}{\sqrt p} e_t$. Together, these bounds yield the required \cref{eq:comp_linearConvergence_singleIteration_result_e}.

Let us begin with the RHS of \cref{eq:comp_common_singleIteration_result_e}.
First, by \cref{eq:deltaProduct_FOmega_bound} of \cref{lem:delta_bounds} we have
\begin{align}
\|(U^*-U_t)(V^*-V_t)^\top\|_{F(\Omega)} &\leq \frac{\sqrt p}{24} e_t. \label{eq:Z*_Zt_et_bound}
\end{align}
Second, we bound $\|\Delta U_t \Delta V_t^\top\|_{F(\Omega)}$ where $\Delta U_t = U_{t+1} - U_t$ and $\Delta V_t = V_{t+1} - V_t$.
The assumption $Z_t \in \mathcal B_\mu$ \cref{eq:comp_linearConvergence_singleIteration_assumption_B} implies $\|U_t\|_{2,\infty} \leq \sqrt{3\mu r\sigma_1^*/n_1}$ and $\|V_t\|_{2,\infty} \leq \sqrt{3\mu r\sigma_1^*/n_2}$.
Similarly, $Z_{t+1} \in \mathcal B_\mu$ implies $\|U_{t+1}\|_{2,\infty} \leq \sqrt{3\mu r\sigma_1^*/n_1}$ and $\|V_{t+1}\|_{2,\infty} \leq \sqrt{3\mu r\sigma_1^*/n_2}$.
Hence,
\begin{align*}
\|\Delta U_t\|_{2,\infty} \leq 2\sqrt{3\mu r\sigma_1^*/n_1},\quad
\|\Delta V_t\|_{2,\infty} \leq 2\sqrt{3\mu r\sigma_1^*/n_2}.
\end{align*}
Invoking \cref{lem:KMO10_lemma71_consequence} with $t\to \sigma_1^*$ and $\epsilon \to 1/(8\sqrt{\kappa})$ thus yields
\begin{align*}
\frac 1p \|\Delta U_t \Delta V_t^\top\|^2_{F(\Omega)} &\leq \frac{\|\Delta U_t\|_F^2 + \|\Delta V_t\|_F^2}{2} \left[c\left(\|\Delta U_t\|_F^2 + \|\Delta V_t\|_F^2\right) + \frac{\sigma_r^*}{64} \right] .
\end{align*}
Together with \cref{eq:comp_common_singleItertion_result_Delta} of \cref{lem:comp_common_singleIteration} we obtain
\begin{align*}
\|\Delta U_t \Delta V_t^\top\|^2_{F(\Omega)} &\leq \frac{9}{2} \left(\frac{9c}{2^{2t} c_e^2} + \frac{1}{32}\right) pe_{t}^2 \leq \frac{p}{6}e^2_t,
\end{align*}
where the last inequality follows for large enough $c_e$.
This, together with \cref{eq:Z*_Zt_et_bound}, shows that the RHS of \cref{eq:comp_common_singleIteration_result_e} is upper bounded by $\tfrac{9}{20}\sqrt p e_t$.

Finally, we prove that the LHS of \cref{eq:comp_common_singleIteration_result_e} is lower bounded by $\tfrac{9}{10}\sqrt p e_t$.
In light of \cref{eq:comp_common_singleIteration_result_B}, $Z_{t+1}$ satisfies the conditions of \cref{thm:RIP} with $\epsilon = 1/10$ for large enough $c_e$. The required lower bound thus follows by \cref{thm:RIP}.
\end{proof}
\begin{proof}[Proof of \cref{lem:comp_common_singleIteration}]
Let us begin by proving \cref{eq:comp_common_singleItertion_result_Delta}.
Denote $\Delta U_t = U_{t+1} - U_t$ and $\Delta V_t = V_{t+1} - V_t$.
By $Z_{t+1} \in \mathcal C^{(t)}$,
\begin{align}\label{eq:Deltat_temp_bound}
\|\Delta U_t\|_F^2 + \|\Delta V_t\|_F^2 \leq \frac{8}{p\sigma_r^*} e^2_{\Omega,t} .
\end{align}
The term $e^2_{\Omega,t}$ can be bounded by combining the RIP assumption \cref{eq:comp_common_Xt_RIP} with $Z_t \in \mathcal B_\textnormal{err}(\epsilon_t)$ \cref{eq:comp_linearConvergence_singleIteration_assumption_B}, as
\begin{align}\label{eq:comp_common_singleIteration_eBound} 
\frac{1}{\sqrt p} e_{\Omega,t} \leq \frac{3}{\sqrt 8} e_t \leq \frac{3\sigma_r^*}{2^t c_e \sqrt{8\kappa}}  .
\end{align}
Plugging these bounds back into \cref{eq:Deltat_temp_bound} yields \cref{eq:comp_common_singleItertion_result_Delta}.

Next, we prove \cref{eq:comp_common_singleIteration_result_B} using \cref{lem:update_balance_distance_bounds}.
Let $a$ be defined as in \cref{lem:update_balance_distance_bounds},
\begin{align*}
a &= \left( \sqrt{2} \max\{\sigma_1(U_t), \sigma_1(V_t)\} + \tfrac 12 d_P(Z_t, Z_{t+1}) \right) d_P(Z_t, Z_{t+1}) \\
&\leq \left( \sqrt{2} \max\{\sigma_1(U_t), \sigma_1(V_t)\} + \tfrac 12 \sqrt{\|\Delta U_t\|_F^2 + \|\Delta V_t\|_F^2}\right) \sqrt{\|\Delta U_t\|_F^2 + \|\Delta V_t\|_F^2} .
\end{align*}
Combining the assumption $Z_t \in \mathcal B_\textnormal{lsv}(\nu_t)$ \cref{eq:comp_linearConvergence_singleIteration_assumption_B}, the definition of $\nu_t$ \cref{eq:epsilon_delta_xi_t} and \cref{eq:comp_common_singleItertion_result_Delta} yields
\begin{align*}
a &\leq \left(\nu_t \sqrt{2\sigma_1^*} + \frac{3 \sqrt{\sigma_r^*}}{2^{t+1} c_e \sqrt\kappa} \right) \frac{3\sqrt{\sigma_r^*}}{2^t c_e \sqrt\kappa} \leq \left(2\sqrt 2 + \frac{6\sqrt 2}{c_e}  + \frac{3}{c_e \kappa} \right) \frac{3\sigma_r^*}{2^t c_e } \leq \frac{9\sigma_r^*}{2^t c_e},
\end{align*}
where the last inequality follows for large enough $c_e$.
Employing \cref{eq:update_balance_bound} of \cref{lem:update_balance_distance_bounds} and the assumption $Z_t \in \mathcal B_\textnormal{bln}(\delta_t)$ \cref{eq:comp_linearConvergence_singleIteration_assumption_B} with $\delta_t$ given in \cref{eq:epsilon_delta_xi_t} yields
\begin{align*}
\|U_{t+1}^\top U_{t+1} - V_{t+1}^\top V_{t+1}\|_F &\leq \|U_t^\top U_t - V_t^\top V_t\|_F + 2a
\leq \frac{\sigma_r^*}{c_l} + \frac{36(1-2^{-t})\sigma_r^*}{c_e} + \frac{18\sigma_r^*}{2^{t} c_e} = \frac{\delta_{t+1} \sigma_r^*}{c_l}. 
\end{align*}
Also, by the triangle inequality, the assumption  $Z_t \in \mathcal B_\textnormal{err}(\epsilon_t)$ \cref{eq:comp_linearConvergence_singleIteration_assumption_B} and \cref{eq:update_distance_bound} of \cref{lem:update_balance_distance_bounds},
\begin{align*}
e_{t+1} \leq e_t + \|U_{t+1}V_{t+1}^\top &- U_tV_t^\top\|_F \leq \frac{\sigma_r^*}{2^t c_e } + a \leq \frac{10\sigma_r^*}{2^t c_e} . 
\end{align*}
In other words, $Z_{t+1} \in \mathcal B_\textnormal{err}({10}/{(2^t c_e)}) \cap \mathcal B_\textnormal{bln}(\delta_{t+1})$.

Next, by the definition of \cref{alg:reg_GNMR}, $Z_{t+1} \in \mathcal B_\mu$.
Hence, to complete the proof of \cref{eq:comp_common_singleIteration_result_B}, we need to show that $Z_{t+1} \in \mathcal B_\textnormal{lsv}(\nu_{t+1})$.
Observe that, by the assumption $Z_{t} \in \mathcal B_\textnormal{lsv}(\nu_{t})$ \cref{eq:comp_linearConvergence_singleIteration_assumption_B}, we have $\sigma_1(U_t) \leq [1 + {6(1 - 2^{-t})}/{c_e}]\sqrt{\sigma_1^*}$.
In addition, combining $Z_{t+1} \in \mathcal C^{(t)}$ with the bound \cref{eq:comp_common_singleIteration_eBound} on $e_{\Omega,t}$ yields $\|U_{t+1} - U_t\|_F \leq \sqrt\frac{8}{p \sigma_r^*} e_{\Omega,t} \leq \frac{3\sqrt{\sigma_1^*}}{2^t c_e}$. Hence,
\begin{align*}
\sigma_1(U_{t+1}) &\leq \sigma_1(U_t) + \|U_{t+1} - U_t\|_2 \leq \left(2 + \frac{6\left(1 - 2^{-t}\right)}{c_e} + \frac{3}{2^t c_e}\right) \sqrt{\sigma_1^*} 
= \nu_{t+1} \sqrt{\sigma_1^*} ,
\end{align*}
and similarly $\sigma_1(V_{t+1}) \leq \nu_{t+1} \sqrt{\sigma_1^*}$. This completes the proof of \cref{eq:comp_common_singleIteration_result_B}.

Finally, to prove \cref{eq:comp_common_singleIteration_result_e}, assume $\mathcal B^* \cap \mathcal B_\mu \cap \mathcal C^{(t)} \neq \emptyset$, and let $Z^* = \begin{psmallmatrix} U^* \\ V^* \end{psmallmatrix} \in \mathcal B^* \cap \mathcal B_\mu \cap \mathcal C^{(t)}$.
Let $F_t^2(Z) = \|\mathcal L_A^{(t)}( Z ) - b_t\|^2$ be the objective of the least squares problem in \cref{alg:reg_GNMR}.
Since $Z^* \in \mathcal B^*$, namely $U^*V^{*\top} = X^*$, we have
\begin{align*}
F_t(Z^*) &= \|U_t V^{*\top} + U^* V_t^\top - U_t V_t - U^* V^{*\top}\|_{F(\Omega)}
= \|(U^* - U_t)(V^* - V_t)^\top\|_{F(\Omega)} .
\end{align*}
In addition, the objective at the new iterate is lower bounded as
\begin{align*}
F_t(Z_{t+1}) &= \|U_t V_{t+1}^{\top} + U_{t+1} V_t^\top - U_t V_t - X^*\|_{F(\Omega)} \nonumber \\
&= \|U_{t+1} V_{t+1}^\top - \left(U_{t+1} - U_{t}\right)\left(V_{t+1} - V_t\right)^\top - X^* \|_{F(\Omega)} \nonumber\\
&\geq e_{\Omega,t+1} - \|\left(U_{t+1} - U_t\right)\left(V_{t+1} - V_t\right)^\top\|_{F(\Omega)}.
\end{align*}
Since $Z^* \in \mathcal B_\mu \cap \mathcal C^{(t)}$, it is a feasible point of the least squares problem in \cref{alg:reg_GNMR}.
As $Z_{t+1}$ is the minimizer of the objective, we conclude $F_t(Z_{t+1}) \leq F_t(Z^*)$.
This completes the proof of \cref{eq:comp_common_singleIteration_result_e}.
\end{proof}

\subsection*{Feasibility of the constraints}
Our guarantees for the matrix completion setting hold for a constrained version of \GNMR, \cref{alg:reg_GNMR}. In the following claim, we show that starting from the initialization described in \cref{alg:comp_initialization} (see also \cref{lem:comp_initialization}), then w.h.p.~the constraints are feasible at all iterations. Note we do not directly use this claim in our proofs.
\begin{claim}\label{claim:feasible_constraints}
Starting from an initialization $Z_0 \in \mathcal B(0)$ where $\mathcal B(t)$ is defined in \cref{eq:comp_linearConvergence_singleIteration_assumption_B}, $\mathcal B_\mu \cap \mathcal C^{(t)} \neq \emptyset$ for all $t$ w.p.~at least $1 - {3}/{n^3}$.
\end{claim}
\begin{proof}
Suppose that $Z_t \in \mathcal B(t)$. We show that if a certain random event occurs, then this implies $\mathcal B_\mu \cap \mathcal C^{(t)} \neq \emptyset$ and $Z_{t+1} \in \mathcal B(t+1)$.
Furthermore, since this event does not depend on $t$ and occurs w.p.~at least $1 - {3}/{n^3}$, the remark follows by induction.

Specifically, assume that the random event of \cref{thm:RIP} occurs.
By \cref{lem:deterministic_nearby_delta_bound}, the assumption $Z_t \in \mathcal B(t)$ guarantees factor matrices $Z^* \in \mathcal B_\mu$ that satisfies \cref{eq:deterministic_nearby_delta_bound} w.r.t.~$\begin{psmallmatrix} U \\ V \end{psmallmatrix} \to Z_t$. In addition, by \cref{thm:RIP} with $\epsilon = 7/8$ we have that $\|X_t - X^*\|_F \leq \frac{8}{7\sqrt p} \|X_t - X^*\|_{F(\Omega)}$.
Combining this with \cref{eq:deterministic_nearby_delta_bound} implies that $Z^* \in \mathcal C^{(t)}$.
Putting together, we obtain $\mathcal B_\mu \cap \mathcal C^{(t)} \neq \emptyset$.
Finally, \cref{lem:comp_linearConvergence_singleIteration} guarantees that $Z_{t+1} \in \mathcal B(t+1)$.
\end{proof}

\section{Proof of \cref{rem:comp_initialization}\label{sec:comp_initialization_proof} (matrix completion, initialization)}
Several previous works used a spectral-based initialization accompanied by some normalization procedure on the rows of the factor matrices \cite{keshavan2010matrix,jain2013low,sun2016guaranteed,yi2016fast, zheng2016convergence}.
In this work, we use the same initialization as in \cite{sun2016guaranteed,yi2016fast,zheng2016convergence}, and clip the rows with large $\ell_2$-norm. The full procedure is described in \cref{alg:comp_initialization}.
\cite{yi2016fast,zheng2016convergence} proved that the resulting initialization $Z_0$ is in $\mathcal B_\textnormal{err}({1}/{c_e}) \cap \mathcal B_\mu$.
In the following analysis we show that $Z_0$ is also approximately balanced, as required by our \cref{thm:comp_linearConvergence}.
Note that \cref{alg:comp_initialization} is given as input the parameter $\mu$.
If this quantity is unknown, it can be estimated from the observed data as discussed in \cref{rem:comp_parametersInput}.

\begin{algorithm}[t]
\caption{Initialization procedure for matrix completion}
\label{alg:comp_initialization}
\SetKwInOut{Return}{return}
\SetKwInOut{Input}{input}
\SetKwInOut{Output}{output}
\Input{$X \in \mathbb R^{n_1\times n_2}$ - observed matrix ($X_{ij} = X^*_{ij} \,\, \forall (i,j)\in \Omega$ and $X_{ij} = 0$ elsewhere) \\
$r, \mu$ - rank and incoherence parameter of $X^*$}
\Output{$Z_0$ - initialization in $\mathbb R^{(n_1+n_2)\times r}$} 
set $Z = \text{b-SVD}(X)$ \\
set $Z_0^{(i)} = Z^{(i)} / \max\{1, \sqrt\frac{n}{2\mu r}\cdot \frac{\|Z^{(i)}\|_F}{\|Z\|_2} \} \text{ for all rows } i=1,\ldots,(n_1+n_2)$ \\
\Return{$Z_0$}
\end{algorithm}

\begin{lemma}\label{lem:comp_initialization}
Let $c_e, c_l > 0$.
There exist constants $c_1, c_2$ and a constant $C = C(c_e, c_l)$ such that the following holds.
Let $X^* \in \mathcal M(n_1, n_2, r, \mu, \kappa)$.
Assume $\Omega \subseteq [n_1]\times [n_2]$ is randomly sampled with $np \geq C \mu r^2 \kappa^4 \log n$.
Then w.p.~at least $1 - c_1 n^{-c_2}$, the output of \cref{alg:comp_initialization} is in $\mathcal B_\textnormal{err}({1}/{(c_e\sqrt\kappa)}) \cap \mathcal B_\textnormal{bln}({1}/{c_l}) \cap \mathcal B_\mu$.
\end{lemma}

For the proof we need the following auxiliary lemma, a variant of \cite[Lemma~1]{zheng2016convergence}, which provides bounds on the Procrustes distance between $\text{b-SVD}(X^*)$ and the matrices $Z, Z_0$ of \cref{alg:comp_initialization}.

\begin{lemma}\label{lem:ZL16_lemma1}
Let $c > 0$. There exist constants $c_1, c_2$ and a constant $C = C(c)$ such that the following holds.
Let $X^* \in \mathcal M(n_1, n_2, r, \mu, \kappa)$ and $Z^* = \text{b-SVD}(X^*)$.
In addition, let $Z = \text{b-SVD}(X)$ be the output of the first step of \cref{alg:comp_initialization}, and $Z_0$ be its final output.
Assume $\Omega \subseteq [n_1]\times [n_2]$ is randomly sampled with $np \geq C \mu r^2 \kappa^4 \log n$.
Then w.p.~at least $1 - c_1 n^{-c_2}$,
\begin{align}\label{eq:ZL16_lemma1}
d_P\left(Z_0, Z^*\right) \leq d_P(Z, Z^*) \leq \frac{\sqrt{\sigma_r^*}}{c \kappa} .
\end{align}
\end{lemma}

\cref{lem:ZL16_lemma1} is similar to \cite[Lemma~1]{zheng2016convergence}, but differs from it in two aspects. First, \cite[Lemma~1]{zheng2016convergence} is stated only for $c=4$, and correspondingly, $C$ is a constant.
Second, instead of \cref{eq:ZL16_lemma1}, they only guarantee a looser bound $d_P(Z, Z^*) \leq \sqrt{\sigma_r^*}/4$. For this bound, however, they require a smaller set $|\Omega|$ by a factor of $\kappa^2$, namely $np \geq C\mu r^2 \kappa^2 \log n$.
It is easy to check that a minor modification in their proof makes it valid for our variant, see \cite[Eq.~(48)]{zheng2016convergence}.

\begin{proof}[Proof of \cref{lem:comp_initialization}]
Let $Z_0 = \begin{psmallmatrix} U_0 \\ V_0 \end{psmallmatrix}$  be the output of \cref{alg:comp_initialization}.
Since $Z^* = \text{b-SVD}(X^*)$ it is perfectly balanced.
\cref{lem:update_balance_distance_bounds} thus implies
\begin{align}\label{eq:init_balance_distance_temp_bounds}
\|U_0^\top U_0 - V_0^\top V_0\|_F &\leq 2a, \quad
\|U_0V_0^\top - X^*\|_F \leq a,
\end{align}
where
\begin{align*}
a = \left( \sqrt{2 \sigma_1^*} + \tfrac 12 d_P(Z_0, Z^*) \right) d_P(Z_0, Z^*).
\end{align*}
Invoking \cref{lem:ZL16_lemma1} with $c = \max\{1, 2c_e, 4c_l\}$ yields
\begin{align*}
a \leq \left( \sqrt{2 \sigma_1^*} + \frac 12 \sqrt{\sigma_r^*} \right) \frac{\sqrt{\sigma_r^*}}{\kappa\cdot \max\{2c_e, 4c_l\}} 
\leq \frac{\sigma_r^*}{\sqrt \kappa \cdot \max\{c_e, 2c_l\}} .
\end{align*}
Inserting this into \cref{eq:init_balance_distance_temp_bounds} gives $Z_0 \in B_\textnormal{err}({1}/{(c_e\sqrt\kappa)}) \cap \mathcal B_\textnormal{bln}({1}/{c_l})$.

Next, we show that $Z_0 \in \mathcal B_\mu$.
The second step in \cref{alg:comp_initialization} guarantees $\|U_0^{(i)}\| \leq \sqrt{2\mu r/n} \|U\|_2$. By \cref{prop:sigmaMin_U}, $\|U\|_2 - \sqrt{\sigma_1^*} \leq d_P(Z, Z^*)$.
Invoking \cref{lem:ZL16_lemma1} with $c = 5$ thus yields
\begin{align*}
\|U_0^{(i)}\| \leq \sqrt\frac{2\mu r}{n} \left[\sqrt{\sigma_1^*} + d_P(Z, Z^*)\right] \leq \sqrt\frac{2\mu r}{n} \cdot \frac 65 \sqrt{\sigma_1^*} \leq \sqrt\frac{3\mu r \sigma_1^*}{n} ,
\end{align*}
and similarly $\|V_0^{(i)}\| \leq \sqrt{3\mu r \sigma_1^*/n}$. This completes the proof.
\end{proof}

\section{Proof of \cref{thm:comp_quadConvergence}\label{sec:comp_quadConvergence_proof} (matrix completion, quadratic convergence)}
Recall the definition \cref{eq:B_s_def} of $\mathcal B_\textnormal{lsv}(\nu)$.
Denote the current and next iterates of \cref{alg:reg_GNMR} by $Z_t = \begin{psmallmatrix} U_t \\ V_t \end{psmallmatrix}$ and $Z_{t+1} = \begin{psmallmatrix} U_{t+1} \\ V_{t+1} \end{psmallmatrix}$, respectively.
Let $X_t = U_tV_t^\top$ and $X_{t+1} = U_{t+1}V_{t+1}^\top$ be the corresponding estimates.
The following lemma is analogous to \cref{lem:comp_linearConvergence_singleIteration}, but here the error contracts with a quadratic rate.

\begin{lemma}\label{lem:comp_quadConvergence_singleIteration}
There exist constants $C, c_e, c_l$ such that the following holds.
Let $X^* \in \mathcal M(n_1, n_2, r, \mu, \kappa)$. Denote $\gamma = {c_e}/{(2\sigma_r^* \sqrt p)}$ and
\begin{align*}
\tilde{\mathcal B}(t) = \mathcal B_\textnormal{err}(\sqrt p \epsilon_t) \cap \mathcal B_\textnormal{err}(\delta_t) \cap \mathcal B_\mu \cap \mathcal B_\textnormal{lsv}(\nu_t) ,
\end{align*}
where $\epsilon_t, \delta_t$ and $\nu_t$ are as in \cref{eq:epsilon_delta_xi_t}.
Assume $\Omega \subseteq [n_1]\times [n_2]$ is randomly sampled with $np \geq C \mu r \log n$.
Further assume that at some iteration $t$,
\begin{align}
Z_t \in \tilde{\mathcal B}(t). \label{eq:comp_quadConvergence_singleIteration_assumption_B}
\end{align}
Then w.p.~at least $1 - {3}/{n^3}$, for all iterates $t' \geq t$,
\begin{subequations}\label{eq:comp_quadConvergence_singleIteration_results}\begin{align}
Z_{t'+1} &\in \tilde{\mathcal B}(t'+1), \label{eq:comp_quadConvergence_singleIteration_result_B} \\
\|X_{t'+1} - X^*\|_F &\leq \gamma \|X_{t'} - X^*\|_F^2 . \label{eq:comp_quadConvergence_singleIteration_result_e}
\end{align}\end{subequations}
\end{lemma}

\begin{proof}[Proof of \cref{thm:comp_quadConvergence}]
Let $Z_0 = \begin{psmallmatrix} U_0 \\ V_0 \end{psmallmatrix}$ be an initial guess which satisfies the conditions of the theorem.
Let us show that it satisfies assumptions \cref{eq:comp_quadConvergence_singleIteration_assumption_B} of \cref{lem:comp_quadConvergence_singleIteration} at $t=0$.
Since $Z_0 \in \mathcal B_\textnormal{err}({\sqrt p}/{(c_e\sqrt\kappa)}) \cap \mathcal B_\textnormal{bln}({1}/{c_l}) \cap \mathcal B_\mu$, we only need to show that $Z_0 \in \mathcal B_\textnormal{lsv}(\nu_0)$ with $\nu_0 = 2$.
Since $p \leq 1$, we have $\mathcal B_\textnormal{err}({\sqrt p}/{(c_e\sqrt\kappa)}) \subseteq \mathcal B_\textnormal{err}({1}/{(c_e\sqrt\kappa)})$, and the proof of $Z_0 \in \mathcal B_\textnormal{lsv}(\nu_0)$ is as in the proof of \cref{thm:comp_linearConvergence}.
The theorem then follows by applying \cref{lem:comp_quadConvergence_singleIteration} at $t=0$.
The fact that $\gamma \|X_0 - X^*\|_F \leq {1}/{(2\sqrt\kappa)}$ follows by the assumption $Z_0 \in \mathcal B_\textnormal{err}({\sqrt p}/{(c_e\sqrt\kappa)})$.
\end{proof}

\subsection*{Proof of \cref{lem:comp_quadConvergence_singleIteration}}
To prove the lemma, we derive a different RIP from the one of \cref{thm:RIP}.
This RIP applies to a much smaller neighborhood of $X^*$, $\|X-X^*\| \lesssim \sigma_r^* \sqrt p$; on the other hand, it holds with fewer number of observations $|\Omega|$, and does not require bounded row norms or balanced factor matrices.
\begin{lemma}\label{lem:quad_RIP}
There exist constants $C, c_e$ such that the following holds.
Let $X^* \in \mathcal M(n_1, n_2, r, \mu, \kappa)$.
Let $\epsilon \in (0,1)$, and assume $\Omega \subseteq [n_1]\times [n_2]$ is randomly sampled with $np \geq \frac{C}{\epsilon^2} \mu r \log n$.
Then w.p.~at least $1 - {3}/{n^3}$, the RIP \cref{eq:RIP} holds for any $X \in \mathbb R^{n_1\times n_2}$ that satisfies
\begin{align}\label{eq:quad_RIP_assumption}
\|X - X^*\|_F &\leq \frac{\epsilon \sigma_r^* \sqrt{p}}{c_e} .
\end{align}
\end{lemma}

\begin{proof}[Proof of \cref{lem:comp_quadConvergence_singleIteration}]
In the following, we prove that if a certain random event occurs, then \cref{eq:comp_quadConvergence_singleIteration_results} holds for $t'=t$.
Since this event does not depend on $t$ and occurs w.p.~at least $1 - {3}/{n^3}$, the lemma follows for any $t'\geq t$ by induction.

Let us begin by showing that the conditions of \cref{lem:comp_common_singleIteration} hold.
First, we prove the RIP condition \cref{eq:comp_common_Xt_RIP}.
By assumption \cref{eq:comp_quadConvergence_singleIteration_assumption_B} for large enough $c_e$, $Z_t$ satisfies the assumptions of \cref{lem:quad_RIP} with $\epsilon = 1/8$. \Cref{lem:quad_RIP} thus guarantees \cref{eq:comp_common_Xt_RIP}.
Next, the second condition of \cref{lem:comp_common_singleIteration} is \cref{eq:comp_linearConvergence_singleIteration_assumption_B}.
Assumption \cref{eq:comp_quadConvergence_singleIteration_assumption_B} is in fact a stronger version of \cref{eq:comp_linearConvergence_singleIteration_assumption_B}, with $\sqrt p \epsilon_t$ replacing $\epsilon_t$ in \cref{eq:epsilon_delta_xi_t}.
As a result, rather than \cref{eq:comp_common_singleIteration_result_B}, \cref{lem:comp_common_singleIteration} now guarantees
\begin{align}\label{eq:comp_quadConvergence_singleIteration_tempResult_B}
Z_{t+1} \in \mathcal B_\textnormal{err}\left({10\sqrt p}/{(2^t c_e)}\right) \cap \mathcal B_\textnormal{bln}(\delta_{t+1}) \cap \mathcal B_\mu \cap \mathcal B_\textnormal{lsv}(\nu_{t+1}),
\end{align}
as can be easily verified by tracing its proof.
It thus remains to show that $Z_{t+1} \in \mathcal B_\textnormal{err}(\sqrt p \epsilon_{t+1})$ and that \cref{eq:comp_quadConvergence_singleIteration_result_e} holds.

Assume for the moment that \cref{eq:comp_quadConvergence_singleIteration_result_e} holds.
Together with the assumption $Z_t \in \mathcal B(\sqrt p \epsilon_t)$ \cref{eq:comp_quadConvergence_singleIteration_assumption_B}, this implies
\begin{align*}
\|X_{t+1} - X^*\|_F \leq \frac{c_e}{2\sigma_r^* \sqrt p} \left(\frac{\sqrt p}{2^t c_e \sqrt\kappa}\right)^2 \leq \frac{\sqrt p}{2^{t+1} \sigma_r^* c_e \sqrt \kappa},
\end{align*}
namely $Z_{t+1} \in \mathcal B_\textnormal{err}(\sqrt p \epsilon_{t+1})$.
Hence, it is sufficient to prove \cref{eq:comp_quadConvergence_singleIteration_result_e}.

To use \cref{eq:comp_common_singleIteration_result_e} of \cref{lem:comp_common_singleIteration}, we need to find some $Z^* \in \mathcal B^* \cap \mathcal B_\mu \cap \mathcal C^{(t)}$.
Assumption \cref{eq:comp_quadConvergence_singleIteration_assumption_B} implies that \cref{lem:deterministic_nearby_delta_bound} holds w.r.t.~$\begin{psmallmatrix} U \\ V \end{psmallmatrix} \to Z_t$. 
Let $Z^* = \begin{psmallmatrix} U^* \\ V^* \end{psmallmatrix} \in \mathcal B^* \cap \mathcal B_\mu$ be the corresponding matrix given by \cref{lem:deterministic_nearby_delta_bound}.
In light of \cref{eq:deterministic_nearby_delta_bound}, $Z^*$ satisfies
\begin{align}\label{eq:Delta*_temp_bound}
\|U^* - U_t\|_F^2 &+ \|V^*-V_t\|_F^2 \leq \frac{25}{4\sigma_r^*} e^2_t.
\end{align}
Combining \cref{eq:Delta*_temp_bound} with the RIP lower bound of the current estimate \cref{eq:comp_common_Xt_RIP} yields that $Z_{t+1} \in \mathcal B^* \cap \mathcal B_\mu \cap \mathcal C^{(t)}$.
\Cref{lem:comp_common_singleIteration} thus guarantees that $Z^*$ satisfies \cref{eq:comp_common_singleIteration_result_e}.
In the following, we shall prove that the LHS of \cref{eq:comp_common_singleIteration_result_e} is lower bounded by $\tfrac 23 \sqrt p e_{t+1}$, and that its RHS is upper bounded by $\tfrac 23 \sqrt p \gamma e_t^2$. Together, these bounds yield the required \cref{eq:comp_quadConvergence_singleIteration_result_e}.

Let us begin with the RHS of \cref{eq:comp_common_singleIteration_result_e}.
By the Cauchy-Schwarz inequality and the fact that $ab \leq (a^2+b^2)/2$, the RHS of \cref{eq:comp_common_singleIteration_result_e} reads
\begin{align*}
&\frac{3}{2\sqrt p} \left[\|(U^*-U_t)(V^*-V_t)^\top\|_{F(\Omega)} + \|(U_{t+1} - U_t)(V_{t+1} - V_t)^\top\|_{F(\Omega)} \right] \\
&\leq \frac{3}{4\sqrt p} \left(\|U^*-U_t\|_F^2 + \|V^*-V_t\|_F^2 + \|U_{t+1} - U_t\|_F^2 + \|V_{t+1} - V_t\|_F^2 \right) .
\end{align*}
First, \cref{eq:Delta*_temp_bound} implies that for large enough $c_e$,
\begin{align*}
\|U^* - U_t\|_F^2 + \|V^* - V_t\|_F^2 \leq \frac{c_e}{6\sigma_r^*}e_t^2 = \frac{\gamma \sqrt p}{3} e_t^2 .
\end{align*}
Second, $Z_{t+1} \in \mathcal C^{(t)}$ with the RIP upper bound \cref{eq:comp_common_Xt_RIP} give that for large enough $c_e$,
\begin{align*}
\|U_{t+1} - U_t\|_F^2 + \|V_{t+1} - V_t\|_F^2 \leq \frac{8}{p\sigma_r^*} e_{\Omega,t}^2 \leq \frac{9}{\sigma_r^*} e_t^2 \leq \frac{c_e}{6\sigma_r^*}e_t^2 = \frac{\gamma \sqrt p}{3} e_t^2 .
\end{align*}
Together, these two bounds show that the RHS of \cref{eq:comp_common_singleIteration_result_e} is upper bounded by $\tfrac 23 \gamma \sqrt p e_t^2$.

Finally, we prove that the LHS of \cref{eq:comp_common_singleIteration_result_e} is lower bounded by $\tfrac 23 \sqrt p e_{t+1}$.
In light of \cref{eq:comp_quadConvergence_singleIteration_tempResult_B}, $Z_{t+1}$ satisfies the conditions of \cref{lem:quad_RIP} with $\epsilon = 1/3$ for large enough $c_e$. The required lower bound thus follows by \cref{lem:quad_RIP}.
\end{proof}
\begin{proof}[Proof of \cref{lem:quad_RIP}]
Let $\begin{psmallmatrix} U \\ V \end{psmallmatrix} = \text{b-SVD}(X)$. Then $UV^\top = X$.
In view of \cref{lem:RIP_fromDeltaProductBounds}, it is sufficient to find $\begin{psmallmatrix} U^* \\ V^* \end{psmallmatrix} \in \mathcal B^*$ that satisfies \cref{eq:delta_bounds}.

As $p \leq 1$, by assumption \cref{eq:quad_RIP_assumption} we have $\begin{psmallmatrix} U \\ V \end{psmallmatrix} \in \mathcal B_\textnormal{err}({\epsilon}/{c_e})$, and by \cref{eq:bSVD_is_balanced} of \cref{lem:bSVD_properties} we have $\begin{psmallmatrix} U \\ V \end{psmallmatrix} \in \mathcal B_\textnormal{bln}({1}/{c_l})$ for any $c_l>0$.
Invoking \cref{lem:deterministic_nearby_delta_bound} thus implies the existence of $\begin{psmallmatrix} U^* \\ V^* \end{psmallmatrix} \in \mathcal B^*$ that satisfies \cref{eq:deterministic_nearby_delta_bound}.
We shall now show that $\begin{psmallmatrix} U^* \\ V^* \end{psmallmatrix}$ satisfies \cref{eq:delta_bounds} of \cref{lem:RIP_fromDeltaProductBounds}.

By assumption \cref{eq:quad_RIP_assumption} we have
\begin{align*}
\|UV^\top - X^*\|_F^2 \leq \frac{\epsilon \sigma_r^* \sqrt p}{c_e} \|UV^\top - X^*\|_F.
\end{align*}
Plugging this into \cref{eq:deterministic_nearby_delta_bound} yields
\begin{align*}
\|U-U^*\|^2_F + \|V-V^*\|^2_F \leq \frac{25 \epsilon \sqrt p}{4c_e} \|UV^\top - X^*\|_F,
\end{align*}
from which \cref{eq:delta_F_bound} follows for large enough $c_e$.
In addition, by combining this equation with the Cauchy-Schwarz inequality and the fact that $ab \leq (a^2+b^2)/2$ we obtain
\begin{align*}
\|(U-U^*)(V-V^*)^\top\|_{F(\Omega)} &\leq
\|U-U^*\|_F \|V-V^*\|_F \\
&\leq \frac 12 \left(\|U-U^*\|_F^2 + \|V-V^*\|_F^2\right) \\
&\leq \frac{25 \epsilon \sqrt p}{8c_e} \|X - X^*\|_F ,
\end{align*}
from which \cref{eq:deltaProduct_FOmega_bound} follows for large enough $c_e$.
This completes the proof.
\end{proof}

\section{Proof of \cref{rem:comp_parametersInput}\label{sec:comp_estimating_proof} (matrix completion, estimating $\sigma_r^*$)}
In \cref{rem:comp_parametersInput} we claimed it is possible to estimate $\sigma_r^*$ to high accuracy with high probability.
To prove this claim we shall use the following lemma \cite[Lemma~2]{chen2015incoherence}.
\begin{lemma}\label{lem:che2015_lemma2}
There exists constants $c, c_1, c_2$ such that the following holds.
Let $X^* \in \mathbb R^{n_1\times n_2}$.
Assume $\Omega \subseteq [n_1]\times [n_2]$ is randomly sampled with $|\Omega| = pn_1n_2$, and let $X \in \mathbb R^{n_1\times n_2}$ be such that $X_{ij} = X^*_{ij}$ for any $(i,j)\in \Omega$ and $X_{ij} = 0$ otherwise.
Then w.p.~at least $1 - c_1 n^{-c_2}$,
\begin{align*}
\left\|\frac 1p X - X^*\right\|_2 \leq c \left(\frac{\log n}{p} \|X^*\|_\infty + \sqrt\frac{\log n}{p} \|X^*\|_{\infty,2}\right),
\end{align*}
where $\|A\|_\infty = \max_{ij} |A_{ij}|$ and $\|A\|_{\infty,2} = \max\{\|A\|_{2,\infty}, \|A^\top\|_{2,\infty}\}$.
\end{lemma}

\begin{proof}[Proof of \cref{rem:comp_parametersInput}]
We assume here $n_1 = n_2$ for convenience, but the proof applies to the rectangular case as well.
Let $X$ be the observed matrix as defined in \cref{lem:che2015_lemma2}, and assume $np \geq C \mu r \kappa^2 \log n$.
Combining Weyl's inequality \cref{eq:Weyl} with \cref{lem:che2015_lemma2} gives that for a suitable constant $c$,
\begin{align}
\left| \sigma_r(X/p) - \sigma_r^* \right| \leq \|X/p - X^*\|_2 \leq c \left(\frac{\log n}{p} \|X^*\|_\infty + \sqrt\frac{\log n}{p} \|X^*\|_{\infty,2}\right). \label{eq:sigmar_diff_tempBound}
\end{align}
Let us now bound the quantities $\|X^*\|_\infty$ and $\|X^*\|_{\infty,2}$.
Let $\begin{psmallmatrix} U \\ V \end{psmallmatrix} = \text{b-SVD}(X^*)$.
Since $X^*$ is $\mu$-incoherent, \cref{eq:bSVD_row_norms} of \cref{lem:bSVD_properties} implies
\begin{align*}
\|X^*\|_\infty \leq \|U\|_{2,\infty} \|V\|_{2,\infty} \leq {\mu r \sigma_1^*}/{n}.
\end{align*}
In addition, using $\|AB\|_{2,\infty} \leq \|A\|_{2,\infty} \|B\|_2$, \cref{eq:bSVD_row_norms} and the definition of b-SVD,
\begin{align*}
\|X^*\|_{2,\infty} \leq \|U\|_{2,\infty} \|V\|_2 \leq \sqrt{\mu r/n} \sigma_1^*,
\end{align*}
and similarly $\|X^{*\top}\|_{2,\infty} \leq \sqrt{\mu r/n} \sigma_1^*$. This implies $\|X^*\|_{\infty,2} \leq \sqrt{\mu r/n} \sigma_1^*$.
Plugging these bounds back into \cref{eq:sigmar_diff_tempBound} and using the assumption $np \geq C \mu r \kappa^2 \log n$ yields $| \sigma_r(X/p) - \sigma_r^* | \leq \tfrac{c}{C} \sigma_r^*$. Assuming $C \geq 10c$ thus yields the required result.
\end{proof}

\section{Proof of \cref{thm:stationaryPoints} (stationary points)}\label{sec:stationaryPoints_proof}
For the following lemmas,
let $X^* \in \mathbb R^{n_1\times n_2}$ be a matrix of rank $r$, $\mathcal A\in \mathbb R^{n_1\times n_2}\to \mathbb R^m$ be a linear operator, and $Z = \begin{psmallmatrix} U \\ V \end{psmallmatrix} \in \mathbb R^{(n_1+n_2)r}$ be a pair of factor matrices.
In addition, in this section we use the following definitions for the operators $\mathcal L, \mathcal L_A$, which are similar to \cref{eq:L_LA_operators}: for any $Z' = \begin{psmallmatrix} U' \\ V' \end{psmallmatrix}$,
\begin{subequations}\label{eq:L_LA_operators_modified}
\begin{align*}
\mathcal L^{(Z)}\left( Z' \right) &= UV'^\top + U'V^\top, \\
\mathcal L_A^{(Z)}(Z') &= \mathcal A \mathcal L^{(Z)}\left(Z'\right) = \mathcal A\left(UV'^\top + U'V^\top\right) .
\end{align*}
\end{subequations}

\begin{lemma}\label{lem:stationaryPoints_unbalanced}
Denote $X = UV^\top$, and recall the definition of $\mathcal F$ from \cref{thm:stationaryPoints}.
Then
\begin{align*}
Z \in \mathcal F \quad\text{ if and only if }\quad \mathcal A(X^* - X) \perp \text{range } \mathcal L_A^{(Z)} .
\end{align*}
\end{lemma}

\begin{lemma}\label{lem:stationaryPoints_feasibleSolutions}
Let $\alpha \in \mathbb R$ and $\tilde Z = \frac{1+\alpha}{2} Z$.
Then $Z$ is a stationary point of the updating variant \cref{eq:updatingVariant}, $Z \in \mathcal S_\text{updt-GNMR}$, if and only if $\tilde Z$ is a feasible solution to the least squares problem \cref{eq:generalVariant_LSQR}.
\end{lemma}

\begin{proof}[Proof of \cref{thm:stationaryPoints}]
The stationary points of \texttt{GD} have been studied in multiple works \cite{ge2016matrix,ge2017no,zhu2018global,li2019symmetry}. The equalities $\mathcal S_\text{GD} = \mathcal F$ and $\mathcal S_\text{reg-GD} = \mathcal F \cap \mathcal G$ follow, for example, by the proof of \cite[Theorem~3]{zhu2018global}, see Eq.~(14-18) there.

Next, we prove $\mathcal S_\text{ALS} = \mathcal F$.
A point $Z = \begin{psmallmatrix} U \\ V \end{psmallmatrix} \in \mathbb R^{(n_1+n_2)\times r}$ is a stationary point of \texttt{ALS}, $Z\in \mathcal S_\text{ALS}$, if and only if it satisfies $U = \arg\min_{U'} f(U'V^\top)$ and $V = \arg\min_{V'} f(UV'^\top)$. Equivalently,
\begin{align*}
0 &= \arg\min_{\Delta U} \|\mathcal A\left(U V^\top + \Delta U V^\top - X^*\right)\|^2 = \arg\min_{\Delta U} \|\mathcal A\left(\Delta U V^\top\right) - e\|^2, \\
0 &= \arg\min_{\Delta V} \|\mathcal A\left(UV^\top + U\Delta V^\top - X^*\right)\|^2 = \arg\min_{\Delta V} \|\mathcal A\left(U\Delta V^\top\right) - e\|^2,
\end{align*}
where $e = \mathcal A(X - X^*)$.
The above equalities hold if and only if $e \perp \{\mathcal A(U V'^\top) \,\mid\, U'\in \mathbb R^{n_1\times r}\} \cup \{\mathcal A(U'V^\top) \,\mid\, V'\in \mathbb R^{n_2\times r}\}$, which is equivalent to $Z \in \mathcal F$ according to \cref{lem:stationaryPoints_unbalanced}.

Next, we analyze the stationary points of \GNMR.
We begin with the updating variant \cref{eq:updatingVariant}, $\alpha = -1$. Given the current iterate $Z = \begin{psmallmatrix} U \\ V \end{psmallmatrix}$, in its first step the updating variant calculates the minimal norm solution to
\begin{align}\label{eq:updatingVariant_forStationaryPoints}
\arg\min_{\Delta Z} \|\mathcal L_A^{(Z)}(\Delta Z) - e\|^2 .
\end{align}
In order to complete the proof of \cref{eq:stationaryPoints_GD}, we need to show that $\Delta Z = 0$ is the minimal norm solution to \cref{eq:updatingVariant_forStationaryPoints} if and only if $Z \in \mathcal F$.
Similar to the argument for \texttt{ALS}, $\Delta Z = 0$ is a feasible solution to \cref{eq:updatingVariant_forStationaryPoints} if and only if $e \perp \text{range } \mathcal L_A^{(Z)}$.
Combined with \cref{lem:stationaryPoints_unbalanced} we obtain that $\Delta Z = 0$ is a feasible solution to \cref{eq:updatingVariant_forStationaryPoints} if and only if $Z \in \mathcal F$. But $\Delta Z = 0$ is a feasible solution if and only if it is the minimal norm one, and thus \cref{eq:stationaryPoints_GD} follows.

Next, consider a stationary point $Z = \begin{psmallmatrix} U \\ V \end{psmallmatrix} \in \mathcal S_\text{GNMR}$ of the other variants of \GNMR, $\alpha\neq -1$.
Since all the variants of \GNMR solve the same least squares problem up to a linear transformation of the variables, the set of feasible solutions is independent of the specific variant of \GNMR. Hence $Z \in \mathcal F$ as we proved for the updating variant.
In order to complete the proof of \cref{eq:stationaryPoints_regGD} it thus remains to show that $\alpha\neq -1$ enforces stationary points to be balanced, $Z \in \mathcal G$.

By the first part of \cref{lem:kernel_minNormSol}, the minimal norm solution $\tilde Z = \begin{psmallmatrix} \tilde U \\ \tilde V \end{psmallmatrix}$ to the least squares problem \cref{eq:generalVariant_LSQR} of \GNMR satisfies
\begin{align}\label{eq:minNormSol_stationaryPoints}
\tilde U^\top U = V^\top \tilde V.
\end{align}
In its second step \cref{eq:generalVariant_update}, \GNMR updates $Z_\text{new} = \frac{1-\alpha}{2} Z + \tilde Z$.
In a stationary point, $Z_\text{new} = Z$, or equivalently $\tilde Z = \frac{1 + \alpha}{2} Z $.
Plugging this back into \cref{eq:minNormSol_stationaryPoints} yields $Z\in \mathcal G$ for any $\alpha\neq -1$.
This proves \cref{eq:stationaryPoints_regGD}.

Finally, we specialize our results to the matrix sensing and matrix completion settings.
Let $Z = \begin{psmallmatrix} U \\ V \end{psmallmatrix}$ and assume $Z \in (\mathcal B^* \cap \mathcal G)$.
We need to show that $Z$ is a stationary point of \GNMR, $Z \in \mathcal S_\text{GNMR}$, or equivalently, that $\tilde Z \equiv \frac{1+\alpha}{2} Z$ is the minimal norm solution to the least squares problem \cref{eq:generalVariant_LSQR}.

Let us first show that $\tilde Z$ is a feasible solution to \cref{eq:generalVariant_LSQR}.
Since any global minimum of \cref{eq:matrixRecovery_factorizedObjective} is in particular a local one, we have $\mathcal B^* \subseteq \mathcal F$, so that $Z \in \mathcal F = \mathcal S_\text{updt-GNMR}$. 
Invoking \cref{lem:stationaryPoints_feasibleSolutions} then implies that $\tilde Z$ is a feasible solution to \cref{eq:generalVariant_LSQR}.


In order to prove that $\tilde Z$ is the minimal norm solution, it remains to show that $\tilde Z \perp \ker \mathcal L_A^{(Z)}$.
Since $X^*$ is of rank exactly $r$ and $Z \in \mathcal B^*$, the factor matrices $U,V$ have full column rank, and the second part of \cref{lem:kernel_minNormSol} holds.
By combining \cref{eq:minNormSol_stationaryPoints} and \cref{eq:L_perp} of \cref{lem:kernel_minNormSol}, we have $\tilde Z \perp \ker \mathcal L^{(Z)}$.
In the rest of the proof, we show that
\begin{align}\label{eq:kerLA_in_kerL}
\ker \mathcal L_A^{(Z)} \subseteq \ker \mathcal L^{(Z)}
\end{align}
both in matrix sensing and matrix completion, so that $\tilde Z \perp \ker \mathcal L_A^{(Z)}$ as required.

Let us begin with the matrix sensing case.
Let $Z' \in \ker \mathcal L_A^{(Z)}$. Then, by the $2r$-RIP of $\mathcal A$,
\begin{align*}
\left\|\mathcal L^{(Z)}(Z')\right\|^2 \leq \frac{1}{1-\delta_{2r}} \left\|\mathcal L^{(Z)}_A(Z')\right\|^2 = 0,
\end{align*}
which implies $Z' \in \ker \mathcal L^{(t)}$. This proves \cref{eq:kerLA_in_kerL} in matrix sensing.

Finally, we prove \cref{eq:kerLA_in_kerL} in the matrix completion setting.
By $Z \in \mathcal B^*$ we have $UV^\top = X^*$. Denote by $U^* \Sigma^* V^{*\top}$ the SVD of $X^*$. Then $U = U^* Q$ and $V = V^* Q^{-\top}$ for some invertible $Q \in \mathbb R^{r\times r}$.
Hence for all $Z' = \begin{psmallmatrix} U' \\ V' \end{psmallmatrix} \in \ker \mathcal L_A^{(Z)}$ we have $\mathcal L^{(Z)}( Z' ) = U^* Q V'^\top + U' Q^{-1} V^{*\top}$. By \cref{lem:YPCC16_lemma9}, this implies
\begin{align*}
\left\| \mathcal L^{(Z)} (Z') \right\|^2 \leq \frac 2p \left\| \mathcal L_A^{(Z)} \left(Z'\right) \right\|^2 = 0
\end{align*}
w.p.~at least $1 - {3}/{n^3}$ uniformly for all $Z' \in \ker \mathcal L_A^{(Z)}$.
Hence $Z' \in \ker \mathcal L^{(Z)}$ as required.
\end{proof}

\subsection*{Proofs of \cref{lem:stationaryPoints_unbalanced,lem:stationaryPoints_feasibleSolutions}\nopunct}
\begin{proof}[Proof of \cref{lem:stationaryPoints_unbalanced}]
Since $\|\mathcal A(X - X^*)\|^2 = \braket{ \mathcal A(X - X^*), \mathcal A(X - X^*)} = \braket{ \mathcal A^* \mathcal A(X - X^*),  X - X^*}$, we have
\begin{align*}
\nabla f(X) = \nabla \|\mathcal A(X - X^*)\|^2 = 2\mathcal A^* \left(\mathcal A(X - X^*)\right) = 2\mathcal A^*(e)
\end{align*}
where $e = \mathcal A(X - X^*)$.
Hence $Z \in \mathcal F$ is equivalent to
\begin{align}\label{eq:A*_stationary_conditions}
\mathcal A^*(e)^\top U = 0, \quad
\mathcal A^*(e)V = 0 .
\end{align}
In order to complete the proof, we shall now show that \cref{eq:A*_stationary_conditions} is equivalent to $e \perp \text{range } \mathcal L_A^{(Z)}$.
By construction, $e \perp \text{range } \mathcal L_A^{(Z)}$ is equivalent to $\mathcal A^*(e) \perp \text{range } \mathcal L^{(Z)}$. This, in turn, is equivalent to
\begin{align*}
0 &= \Tr \left[ \mathcal A^*(e)^\top \left(UV'^\top + U'V^\top\right) \right] \\
&= \Tr \left[ \mathcal A^*(e)^\top UV'^\top \right] + \Tr \left[ U'^\top \mathcal A^*(e)V \right], \quad \forall\, U'\in \mathbb R^{n_1\times r}, V'\in \mathbb R^{n_2\times r} ,
\end{align*}
where in the second equality we used the trace property $\Tr[AB] = \Tr[BA] = \Tr[A^\top B^\top]$ for $A \in \mathbb R^{n\times r}$, $B \in \mathbb R^{r\times n}$.
The lemma follows since the last equation is equivalent to \cref{eq:A*_stationary_conditions}.
\end{proof}

\begin{proof}[Proof of \cref{lem:stationaryPoints_feasibleSolutions}]
By construction, $Z$ is a stationary point of the updating variant \cref{eq:updatingVariant} if and only if $\Delta Z = 0$ is the minimal norm solution to the least squares problem \cref{eq:updatingVariant_LSQR}.
This, in turn, holds if and only if $\Delta Z = 0$ is a feasible solution to \cref{eq:updatingVariant_LSQR}.
As discussed in \cref{sec:GNMR_description}, the least squares problems \cref{eq:updatingVariant_LSQR} and \cref{eq:generalVariant_LSQR} are equivalent up to the transformation of variables $\Delta Z = \tilde Z - \frac{1+\alpha}{2} Z$. By this transformation, $\Delta Z = 0$ is a feasible solution to \cref{eq:updatingVariant_LSQR} if and only if $\tilde Z$ is a feasible solution to \cref{eq:generalVariant_LSQR}.
\end{proof}

\section{Additional experimental details}\label{sec:experimental_details}
To simplify notations, let us divide the algorithms into two groups.
The first group consists of methods which employ simple operations at each iteration, such as gradient descent: \texttt{LRGeomCG} and \texttt{ScaledASD}. Each iteration of these methods is in general extremely fast. For these methods we thus allow a relatively large value for the maximal number of iterations, which we denote by $N^{(1)}$. 
The second group contains \texttt{RTRMC}, \texttt{R2RILS}, \texttt{MatrixIRLS} and \GNMR.
These methods are more complicated, in the sense that at each outer iteration they solve an inner optimization sub-problem, which by itself is solved iteratively.
Hence, these methods have two parameters: $N^{(2)}_\text{outer}$ and $N^{(2)}_\text{inner}$ for the maximal number of outer and inner iterations, respectively.
However, since one iteration of \texttt{R2RILS} and \GNMR is significantly slower than that of \texttt{RTRMC} and \texttt{MatrixIRLS}, we give them a smaller value of $N^{(2)}_\text{slow-outer} < N^{(2)}_\text{outer}$ outer iterations.

In addition to maximal number of iterations, we used the following three early stopping criteria: (1) Small observed relative RMSE, $\frac{\|\mathcal P_\Omega (X^* - \hat X_t)\|}{\|\mathcal P_\Omega (X^*)\|} \leq \epsilon_\text{rmse}$; (2) Small relative change, $\frac{\|\hat X_{t+1} - \hat X_t\|_F}{\|\hat X_t\|_F} \leq \epsilon_\text{diff}$; (3) For some integer $t_\text{min-rmse}$, let $x_i = \min\{\frac{\|\mathcal P_\Omega (X^* - \hat X_t)\|}{\|\mathcal P_\Omega (X^*)\|} \,\mid\, t=i\cdot t_\text{min-rmse}, \ldots, (i+1)\cdot t_\text{min-rmse}\}$.
The algorithm stops if the relative RMSE does not change by a factor of $r_\text{min-rmse}$ in each $t_\text{min-rmse}$ iterations, namely if $\frac{x_{i+1}}{x_i} > r_\text{min-rmse}$ for some $i$.
All other stopping criteria defined by the algorithms were disabled.

In the first experiment (\cref{fig:highOversampling}), we set $N^{(1)} = 5000$, $N^{(2)}_\text{inner} = 1500$, $N^{(2)}_\text{outer} = 500$, and $N^{(2)}_\text{outer,GNMR} = 100$.
To allow the algorithms to either converge or fully exploit their maximal number of iterations, we set the thresholds of the first two stopping criteria to $\epsilon_\text{rmse} = \epsilon_\text{diff} = 10^{-16}$, and did not use the third criterion.
In the second experiment (\cref{fig:lowOversampling_median}), we set $N^{(2)}_\text{inner} = 7000$, $N^{(2)}_\text{outer} = 25000$, and $N^{(2)}_\text{outer,GNMR} = 700$. The stopping criteria were set as in the previous experiment.
In the next two experiments (\cref{fig:varyingDim_median,fig:OSR_vs_CN}), we set $N^{(1)} = 10^6$ and $N^{(2)}_\text{inner} = N^{(2)}_\text{outer} = N^{(2)}_\text{outer,GNMR} = 10^5$. The thresholds of the stopping criteria were set to (1) $\epsilon_\text{rmse} = 10^{-14}$; (2) $\epsilon_\text{diff} = 10^{-15}$ for algorithms in the first group and $\epsilon_\text{diff} = 10^{-14}$ for algorithms in the second group; and (3) $r_\text{min-rmse} = 1/2, t_\text{min-rmse} = 200$. In addition, in the third experiment (\cref{fig:varyingDim_median}), for each dimension $n$ and oversampling ratio $\rho$ we ran $150$ attempts to generate a sampling pattern $\Omega$ with $r$ observed entries in each row and column. Each attempt lasted at most $3$ hours. In \cref{fig:varyingDim_median} we presented only the oversampling ratios for which at least $50$ attempts succeeded.
In the next experiment (\cref{fig:setavg_timeVsCN}), we set $N^{(2)}_\text{inner} = 200$, $N^{(2)}_\text{outer} = 300$, and stopping criteria thresholds as in the previous experiment. Since the focus in this experiment was comparing the runtime of \GNMR as function of the condition number rather than measuring the runtime itself, no effort was made to optimize \GNMR's runtime beyond assigning $N^{(2)}_\text{inner}$ with a relatively small value. Note, however, that the qualitative behavior demonstrated in \cref{fig:setavg_timeVsCN} is not sensitive to the value of $N^{(2)}_\text{inner}$, see \cref{fig:setavg_timeVsCN_slow}. The runtimes were measured on a Windows 10 laptop with Intel i7-10510U CPU and 16GB RAM using MATLAB 2020a. 
Finally, the parameter specifications in the last experiment (\cref{fig:setavg_noise_median}) are similar to those of the first experiment (\cref{fig:highOversampling}).

\section{Additional experimental results}\label{sec:experimental_additional_results}
\Cref{fig:lowOversampling_varyingDim_prob} complements \cref{fig:lowOversampling_varyingDim_median} from the main text by showing the recovery probability instead of the median error for the same experiments.
\Cref{fig:setavg_timeVsCN_slow} complements the left panel of \cref{fig:OSR_vs_CN} by showing that qualitatively, the performance of \GNMR is not sensitive to the number of inner least squares iterations.

\begin{figure}[h]
	\centering
	\subfloat{
		\includegraphics[width=0.495\linewidth]{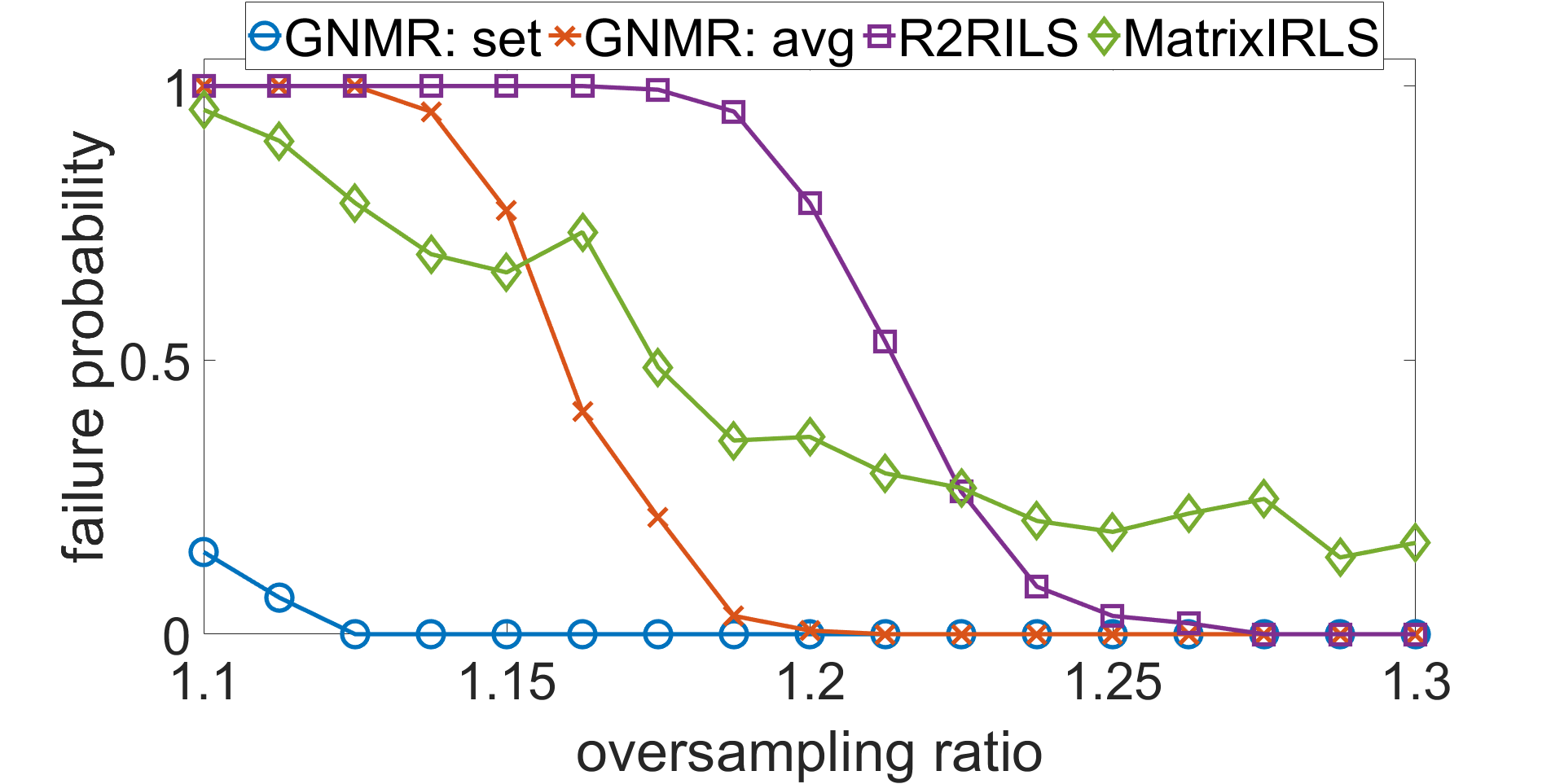}
	}
	\subfloat{
		\includegraphics[width=0.495\linewidth]{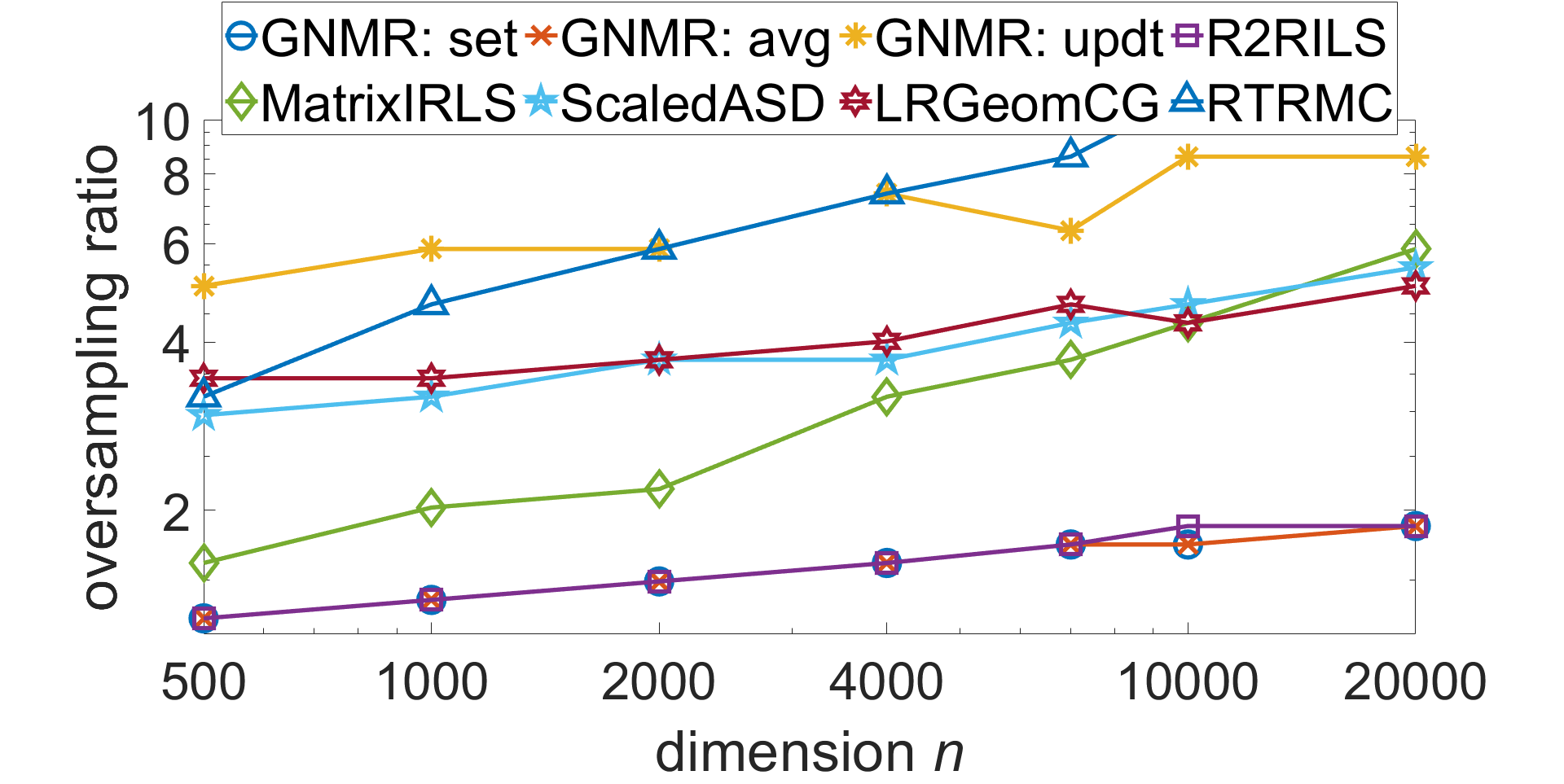}
	}
	\caption{Same as \cref{fig:lowOversampling_varyingDim_median}, but with Y-axes correspond to failure probability, defined as  $\text{Pr}[\texttt{rel-RMSE} >10^{-4}]$ (left panel), and lowest oversampling ratio from which the failure probability is smaller than $0.1$ (right panel).}
	 \label{fig:lowOversampling_varyingDim_prob}
\end{figure}

\begin{figure}[h]
	\centering
	\includegraphics[width=0.495\linewidth]{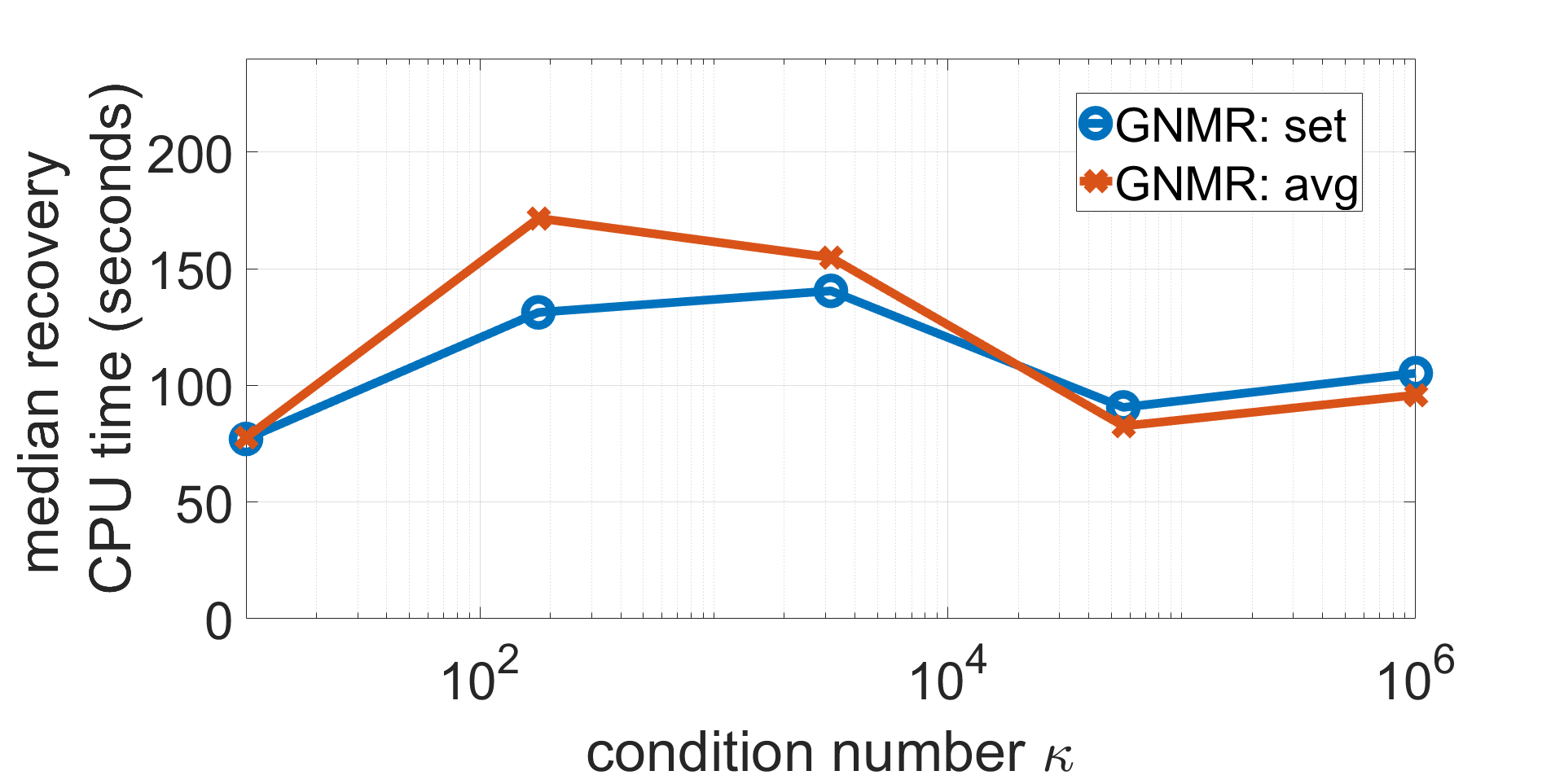}
	\caption{Same as \cref{fig:OSR_vs_CN}, but with $N^{(2)}_\text{inner} = 1500$ instead of $200$.}
	\label{fig:setavg_timeVsCN_slow}
\end{figure}

\end{document}